\documentclass[11 pt]{article}%
\usepackage{amsmath, amsfonts, amsthm, color,latexsym}
\usepackage{amssymb}
\usepackage{hyperref}
\usepackage{mathrsfs}

\usepackage{scrextend}

\usepackage{amsmath}
\usepackage{amsfonts}
\usepackage{amssymb}
\usepackage{color, soul}
\usepackage[all]{xy}
\hypersetup{colorlinks=true,linkcolor=blue} 
\usepackage{graphicx}%
\providecommand{\U}[1]{\protect\rule{.1in}{.1in}}
\allowdisplaybreaks[4]
\newtheorem{theorem}{Theorem}[section]
\newtheorem{proposition}[theorem]{Proposition}

\newtheorem{lemma}[theorem]{Lemma}
\newtheorem{final remark}[theorem]{Final Remark}

\allowdisplaybreaks[4]

\begin{document}

\title{ On a study and applications of the Concentration-compactness type principle for Systems with critical terms in $\mathbb{R}^{N}$}

\author{L.M.M. Bonaldo, E.J. Hurtado, W. Neves 
\thanks{Corresponding author.\newline 2020 Mathematics Subject Classification: 35A15, 35B33, 35B38, 35D30, 35J50, 35J60, 35R11, 46E35, 47G2O.\newline Keywords: Nonlocal integrodifferential  operator,  Fractional Sobolev space with variable exponents, Fractional Systems, Critical  nonlinearities, Variational methods } }
%\thanks{Department of Mathematics Universidade %Federal do Rio de Janeiro, Rio de Janeiro, RJ,  %Brazil.
%E-mail: nn.lauren@hotmail.com. Supported
%by FAPERJ}\, 

%\thanks{Department of Mathematics Universidade %Federal do Rio de Janeiro, Rio de Janeiro, RJ,  %Brazil.
%E-mail: wladimir@im.ufrj.br Supported
%by CAPES} \,

%\thanks{Department of Mathematics, Universidade de %Brasília, Brasília,  DF, Brazil.
%E-mail: elardjh2@gmail.com Supported
%by CAPES}
%\  {2020 Mathematics Subject Classification: %35A15, 35B33, 35B38, 35D30, 35J50, 35J60, 35R11, %46E35, 47G2O}
%\keywords{Nonlocal integrodifferential  operator, % Fractional Sobolev space with variable exponents, %Fractional Systems, Critical  nonlinearities, %Variational methods}

%\date{}
\maketitle

\begin{abstract}
\noindent In this paper,  we obtain some important variants of the Lions and Chabrowski Concentration-compactness principle, in the context of fractional Sobolev spaces with variable exponents, especially for nonlinear systems. As an application of the results, we show the existence and assymptotic behaviour of nontrivial solutions for elliptic systems involving a new class of general nonlocal integrodifferential operators with exponent variables and critical growth conditions in  $\mathbb{R}^{N}$.

%In this paper, we obtain some variants of the Lions and Chabrowski Concentration-compactness principle, which is a key result to overcome the lack of compactness of Sobolev critical embedding, in the context of fractional Sobolev spaces with variable exponents, especially for nonlinear systems. As an application of the results, we show the existence and asymptotic behavior of nontrivial solutions for elliptic systems involving a new class of general nonlocal integrodifferential operators with exponent variables and critical growth conditions in  $\mathbb{R}^{N}$. 

\end{abstract}  
\tableofcontents
  \section{\textbf{Introduction}}
%   \hfill \break 
%\subsection{\textbf{ Motivations}}
In this paper we complement the research begun in the articles \cite{ bonder2,analia,chabro,pucci,kiho,kang,lions1,lions2} key in the study of the Lions and Chabrowski Concentration-compactness principle (the CCP, for short) in the context of fractional Sobolev spaces with variable exponents, which are fundamental, especially in nonlinear systems with critical growth.

In recent decades, there has been a growing interest in the study of problems involving fractional function spaces (also called Slobodeskii spaces, which are a class in the huge family Besov spaces) that play an important role in many issues involving differential equations, in special, the study of nonlocal integrodifferential equations. For example, their applications can be presented as a model for many physical phenomena, finance, image processing,  ecology, and others, see e.g. \cite{apple,bucur,caf1,cont,castillo,huaraoto1,huaraoto2,juan} and references therein. 

%We point out that 
On the other hand, partial differential equations and variational problems involving $p(\cdot)$-variable growth conditions, and hence Sobolev spaces with variable exponent have attracted considerable interest in the latest years. Indeed, their applications can be presented as a model for many physical phenomena that
arise for example in the research of image processing, nonlinear elasticity, electrorheological fluid, etc., we address the reader to 
\cite{dieming,ruzicka} and the references therein.

When one studies elliptic  problems involving nonlinear critical terms, both in bounded and unbounded domains, they are characterized by a lack of compactness, which create serious complications when trying to find critical points, by standard variational methods, for the energy functional associated with the problem. 
%
%In the literature, 
The CCP developed by Lions  \cite{lions3,lions1}(Lemma I.1 in \cite{lions3} and Lemma I.1 in \cite{lions1}, hereafter called CCP1 and CCP2 respectively)  has proven to be an extremely important tool that plays a key role to find solutions for this type of problems. 

However, it is important to remark that, the CCP2 does not provide any information about the mass loss of the sequence at infinity. Thus, to overcome this difficulty the CCP1 is used to show that the sequence is tight and so the so-called vanishing and dichotomy cannot occur, but it is in general very technical and complicated, (see \cite{son}). 
In this sense, to avoid this technical difficulty, Chabrowski \cite{chabro} proved an infinity CCP. In that paper, he formulated the variants of these two principles, that is, he can avoid the  CCP1 by applying an infinity version of the CCP2.

In the fractional context, the first extension of the CCP was studied in bounded domains in \cite{palatucci} for $p=2$, and in \cite{mosconi} the extension was obtained for any $1<p<N/s$. Concerning unbounded domains, the CCP extension has been proven for $p=2$ in \cite{dipierro}, later for any $1 < p < N/s$ in \cite{analia}. Recently for $p( \cdot)$,  \cite{kiho} got an important critical embedding, and thanks to that  established Concentration-compactness principles for fractional Sobolev spaces with variable exponents.

In addition to the extensions of the above mentioned results, variants of the CCP related to critical elliptic systems have also been studied. The pioneering work \cite{kang}, which is based on the ideas of  Lions \cite{lions1,lions2}, established some important variants of the Concentration-compactness principle for systems. These variants allow the study of the existence of nontrivial solutions for elliptic systems, through variational methods.  Indeed, in \cite{kangyu}, applying the result of \cite{kang} combined with variational methods, it was proved the existence of minimizers for the Rayleigh quotient and ground state solutions for a system of semilinear elliptic equations, which involves homogeneous critical nonlinearities and different Hardy-type terms.
 Already in \cite{pucci}, they use the result of \cite{kang} together with the Mountain Pass Theorem to prove the existence of solutions for critical Hardy quasilinear systems. Recently, in the fractional context, \cite{temperini} proved some variants of Lions CCP and Chabrowski's infinite CCP which allowed, through the mountain pass theorem, to show the existence of nontrivial integer solutions for fractional systems $(p, q) $ with Hardy critical terms in $\mathbb{R}^{N}. $

Motivated by the contributions mentioned above, the first objective of this paper is to obtain new variants of the Lions CCP and the Chabrowski infinite CCP in the context of fractional Sobolev spaces with variable exponents. Then, as an application of these new variants obtained (see Theorem \ref{lions} and \ref{lionsinfity}), the second objective is to obtain new results of existence and assymptotic behaviour of nontrivial solutions for elliptic systems. In particular, for the elliptic systems involving a new class of nonlocal integral differential operators with variable exponents and critical growth in $\mathbb{R}^{N}$.

%%%%%%%%%%%%%%%%%%%%%%%%%%%%%
\subsection{Main results}
%%%%%%%%%%%%%%%%%%%%%%%%%%%%%

In order to state the main results of this paper, let us introduce some notations. We consider   $\mathcal{M}(\mathbb{R}^{N})$ the space of all signed finite Radon measures on $\mathbb{R}^{N}$ endowed with the total variation norm, which we may identify with the dual of $C_{0}(\mathbb{R}^{N})$. 
%which is the completion of all continuous functions $u:\mathbb{R}^{N}\to \mathbb{R}$ whose %support is the compact relative to the supremum norm $\|\cdot\|_{\infty}$.

We assume that the variable exponents $p(\cdot, \cdot)$ and  $q(\cdot)$  satisfy the following assumptions:
  \begin{itemize}
  \item[$(\mathit{P})$]  $ p: \mathbb{R}^{2N} \to \mathbb{R}$ is uniformly continuous, $ p$ is symmetric, that is, $ p(x,y)=p(y,x)$,  and  
  \begin{equation*}\label{a234}
  \begin{split}
    1< p^{-} := \displaystyle{\inf_{(x,y) \in \mathbb{R}^{2N}}}\,p(x,y) \leqslant  \displaystyle{ \sup_{(x,y)\in \mathbb{R}^{2N}}}\,p(x,y):= \overline{p}<\frac{N}{s}, s\in (0,1);
 \end{split}
  \end{equation*}
  there exists $\varepsilon_{0}\in \big(0, \frac{1}{2}\big),$ such that $p(x,y)=\overline{p}$ for all $x,y \in \mathbb{R}^{N}$ satisfying $|x-y|< \varepsilon_{0}$, $\displaystyle{\sup_{y\in \mathbb{R}^{N}}}\,p(x,y):= \overline{p}$ for all $x \in \mathbb{R}^{N}$; and $|\{ x \in \mathbb{R}^{N}:p_{\star}(x)\neq\overline{p}\}|<\infty $ where, $p_{\star}(x):=\displaystyle{ \inf_{y \in \mathbb{R}^{N}}p(x,y)}$ for $x \in \mathbb{R}^{N}.$
  \item[$(\mathit{Q})$]$q:\mathbb{R}^{N}\to \mathbb{R}$ is uniformly continuous such that $p_{\star}(x ) \leqslant q(x)\leqslant \overline{p}^{\star}_{s}:=\frac{N\overline{p}^{\star}_{s}}{N-s\overline{p}^{\star}_{s}},$ for all $x \in \mathbb{R}^{N}$ and $\mathscr{C}:=\{x \in \mathbb{R}^{N}; q(x)=\overline{p}^{\star}_{s}\}\neq \emptyset.$

  \end{itemize}
  
  We define  
  \begin{equation}\label{spacec}
C_{+}(\mathbb{R}^{N}):=\bigg\{h\in C(\mathbb{R}^{N}): 1 < \displaystyle{\inf_{x\in \mathbb{R}^{N}}h(x)} \leqslant  \displaystyle{\sup_{x\in \mathbb{R}^{N}}h(x)} < \infty \bigg\}
  \end{equation}
  and consider
  \begin{equation}\label{numero}
      \mathit{W}= W^{s,p(\cdot, \cdot)}(\mathbb{R}^{N})\times W^{s,p(\cdot, \cdot)}(\mathbb{R}^{N})
  \end{equation}
    the Banach space  defined in Subsection \ref{espacofra} with  norm $ \| (u, v) \|:= \| u \|+\| v \|$.
  
   The following result establishes the Lions type Concentration-compactness principle for the space $\mathit{W}$.

  \begin{theorem} \label{lions}
  Assume that $(\mathit{P})$ and $(\mathit{Q})$ hold and $\alpha, \beta \in C_{+}(\mathbb{R}^{N})$   satisfying $\alpha(x)+\beta(x)=q(x)$ for all $x \in \mathbb{R}^{N}$. Let $\{(u_k, v_{k})\}_{k \in \mathbb{N}}$ be a bounded sequence in $\mathit{W}$   such that 
  \begin{equation}\label{c1}
  (u_k, v_{k})\rightharpoonup (u,v) \mbox{ in  } \mathit{W} \mbox{ as  } k \to +\infty,
  \end{equation}
  \begin{equation}\label{c2}
  |u_k|^{\overline{p}}+|v_{k}|^{\overline{p}}+ \int_{\mathbb{R}^{N} }\frac{|u_k(x)-u_k(y)|^{p(x,y)}}{|x-y|^{N+sp(x,y)}} \,dy+\int_{\mathbb{R}^{N} }\frac{|v_k(x)-v_k(y)|^{p(x,y)}}{|x-y|^{N+sp(x,y)}} \,dy \overset{\ast}{\rightharpoonup}  \mu \mbox{ in }  \mathcal{M}(\mathbb{R}^{N}),
  \end{equation}
  \begin{equation}\label{c3}
  |u_k|^{\alpha(x)}|v_k|^{\beta(x)} \overset{\ast}{\rightharpoonup} \nu \mbox{ in }  \mathcal{M}(\mathbb{R}^{N}),
  \end{equation}
  where $\mu$ and $\nu$ are nonnegative finite Radon measure on $\mathbb{R}^{N}.$
  
  Then, there exist two families of numbers $\{\mu_{i}\}_{i \in I}\subset (0, +\infty)$, $\{\nu_{i}\}_{i \in I}\subset (0, +\infty)$  and  a family  $\{x_i\}_{i \in I}\subset\mathscr{C}$ of
distinct points in $\mathbb{R}^{N}$, where $I$ is an at most countable index set, such that
\begin{equation}\label{cc1}
\nu = |u|^{\alpha(x)}|v|^{\beta(x)}+ \sum_{i \in I}\nu_{i}\delta_{x_i},
\end{equation}
\begin{equation}\label{cc2}
\mu \geqslant|u|^{\overline{p}}+|v|^{\overline{p}} +  \int_{\mathbb{R}^{N} }\frac{|u(x)-u(y)|^{p(x,y)}}{|x-y|^{N+sp(x,y)}} \,dy+\int_{\mathbb{R}^{N} }\frac{|v(x)-v(y)|^{p(x,y)}}{|x-y|^{N+sp(x,y)}} \,dy  + \sum_{i \in I}\mu_{i}\delta_{x_{i}},
\end{equation}
\begin{equation}\label{cc3}
\mathcal{S}_{\alpha\beta}\nu_{i}^{\frac{1}{\overline{p}^{\star}_{s}}}\leqslant \mu_{i}^{\frac{1}{\overline{p}}} 
\end{equation}
$\mbox{ for all } i \in I$, where
\begin{equation}\label{SS}
\mathcal{S}_{\alpha\beta}:= \inf_{\substack{(u,v)\in \mathit{W},\\ u\neq 0, v\ne 0}}\frac{\|(u,v)\|}{\bigg\|u^{\frac{\alpha(\cdot)}{q(\cdot)}}v^{\frac{\beta(\cdot)}{q(\cdot)}}\bigg\|}_{L^{q(\cdot)}(\mathbb{R}^{N})}.\, \, 
\end{equation} 

  \end{theorem}
  
  One remarks that, the Theorem  \ref{lions} extends in several directions previous results  for the variable exponent   context and  the fractional vectorial context the results of \cite{bonder2,analia,kiho, kang,lions1,lions2,temperini}. 
  Moreover, the strategy to show this theorem is the same as that papers P.L. Lions \cite{lions1,lions2}, but there exist several difficulties because we are working in a nonlocal context, with variable exponents involved in $\mathbb{R}^{N}$. In order to overcome these nontrivial technical difficulties, we apply Theorem \ref{critico} which gives us the sufficient conditions for the variable exponents, such as the log-H\"older continuity condition. Thus we obtain a critical embedding of fractional Sobolev spaces with variable exponents in Lebesgue spaces with variable exponents.  As well the technical lemmas, Lemma \ref{4.4a} and \ref{lieb}.
  
  The following result establish  the Chabrowsky type Concentration-compactness infinity principle for the space $\mathit{W}$.
  
  \begin{theorem} \label{lionsinfity}
  Assume that $(\mathit{P})$ and $(\mathit{Q})$ hold and $\alpha, \beta \in C_{+}(\mathbb{R}^{N})$   satisfying $\alpha(x)+\beta(x)=q(x)$ for all $x \in \mathbb{R}^{N}$. Let $\{(u_k, v_{k})\}_{k \in \mathbb{N}}$ be a bounded sequence in $\mathit{W}$   such that 
  \begin{equation*}\label{c1}
  (u_k, v_{k})\rightharpoonup (u,v) \mbox{ in  } \mathit{W} \mbox{ as  } k \to +\infty,
  \end{equation*}
  and define 
  \begin{equation}\label{ccinfi}
\nu_{\infty}:= \displaystyle{\lim_{R \to \infty}\limsup_{k \to \infty}\int_{B^{c}_{R}}|u_k|^{\alpha(x)}|v_k|^{\beta(x)}\,dx},
\end{equation}
\begin{equation}\label{cc2infi}
\begin{split}
\mu_{\infty}:= \lim_{R \to \infty}\limsup_{k \to \infty}\int_{B^{c}_{R}}\bigg[|u_k|^{\overline{p}}+|v|^{\overline{p}} +  \int_{\mathbb{R}^{N} }\frac{|u_k(x)-u_k(y)|^{p(x,y)}}{|x-y|^{N+sp(x,y)}} \,dy \\ +\int_{\mathbb{R}^{N} }\frac{|v_k(x)-v_k(y)|^{p(x,y)}}{|x-y|^{N+sp(x,y)}} \,dy \bigg]\,dx.
\end{split}
\end{equation}
Then,
\begin{equation}\label{cc2inf}
\limsup_{k \to \infty}\displaystyle{ \int_{\mathbb{R}^{N} }|u_k|^{\alpha(x)}|v_k|^{\beta(x)}\,dx}= \nu(\mathbb{R}^{N})+\nu_{\infty}.
\end{equation}
\begin{equation}\label{cc2infi}
\begin{split}
 \limsup_{k \to \infty}\int_{\mathbb{R}^{N} }\bigg[|u_k|^{\overline{p}}+|v_k|^{\overline{p}}  +  \int_{\mathbb{R}^{N} }\frac{|u_k(x)-u_k(y)|^{p(x,y)}}{|x-y|^{N+sp(x,y)}} \,dy \\ +\int_{\mathbb{R}^{N} }\frac{|v_k(x)-v_k(y)|^{p(x,y)}}{|x-y|^{N+sp(x,y)}} \,dy \bigg]\,dx= \mu (\mathbb{R}^{N}) + \mu_{\infty}.
 \end{split}
\end{equation}
Assume in addition that,
\begin{itemize}
\item[($\varepsilon_{0}$)] there exist $\displaystyle{\lim_{|x|,|y|\to \infty}p(x,y)=\overline{p}}$ and $\displaystyle{\lim_{|x|\to \infty}q(x)=q_{\infty}}$ for $\overline{p}$ given by $(\mathit{P})$ and some  $q_{\infty} \in (1,\infty)$.  
\end{itemize}
Then, 
\begin{equation}\label{SSin}
\mathcal{S}_{\alpha\beta}\mu_{\infty}^{\frac{1}{q_{\infty}}}\leqslant\nu_{\infty}^{\frac{1}{\overline{p}}}.
\end{equation} 
 \end{theorem}
  
  One remarks that, Theorem \ref{lionsinfity} brings information about the possible loss of mass at infinity, that is, in the same sense as Chabrowsky \cite{chabro} in the scenario of fractional Sobolev spaces with variable exponents. In this sense, this result extends the  \cite{chabro,kiho,temperini} ones. Moreover, the technique used to show this result  is similar to that one in the proof of Theorem \ref{lions}, were  the technical  Lemma \ref{4.5inf} is  crucial. Here, it is important to note that, we will apply this principle to obtain sufficient conditions for the existence of solutions for systems of type \eqref{s1} (see below) in all $\mathbb{R}^{N}$.

% As an application of  the concentration-compactness principle of Lions for systems, via the mountain pass geometry, Theorem \ref{aplica}, we will obtain a result of the existence of a solution for the following  elliptical system involving a new nonlocal integrodifferential operator with variable exponents throughout $ \mathbb{R}^{N}$, 
 
As an application of the Theorems \ref{lions} and \ref{lionsinfity}, this paper is devoted to the study of existence
of solutions for the following  elliptical system, which involves a new nonlocal integrodifferential operator with variable exponents in all of $\mathbb{R}^{N}$, that is,
 
 \begin{equation}\tag{$\wp_{S}$}
\label{s1}\left\{\begin{array}{rc}
\begin{split}
& \mathcal{L}_{\mathcal{A}K}u  = \lambda \mathrm{H}_{u}(x,u,v)+ \frac{\alpha(x)}{q(x)}|v|^{\beta(x)}|u|^{\alpha(x)-2}u,\\
&\mathcal{L}_{\mathcal{A}K}v  = \lambda \mathrm{H}_{v}(x,u,v)+ \frac{\beta(x)}{q(x)}|u|^{\alpha(x)}|v|^{\beta(x)-2}v,
\end{split}
\end{array}\right.
\end{equation}
where $\lambda$ is a positive real parameter, $p(\cdot)$ satisfies $(\mathit{P})$, $q(\cdot)$ satisfies $(\mathit{Q})$, $\alpha$ and $ \beta \in C_{+}(\mathbb{R}^{N})$  such that $\alpha(x)+\beta(x)=q(x)$ for all $x \in \mathbb{R}^{N}$. Moreover, the operator    $$ \mathcal{L}_{\mathcal{A}K}u:=\mathscr{L}_{\mathcal{A}K}u+|u|^{\overline{p}-2}u$$ 
is defined on suitable fractional Sobolev spaces  where $ \mathscr{L}_{\mathcal{A}K}$ is a  general nonlocal integrodifferential operator defined by
 \begin{equation*}
  \mathscr{L}_{\mathcal{A}K}u(x)= P.V. \int_{\mathbb{R}^{N}}\mathcal{A}(u(x)-u(y))K(x,y)\,dy, \hspace{0.2cm} x \in \mathbb{R}^{N},
  \end{equation*}
 where $P.V.$ is the principal value, and the map $\mathcal{A}:\mathbb{R}\to \mathbb{R}$ is a measurable function  satisfying the next assumptions:
\begin{itemize}
 \item[ $(a_1)$] $\mathcal{A}$ is continuous, odd, and the map $\mathscr{A}:\mathbb{R}\to \mathbb{R}$ given by
 $$\mathscr{A}(\zeta):= \int^{|\zeta|}_{0} \mathcal{A}(\tau) d\tau$$
is  strictly convex;
 \item[$(a_{2})$]  There exist positive constants $c_{\mathcal{A}}$ and $C_{\mathcal{A}}$, such that  for  all $\zeta \in \mathbb{R}$ and  for all $(x,y)\in \mathbb{R}^{2N}$
\begin{equation*}
\mathcal{A}(\zeta)\zeta\geqslant  c_{\mathcal{A}} \vert \zeta \vert^{p(x,y)} \hspace{0.3cm} \mbox{ and } \hspace{0.3cm} \vert \mathcal{A}(\zeta) \vert\leqslant C_{\mathcal{A}} \vert \zeta \vert^{p(x,y)-1};
\end{equation*}
\item[$(a_{3})$]$\mathcal{A}(\zeta)\zeta\leqslant \overline{p}\mathscr{A}(\zeta)$ for all $\zeta \in \mathbb{R}$.
\end{itemize}
%Furthermore we a assume convexity condition of Simon's type, that is
%\begin{itemize}
%\item[$(a_{4})$] There is positive constant  $C_{0} $  such that for all $(x,y)\in \mathbb{R}^{N}\times \mathbb{R}^{N}$
%\begin{equation*}\label{simon} |\zeta_2-\zeta_1|^{p(x,y)}\leqslant C_{0}\left\{\begin{array}{rc}
%\begin{split}
%&(\mathcal{A}(\zeta_2)- \mathcal{A}(\zeta_1))(\zeta_2-\zeta_1) \mbox{ for } p(x,y)\geqslant 2,\\
%&[(\mathcal{A}(\zeta_2)- \mathcal{A}(\zeta_1))(\zeta_2-\zeta_1)]^{\frac{p(x,y)}{2}}(|\zeta_2|^{p(x,y)}+|\zeta_1|^{p(x,y)})^{\frac{2-p(x,y)}{2}} \mbox{ for } 1<p(x,y)<2.
%\end{split}
%\end{array}\right.
%\end{equation*}
%\end{itemize}
  The kernel $K: \mathbb{R}^{2N} \to \mathbb{R}^{+}$ is a measurable function  satisfying the following property:
\begin{itemize}
 \item[ $(\mathcal{K})$] There exist constants $b_0$ and $b_1$, such that $0<b_0\leqslant b_1$,
 $$b_0\leqslant K(x,y)|x-y|^{N+sp(x,y)}\leqslant b_1 \mbox{ for all } (x,y)\in \mathbb{R}^{2N} \mbox{ and } x\neq y.$$
 \end{itemize}
 
We observe that, the assumptions $(a_1)$-$(a_3)$ and $(\mathcal{K})$ were similarly introduced  in \cite{bonaldo, bonaldo2}.
Note that a very special case for $\mathcal{A}$ and $K$ satisfying $(a_1)$-$(a_3)$ and $(\mathcal{K})$ is $\mathcal{A}(\zeta)=|\zeta|^{p(x,y)-2}\zeta$ and  $K(x,y)= |x-y|^{-(N+sp(x,y))}$. Hence the operator $\mathscr{L}_{\mathcal{A}K}$ reduces to the  fractional $p(\cdot)$-Laplacian operator $(-\Delta)^{s}_{p(\cdot)}$, which is defined by
\begin{equation*}
  (-\Delta)^{s}_{p(\cdot)}u(x)= P.V. \int_{\mathbb{R}^{N}}\frac{|u(x)-u(y)|^{p(x,y)-2}(u(x)-u(y))}{|x-y|^{N+sp(x,y)}}\,dy, \,\,\,\,\,  x \in \mathbb{R}^{N}.
  \end{equation*}
In the constant case, $p(\cdot,\cdot)\equiv p$, the operator $\mathcal{L}_{\mathcal{A}K}$ is the fractional $p-$Laplacian, denoted by $(-\Delta)^{s}_{p}$, and for $p=2$, it is the fractional Laplacian denoted by 
$(-\Delta)^{s}$. Moreover, it is well known that $(-\Delta)_{p}^{s}$ reduces to the standard $p$--Laplacian as $s \nearrow 1$ in the limit sense of Bourgain-Brezis-Mironescu, see e.g. \cite{mironescu}.

Moreover, the  operator $\mathcal{L}_{\mathcal{A}K}$ in the problem \eqref{s1} is a nonlocal general integrodifferencial operator, where the map $ \mathcal{A} $ and the kernel $ K $ under consideration are very general ($ K $ includes singular kernels). On the other hand,  problem \eqref{s1} involves the fractional $ p(\cdot)$--Laplacian operator, which has 
%more 
complicated nonlinearities, for instance, they are nonhomogeneous.
 
 We assume that the perturbed terms $\mathrm{H}_{u}$ and $\mathrm{H}_{v}$ are partial derivatives of a Carathéodory type function $\mathrm{H}$ satisfying the subcritical growth conditions:
  \begin{itemize}
  \item[$(\mathcal{H})$]For a.e.  $x \in \mathbb{R}^{N}$ it results  $\mathrm{H}(x, \cdot, \cdot)\in C^{1}(\mathbb{R}^{2})$,   $\mathrm{H}(x, \cdot, \cdot)\geqslant 0$ in $\mathbb{R}^{2}$,  $\mathrm{H}_{z}(x, 0, 0)=(0,0)$ with  $\mathrm{H}_{z}=(\mathrm{H}_{u}, \mathrm{H}_{v})$ and $z=(u,v)$. Furthermore, there are $\theta, \sigma, a \mbox{ and } b \in C_{+}(\mathbb{R}^{N})$  such that $\overline{p}<\theta^{-}\leqslant \theta^{+}<\sigma^{-}\leqslant \sigma^{+}< q^{-}$, $a\in L^{\frac{q(\cdot)}{q(\cdot)-\theta(\cdot)}}_{+}(\mathbb{R}^{N})$ and $b\in L^{\frac{q(\cdot)}{q(\cdot)-\sigma(\cdot)}}_{+}(\mathbb{R}^{N})$ (see \eqref{lali}), for which the inequality
  $$|\mathrm{H}_{u}(x,u,v)|+|\mathrm{H}_{v}(x,u,v)|\leqslant  a(x)|u|^{\theta(x)-1}+b(x)|v|^{\sigma(x)-1} $$
   for all  $(u,v) \in \mathbb{R}^{2}$  and a.e. $x\in \mathbb{R}^{N}$. Moreover, for all $ z \in \mathbb{R}^{2}$ hold for  a.e. $ x \in \mathbb{R}^{N}$  
  $$0 \leqslant \theta^{-}\mathrm{H}(x,z)\leqslant\mathrm{H}_z(x,z)\cdot z $$
 and  for all $(u,v) \in \mathbb{R}^{+}\times\mathbb{R}^{+}$,  the function  $\mathrm{H}(\cdot,u,v)>0$ for a.e. $x \in \mathbb{R}^{N}$.
    \end{itemize}

\noindent An example for the nonlinearity $\mathrm{H}$ is given by 
\begin{equation*}
    \mathrm{H}(x,u,v)= a(x) / \theta(x)  |u|^{\theta(x)} + b(x) / \sigma (x) |v|^{\sigma(x)} 
\end{equation*}
 with   $\theta, \sigma, a \mbox{ and } b \in C_{+}(\mathbb{R}^{N})$, such that $\overline{p}<\theta^{-}\leqslant \theta^{+}<\sigma^{-}\leqslant \sigma^{+}< q^{-}$,  $a\in L^{\frac{q(\cdot)}{q(\cdot)-\theta(\cdot)}}_{+}(\mathbb{R}^{N})$, and $b\in L^{\frac{q(\cdot)}{q(\cdot)-\sigma(\cdot)}}_{+}(\mathbb{R}^{N})$.
  
Now, we can state the result of existence of solutions for the system \eqref{s1}.
  
    \begin{theorem}\label{aplica}
  Assume that $(\mathit{P})$, $(\mathit{Q})$, $(a_{1})$-$(a_{3})$, $(\mathcal{K})$, and $(\mathcal{H})$ hold. Then  there exists $\lambda^{\star}>0$, such that the system  \eqref{s1} admits at least one nontrivial solution  $(u_{\lambda},v_{\lambda})$ in $\mathit{W}$ for all $\lambda \geqslant \lambda^{\star}$. Besides  each component of  $(u_{\lambda},v_{\lambda})$ is nontrivial and 
\begin{equation}\label{assim}
\lim_{\lambda\to \infty}\|(u_{\lambda},v_{\lambda})\|=0.
\end{equation}  
    \end{theorem}
Theorem \ref{aplica} completes in several directions the results of \cite{pucci,gou,temperini} not only for the presence of the integrodifferential operator but also for the conditions of moderate growth in the main elliptical operator. The presence of the nonlocal integrable differentiable operator $\mathcal{L}_{\mathcal{A}K}$ and the lack of compactness in the vector scenario make our analysis quite delicate. However, we overcome the difficulties presented applying the Theorems \ref{lions} and \ref{lionsinfity} and obtain the existence and assymptotic behaviour of a nontrivial solution for the system \eqref{s1} through the Mountain Pass Theorem. Further, it is important to note that since the solution $(u_{\lambda},v_{\lambda})$ built in this theorem, has both nontrivial components, it is evident that it solves a real system, which does not reduce to an equation.
 
We also consider a simplified version of the system \eqref{s1}, i.e. the following  system in $\mathbb{R}^{N}$
  
   \begin{equation}\tag{$\tilde{\wp}_{S}$}
\label{ss2}\left\{\begin{array}{rc}
\begin{split}
& \mathcal{L}_{\mathcal{A}K}u  = \lambda \mathrm{H}_{u}(u,v)+ \frac{\alpha(x)}{q(x)}|v|^{\beta(x)}|u|^{\alpha(x)-2}u,\\
&\mathcal{L}_{\mathcal{A}K}v  = \lambda \mathrm{H}_{v}(u,v)+ \frac{\beta(x)}{q(x)}|u|^{\alpha(x)}|v|^{\beta(x)-2}v.
\end{split}
\end{array}\right.
\end{equation}
The system \eqref{ss2} satisfies the conditions $(a_1)$-$(a_3)$, $\mathcal{K}$ and $\mathrm{H}$ independent of $x$, that is 
\begin{itemize}
  \item[$(\mathcal{H}')$]
   $\mathrm{H}( \cdot, \cdot)\in C^{1}(\mathbb{R}^{2})$,     $\mathrm{H}( \cdot, \cdot)\geqslant 0$ in $\mathbb{R}^{2}$,  $\mathrm{H}_{z}( 0, 0)=(0,0)$ with  $\mathrm{H}_{z}=(\mathrm{H}_{u}, \mathrm{H}_{v})$ and $z=(u,v)$. Then,  there are $\theta, \sigma \in C_{+}(\mathbb{R}^{N})$  such that $\overline{p}<\theta^{-}\leqslant \theta^{+}<\sigma^{-}\leqslant \sigma^{+}< q^{-}$ and positive constant  $a$ and $b$, such that
   $$|\mathrm{H}_{u}(u,v)|+|\mathrm{H}_{v}(u,v)|\leqslant  a|u|^{\theta(x)-1}+b|v|^{\sigma(x)-1} \mbox{ for all }  (u,v) \in \mathbb{R}^{2} \mbox{ and a.e. }x\in \mathbb{R}^{N}.$$
Moreover,  for all $ z \in \mathbb{R}^{2}$
$$0 \leqslant \theta^{-}\mathrm{H}(z)\leqslant\mathrm{H}_z(z)\cdot z $$ 
and for all $(u,v) \in \mathbb{R}^{+}\times\mathbb{R}^{+}$  the function  $\mathrm{H}(u,v)>0$.
    \end{itemize} 
In addition, we have the following
 \begin{theorem}\label{aplica2}
  Suppose  that $(\mathit{P})$, $(\mathit{Q})$, $(a_{1})$-$(a_{3})$, $(\mathcal{K})$, and $(\mathcal{H}')$ hold. Then  the system  \eqref{ss2} admits at least one nontrivial solution  $(u_{\lambda},v_{\lambda})$ in $\mathit{W}$ for all $\lambda > 0$. Besides  each component of  $(u_{\lambda},v_{\lambda})$ is nontrivial and  \eqref{assim}  holds.
  \end{theorem}
 
The strategy to prove 
%on which the argument for the existence of the 
Theorem \ref{aplica2} does not follow from 
%is based ignores 
the proof of the Palais-Smale condition ($(PS)$, for brevity). Therefore, we can find solutions for the system \eqref{ss2} with, possibly, different values of $\lambda$. Moreover, since the critical value $\lambda^{\star}$, in Theorem \ref{aplica}, derives from the  Proposition \ref{lemaD} with the application of Theorem \ref{lions} and \ref{lionsinfity}, we are allowed to find level sets $c_{\lambda}$ for which the condition $(PS)_{c_{\lambda}}$ is valid.

However, avoid proving the strong convergence of the $(PS)$ sequences, the fact that the solution constructed in Theorem \ref{aplica} is non-trivial does not follow automatically. So to overcome this difficulty, it was essential to prove a Lions type Lemma for fractional Sobolev spaces with variable exponents. Thus it was proved Lemma \ref{lemalions}, which is new and it seems an important result in the literature. Indeed, it shows us the conditions that guarantee strong convergence on Lebesgue spaces with variable exponents of a bounded sequence on fractional Sobolev spaces with variable exponents.
Accordingly, due to Lemma \ref{lemalions} and the Proposition \ref{bounded}, the proof of Theorem \ref{aplica2} is obtained in an elegant and interesting way, since Lemma \ref{lemalions} is another strategy to overcome the lack of compactness. Therefore, Theorem \ref{aplica2} continues to improve in several directions the previous results of \cite{pucci,gou,temperini}.
 
\medskip 
Below, we summarize the main contributions of this paper:
 \begin{itemize}
\item[$\bullet$] We deal with elliptic systems involving a new class of general nonlocal integrodifferential operators with exponent variables and critical growth conditions in $\mathbb{R}^{N}$.
\item[$\bullet$] A version of the Lions and Chabrowski Concentration-compactness principle is proved in the context of fractional Sobolev spaces with variable exponents, especially for nonlinear systems.
\item[$\bullet$]  An important result is obtained, namely Lions type Lemma, for fractional Sobolev spaces with variable exponents in the space $\mathit{W}$.
\item[$\bullet$] We achieved an assymptotic behaviour of the solutions  for the systems \eqref{s1} and \eqref{ss2}, that is, $\displaystyle\lim_{\lambda\to \infty}\|(u_{\lambda},v_{\lambda})\|=0.$
\item[$\bullet$] As far as the authors knowledge, this is the first time that a critical problem for systems has been studied in the context of variable exponents for fractional Sobolev spaces.
%item[$\bullet$] As fa, this is the first time that a critical problem for systems involving a new class of general nonlocal integrodifferential operators, has been studied in the context of variable exponents for the fractional Sobolev spaces in $\mathbb{R}^{N}$.
        \end{itemize}
        
%\subsection{Plan of the paper}
% In Section \ref{sec2}, we briefly review some properties of Lebesgue and fractional %Sobolev spaces with variable exponents. Section \ref{sec3},  defines the function space %for the systems \eqref{s1} and \eqref{ss2}, establish a critical Sobolev type embedding %for the fractional Sobolev spaces with variable exponents and we also a  version of %Brezis–Lieb type Lemma.  Further, it is  proved a version of the Lions and Chabrowski %Concentration-compactness principle in the context of fractional Sobolev spaces with %variable exponents. Section \ref{aplication},  provides results about the $ %\mathcal{L}_{AK} $ operator,  also shows the existence and asymptotic behaviour of a %nontrivial solution to system \eqref{s1}. Moreover,    is proved a version of Lions Lemma %type for fractional Sobolev spaces with variable exponents to obtain the existence and %assymptotic behaviour of a nontrivial solution to a simplified system \eqref{ss2}.

 % \section{\textbf{Preliminary stuff}}\label{sec2}
  \section{\textbf{Background and Preliminaries}}\label{sec2}
%  \hfill \break

  In this section, we recall  some results  and  definitions involving the Lebesgue spaces with variable exponents. We refer to \cite{bonaldo, bonaldo2, dieming,kiho, kihosim, kok} for the fundamental properties of these spaces.
\subsection*{  Basic notation}
Throughout this manuscript we will sometimes use the following notations for simplicity. $\Omega \times \Omega := \Omega^{2}$, $\mathbb{R}^{N}\times \mathbb{R}^{N}:= \mathbb{R}^{2N}$, $\mathbb{R}^{+}=(0,+\infty),$ and
$$B_{R}(x_{0})=\bigg\{x\in \mathbb{R}^{N}:|x-x_{0}|<R\bigg\}$$ denotes the ball of $\mathbb{R}^{N}$ centered at $x\in \mathbb{R}^{N}$ with radius $R>0$.

\subsection*{ Variable exponent function spaces}
 For   $h \in C_{+}(\mathbb{R}^{N})$, see \eqref{spacec} we denote  $$ h^{-} := \displaystyle{\inf_{x\in \mathbb{R}^{N}}h(x)}  \mbox{  and } h^{+} := \displaystyle{\sup_{x\in \mathbb{R}^{N}}h(x)}.  $$
 
  For $h \in  C_{+}(\mathbb{R}^{N})$ and a $\sigma$-finite, complete measure $\mu$ in $\mathbb{R}^{N}$,  the variable exponent Lebesgue space $L^{h(\cdot)}_{\mu}(\mathbb{R}^{N})$  is defined by
 \begin{equation*}\label{lp}
L^{h(\cdot)}_{\mu}(\mathbb{R}^{N}):=\Bigg\{u:\mathbb{R}^{N} \to \mathbb{R} \mbox{ is } \mu-\mbox{measurable},  \int_{\mathbb{R}^{N}}|u(x)|^{h(x)}\, d\mu <+\infty \Bigg\}.
\end{equation*}
This space is endowed with the so-called Luxemburg norm
 $$\|u\|_{L^{h(\cdot)}_{\mu}(\mathbb{R}^{N})}:= \inf\Bigg\{\zeta>0:  \int_{\mathbb{R}^{N}}\Big|\frac{u(x)}{\zeta}\Big|^{h(x)}\, d\mu\leqslant 1 \Bigg\}.$$
  When $\mu$ is the Lebesgue measure, we write $dx$, $L^{h(\cdot)}(\mathbb{R}^{N})$ and  $\|u\|_{L^{h(\cdot)}(\mathbb{R}^{N})}$ instead of $\,d\mu$, $L^{h(\cdot)}_{\mu}(\mathbb{R}^{N})$ and $\|u\|_{L^{h(\cdot)}_{\mu}(\mathbb{R}^{N})}$ respectively. Set 
  \begin{equation}\label{lali}
      L^{h(\cdot)}_{+}(\mathbb{R}^{N}):=\{ u \in L^{h(\cdot)}_{\mu}(\mathbb{R}^{N}): u>0 \mbox{ a.e. in }  \mathbb{R}^{N} \} 
  \end{equation}
   and for a Lebesgue measurable and positive a.e. function $\kappa:\mathbb{R}^{N}\to \mathbb{R}$, set $L^{h(\cdot)}(\kappa,\mathbb{R}^{N})
:= L^{h(\cdot)}_{\mu}(\mathbb{R}^{N})$ with $d\mu=\kappa(x)dx$. 
% Some basic properties of   
% $L^{h(\cdot)}_{\mu}(\mathbb{R}^{N})$
% are listed in the next  propositions.

 \begin{proposition}\textbf{(\cite{dieming, kok}).}
\label{masmenos}
Let $u\in L^{h(\cdot)}_{\mu}(\mathbb{R}^{N}) $, $\{u_{k}\}_{k\in \mathbb{N}}\subset L^{h(\cdot)}_{\mu}(\mathbb{R}^{N}),$ and define $h(\cdot)$-modular function of the  space $ L^{h(\cdot)}_{\mu}(\mathbb{R}^{N})$ as 
$$\rho_{h(\cdot)}(u)=\int_{\mathbb{R}^{N}}|u(x)|^{h(x)}\,d\mu.$$
Then we have the following relation between the modular function and the norm:
\begin{itemize}
  \item[$(a)$]For $u \in L^{h(\cdot)}_{\mu}(\mathbb{R}^{N})\setminus \{0 \}$, $\zeta = \|u\|_{L^{h(\cdot)}_{\mu}(\mathbb{R}^{N})} $ if and only if $ \rho_{h(\cdot)}\big(\frac{u}{\zeta}\big)=1 $;
    \item[$(b)$]$\|u\|_{L^{h(\cdot)}_{\mu}(\mathbb{R}^{N})}\geqslant 1\Rightarrow \|u\|_{L^{h(\cdot)}_{\mu}(\mathbb{R}^{N}) }^{h^{-}}\leqslant \rho_{h(\cdot)}(u)\leqslant \|u\|_{L^{h(\cdot)}_{\mu}(\mathbb{R}^{N}) }^{h^{+}};$
    \item[$(c)$] $\|u\|_{L^{h(\cdot)}_{\mu}(\mathbb{R}^{N})}\leqslant 1\Rightarrow \|u\|_{L^{h(\cdot)}_{\mu}(\mathbb{R}^{N})}^{h^{+}}\leqslant \rho_{h(\cdot)}(u)\leqslant \|u\|_{L^{h(\cdot)}_{\mu}(\mathbb{R}^{N})}^{h^{-}};$ 
 \item[$(d)$] $\lim\limits_{k\to+\infty} \|u_{k}\|_{L^{h(\cdot)}_{\mu}(\mathbb{R}^{N})}=0 \Leftrightarrow \lim\limits_{k\to+\infty} \rho_{h(\cdot)}(u_{k})=0;$
 \item[$(e)$] $\lim\limits_{k\to+\infty} \|u_{k}\|_{L^{h(\cdot)}_{\mu}(\mathbb{R}^{N})}=+\infty \Leftrightarrow \lim\limits_{k\to+\infty} \rho_{h(\cdot)}( u_{k})=+\infty.$
\end{itemize}
\end{proposition}

  \begin{proposition}\textbf{(\cite{kihosim}).}
\label{mista}
Let  $h$ and $g$  be measurable functions such that $h(x) \in L^{\infty}$ and $1< h(x)g(x) \leqslant\infty$, for a.e. $x \in \mathbb{R}^{N}$. Let $u\in L^{h(\cdot)}(\mathbb{R}^{N}) $,  $u\neq 0.$ Then 
\begin{itemize}
    \item[$(a)$] $\|u\|_{L^{h(\cdot)g(\cdot)}(\mathbb{R}^{N})}\leqslant 1 \Rightarrow \|u\|^{h^{-}}_{L^{h(\cdot)g(\cdot)}(\mathbb{R}^{N})}\leqslant  \|| u|^{h(\cdot)}\|_{L^{g(\cdot)}(\mathbb{R}^{N})}\leqslant \|u\|^{h^{+}}_{L^{h(\cdot)g(\cdot)}(\mathbb{R}^{N})};$
 \item[$(b)$] $\|u\|_{L^{h(\cdot)g(\cdot)}(\mathbb{R}^{N})}\geqslant 1 \Rightarrow \|u\|^{h^{+}}_{L^{h(\cdot)g(\cdot)}(\mathbb{R}^{N})}\leqslant  \|| u|^{h(\cdot)}\|_{L^{g(\cdot)}(\mathbb{R}^{N})}\leqslant \|u\|^{h^{-}}_{L^{h(\cdot)g(\cdot)}(\mathbb{R}^{N})}.$
 \end{itemize}
\end{proposition}

  \begin{proposition}\textbf{(\cite{dieming, kok}).}\label{hold3}
\begin{itemize}
\item[$(a)$] The space $(L^{h(\cdot)}(\mathbb{R}^{N}), \|\cdot\|_{L^{h(\cdot)}(\mathbb{R}^{N})})$ is a separable and reflexive Banach space;
\item[$(b)$]  For $u \in L^{h(\cdot)}(\mathbb{R}^{N})$ and $v \in L^{h'(\cdot)}(\mathbb{R}^{N})$. Then $uv \in L^{1}(\mathbb{R}^{N})$ and
 \begin{equation*} 
\int_{\mathbb{R}^{N}}|uv|\,dx\leqslant \bigg(\frac{1}{h^{-}}+\frac{1}{h'^{-}}\bigg)\|u\|_{ L^{h(\cdot)}(\mathbb{R}^{N})}\|v\|_{ L^{h'(\cdot)}(\mathbb{R}^{N})}.
\end{equation*}

\end{itemize}

\end{proposition}

  \subsection{Fractional Sobolev Space}\label{espacofra}
  In this section, we recall the definitions and some results involving the fractional Sobolev spaces with variable exponents. We refer the reader to \cite{bonaldo, bonaldo2,  kiho} for further details about the functional framework that follows.
  
  Throughout this subsection, we assume that:
\begin{itemize}
\item[$(\mathfrak{p})$] Let $ p(x):=p(x,x)$ for all $ x \in \mathbb{R}^{N}$ with $ p: \mathbb{R}^{N} \times\mathbb{R}^{N}\to \mathbb{R}$  uniformly continuous satisfying: 
\begin{equation*}\label{a23}
  \begin{split}
   ¨& p \mbox{ is symmetric, that is, } p(x,y)=p(y,x),\\
  & 1< p^{-} :=\displaystyle{ \inf_{(x,y) \in \mathbb{R}^{2 N}}}\,p(x,y) \leqslant \displaystyle{\sup_{(x,y)\in \mathbb{R}^{2N}}}\,p(x,y):= p^{+}<\frac{N}{s},\hspace{0.2cm} s \in(0,1).
 \end{split}
  \end{equation*}
\end{itemize}
  We denote by  $p^{\star}_{s}(x)= Np(x)/(N-sp(x))$  for all $x \in \mathbb{R}^{N}$, the fractional critical variable exponent related to $ p  \in C(\mathbb{R}^{2N}).$
   Let   $W^{s,p(\cdot, \cdot)}(\mathbb{R}^{N})$, the fractional Sobolev spaces with variable exponents defined by
  \begin{equation*}
\begin{split}
  W^{s,p(\cdot, \cdot)}(\mathbb{R}^{N}):=  \Bigg\{u\in L^{p(\cdot)}(\mathbb{R}^{N}):  \int_{\mathbb{R}^{N}}\int_{\mathbb{R}^{N}}\frac{|u(x)-u(y)|^{p(x,y)}}{\zeta^{p(x,y)}|x-y|^{N+sp(x,y)}} \,dy\,dx< +\infty, \mbox{ for some } \zeta>0 \Bigg\} 
\end{split}
\end{equation*}
and the corresponding norm for this space  is
$$ \| u \|:= \inf\bigg\{ \zeta >0; \int_{\mathbb{R}^{N}}\frac{|u|^{p(x)}}{\xi^{p(x)}}\,dx+  \int_{\mathbb{R}^{N}}\int_{\mathbb{R}^{N}}\frac{|u(x)-u(y)|^{p(x,y)}}{\zeta^{p(x,y)}|x-y|^{N+sp(x,y)}}\,dy\,dx<1\bigg\}.$$

 By \cite[Section 3]{kiho}, $W^{s,p(\cdot, \cdot)}(\mathbb{R}^{N})$   is a separable and reflexive Banach space.
  
  \begin{theorem}\label{subcrtiticalembeddings} \textbf{(\cite{kiho}).}
  \begin{itemize}
  \item[$(a)$]$W^{s,p(\cdot, \cdot)}(\mathbb{R}^{N})\hookrightarrow L^{r(\cdot)}(\mathbb{R}^{N})$ for any uniformly continuous function $r \in C_{+}(\mathbb{R}^{N})$ satisfying $p(x)\leqslant r(x)$ for all $x \in \mathbb{R}^{N}$ and $\displaystyle{\inf_{x \in \mathbb{R}^{N}}[p^{\star}_{s}(x)-r(x)]>0}$;
  \item[$(b)$] $W^{s,p(\cdot, \cdot)}(\mathbb{R}^{N})\hookrightarrow\hookrightarrow L^{r(\cdot)}_{loc}(\mathbb{R}^{N})$ for any $r \in C_{+}(\mathbb{R}^{N})$ satisfying $r(x)<p^{\star}_{s}(x)$ for all $x \in \mathbb{R}^{N}$.
  \end{itemize}
  \end{theorem}
  %cite[Critical embedding, Theorem 3.3]{kiho} ]
  \begin{theorem}\textbf{(\cite{kiho}).}\label{critico}
    Let $\Omega$ be a bounded Lipschitz domain in $\mathbb{R}^{N}$ or $\Omega=\mathbb{R}^{N}$. Let $(\mathfrak{p})$ hold. Furthermore, let the variable exponent $p$ satisfy the following log-H\"older  type continuity condition
         \begin{equation} \label{3.2}
\inf_{\varepsilon>0}\sup_{\substack{
(x,y)\in \Omega^2 \\
0<|x-y|<1/2}}
|p(x,y)-p^{-}_{\Omega^{2}{x,\varepsilon} }  |
\log\frac{1}{|x-y|}<\infty
\end{equation}
  where $\Omega_{z,\varepsilon}:=B_{\varepsilon}(z)\cap \Omega$ for $z \in \Omega$, $\varepsilon>0,$ and $\displaystyle{p^{-}_{\Omega^{2}_{x,\varepsilon}}:=\inf_{(x',y')\in \Omega^{2}_{x,\varepsilon}}p(x',y')}$. Let $q: \overline{\Omega}\to \mathbb{R}$ be a function satisfying
  \begin{itemize}
\item[$(\mathfrak{q})$]  $ q \in C_{+}(\Omega)$ such that for any $x \in \Omega$, there exists $\varepsilon=\varepsilon(x) > 0$ such that 
\begin{equation*}\label{a23}
  \begin{split}
  \sup_{y \in \Omega_{x,\varepsilon}}q(y)\leqslant \frac{N \inf_{(y,z)\in \Omega^{2}_{y,\varepsilon}}p(y,z)}{N-s\inf_{(y,z)\in \Omega^{2}_{y,\varepsilon}}p(y,z)}.
 \end{split}
  \end{equation*}
  In addition, when $\Omega=\mathbb{R}^{N}$ , $q$ is uniformly continuous and $p(x)< q(x)$ for all $ x \in \mathbb{R}^{N}$. Then, it holds that
$$W^{s,p(\cdot,\cdot)}(\Omega)\hookrightarrow L^{q(\cdot)}(\Omega).$$
\end{itemize}
  \end{theorem}
 One remarks that, the result above permit us to conclude a compact embedding from $W^{s,p(\cdot,\cdot)}(\mathbb{R}^{N})$ into the weighted Lebesgue spaces with variable exponents.

\begin{theorem} \textbf{(\cite{kiho}).} \label{criticoa}
 Assume that $(\mathfrak{p})$, $(\mathfrak{q})$ and  \eqref{3.2} hold. Let $w \in L^{\frac{q(\cdot)}{q(\cdot)-r(\cdot)}}_{+}(\mathbb{R}^{N})$ for some $ r \in C_{+}(\mathbb{R}^{N})$, such that  $\displaystyle{\inf_{x \in \mathbb{R}^{N}}[q(x)-r(x)]>0}$. Then, 
 $$W^{s,p(\cdot, \cdot)}(\mathbb{R}^{N})\hookrightarrow\hookrightarrow L^{r(\cdot)}(w,\mathbb{R}^{N}).$$
\end{theorem}

It follows that, the space $\mathit{W}$ defined by \eqref{numero} is a separable and reflexive Banach space, endowed with the norm  
$$ \| (u, v) \|:= \| u \|+\| v \|.$$

 Let $(s, p(\cdot, \cdot))$-convex modular function $\rho_{\mathit{W}}: \mathit{W} \to \mathbb{R}$ defined by
  \begin{equation*}
  \begin{split}
  \rho_{\mathit{W}}(u,v)=&\int_{\mathbb{R}^{N}}\frac{|u|^{p(x)}}{\xi^{p(x)}}\,dx +  \int_{\mathbb{R}^{N}}\int_{\mathbb{R}^{N}}\frac{|u(x)-u(y)|^{p(x,y)}}{\xi^{p(x,y)}|x-y|^{N+sp(x,y)}}\,dy\,dx\\&+\int_{\mathbb{R}^{N}}\frac{|v|^{p(x)}}{\xi^{p(x)}}\,dx +  \int_{\mathbb{R}^{N}}\int_{\mathbb{R}^{N}}\frac{|v(x)-v(y)|^{p(x,y)}}{\xi^{p(x,y)}|x-y|^{N+sp(x,y)}}\,dy\,dx.
  \end{split}
  \end{equation*}

The following proposition  shows the relationship between the norm  $\|(\cdot,\cdot)\|$  of the space   $\mathit{W}$  and the $\rho_{\mathit{W}}$ convex modular function. This result is obtained with a similar argumentation given by the proof of Proposition \ref{masmenos}.
  
  \begin{proposition}\label{norma}
For $(u,v)\in \mathit{W}$ and $\{(u_{k}, v_k)\}_{k\in \mathbb{N}}\subset\mathit{W}$, we have:
\begin{itemize}
\item[$(a)$]For $(u,v)\in \mathit{W}\setminus \{(0,0) \}$, $\varsigma = \|(u,v)\| $ if and only if $ \rho_{\mathit{W}}\big(\frac{(u,v)}{\varsigma}\big)=1 $;
    \item[$(b)$]$\|(u,v)\|\geqslant 1\Rightarrow \|(u,v)\|^{p^{-}}\leqslant \rho_{ \mathit{W}}(u,v)\leqslant \|(u,v)\|^{p^{+}};$
    \item[$(c)$] $\|(u,v)\|\leqslant 1\Rightarrow \|(u,v)\|^{p^{+}}\leqslant \rho_{ \mathit{W}}(u,v)\leqslant \|(u,v)\|^{p^{-}};$ 
    \item[$(d)$] $\lim\limits_{k\to+\infty} \|(u_{k}, v_k)-(u,v)\|=0 \Leftrightarrow \lim\limits_{k\to+\infty} \rho_{ \mathit{W}}((u_{k}, v_k)-(u,v))=0;$
 \item[$(e)$] $\lim\limits_{k\to+\infty} \|(u_{k},v_k)\|=+\infty \Leftrightarrow \lim\limits_{k\to+\infty} \rho_{ \mathit{W}}((u_{k}, v_{k}))=+\infty.$
\end{itemize}
\end{proposition}
 As consequence of the Theorem \ref{subcrtiticalembeddings} and  \ref{criticoa}, we obtain the following lemmas:
   \begin{lemma}\label{mergulho}
  \begin{itemize}
  \item[]
  \item[$(a)$]$\mathit{W}\hookrightarrow L^{r(\cdot)}(\mathbb{R}^{N})\times L^{r(\cdot)}(\mathbb{R}^{N})$ for any uniformly continuous function $r \in C_{+}(\mathbb{R}^{N})$ satisfying $p(x)\leqslant r(x)$ for all $x \in \mathbb{R}^{N}$ and $\displaystyle{\inf_{x \in \mathbb{R}^{N}}(p^{\star}_{s}(x)-r(x))>0}$;
  \item[$(b)$] $\mathit{W}\hookrightarrow\hookrightarrow L^{r(\cdot)}_{loc}(\mathbb{R}^{N})\times L^{r(\cdot)}_{loc}(\mathbb{R}^{N})$ for any $r \in C_{+}(\mathbb{R}^{N})$ satisfying $r(x)<p^{\star}_{s}(x)$ for all $x \in \mathbb{R}^{N}$.
  \end{itemize}
  \end{lemma}
  
 \begin{lemma} \label{mergulhocom}
 Assume that $(\mathfrak{p})$, $(\mathfrak{q})$ and condition \eqref{3.2} hold. Let $\kappa \in L^{\frac{q(\cdot)}{q(\cdot)-r(\cdot)}}_{+}(\mathbb{R}^{N})$ for some $ r \in C_{+}(\mathbb{R}^{N})$ such that  $\displaystyle{\inf_{x \in \mathbb{R}^{N}}[q(x)-r(x)]>0}$. Then, it holds that
 $$\mathit{W}\hookrightarrow\hookrightarrow L^{r(\cdot)}(\kappa,\mathbb{R}^{N})\times L^{r(\cdot)}(\kappa,\mathbb{R}^{N}).$$
\end{lemma}

\subsection{Concentration-compactness principle  in $\mathit{W}$ }\label{sec3}
We consider, without further mentioning, that the assumptions required in Theorems \ref{lions} and  \ref{lionsinfity} are satisfied.   Moreover, we  will denote by   $c_{i}$, $c_{i}(\mathcal{R})$ positive constants for all $i \in \mathbb{N}$ and $\mathcal{R} > 2$.
 %we establish Lions  type concentration compactness principles for the spaces  $\mathit{W}$ %
 
First, we note that if $p$ satisfies $(\mathit{P})$, then it satisfies $(\mathfrak{p})$ and \eqref{3.2}, also $p(x, x) = \overline{p} $ for all $x \in \mathbb{R}^{N}$.
    Hence, by Theorem \ref{critico}, we have that
  $ W^{s,p(\cdot,\cdot)}(\mathbb{R}^{N})\hookrightarrow L^{\overline{p}^{\star}_{s}}(\mathbb{R}^{N}).$
Then,
\begin{equation}\label{imer}
\mathit{W} \hookrightarrow L^{\overline{p}^{\star}_{s}}(\mathbb{R}^{N})\times  L^{\overline{p}^{\star}_{s}}(\mathbb{R}^{N}).
\end{equation} 
  Moreover, we conclude that $L^{\overline{p}}(\mathbb{R}^{N})\subset L^{p_{\star}(\cdot)}(\mathbb{R}^{N})$ and $W^{s,p(\cdot,\cdot)}(\mathbb{R}^{N})\hookrightarrow L^{t(\cdot)}(\mathbb{R}^{N})$ for any $t  \in C_{+}(\mathbb{R}^{N})$ satisfying $p_{\star}(x)\leqslant t(x) \leqslant \overline{p}^{\star}_{s}$ for all $x \in \mathbb{R}^{N}$ (see   \cite{kiho}).  
  
   Thus, for any $t  \in C_{+}(\mathbb{R}^{N})$ satisfying $p_{\star}(x)\leqslant t(x) \leqslant \overline{p}^{\star}_{s}$ for all $x \in \mathbb{R}^{N}$,
   \begin{equation}\label{remark}
   \mathit{W}\hookrightarrow L^{t(\cdot)}(\mathbb{R}^{N})\times  L^{t(\cdot)}(\mathbb{R}^{N}).
   \end{equation}
   
 Further, since $\alpha(x)+\beta(x)=q(x)$  for all $x \in \mathbb{R}^{N}$,  $q$ satisfies $(\mathit{Q})$, and by Proposition \ref{mista}, also H\"older inequality (see Proposition \ref{hold3}), we have
 \begin{equation}\label{4.4}
 \begin{split}
 \int_{\mathbb{R}^{N}}|u|^{\alpha(x)}|v|^{\beta(x)}\,dx \leqslant c_{1} \max \Big\{\|u\|^{\alpha^{-}}_{L^{q(\cdot)}(\mathbb{R}^{N})}\|v\|^{\beta^{-}}_{L^{q(\cdot)}(\mathbb{R}^{N})},  \|u\|^{\alpha^{-}}_{L^{q(\cdot)}(\mathbb{R}^{N})}\|v\|^{\beta^{+}}_{L^{q(\cdot)}(\mathbb{R}^{N})},\\ \|u\|^{\alpha^{+}}_{L^{q(\cdot)}(\mathbb{R}^{N})}\|v\|^{\beta^{-}}_{L^{q(\cdot)}(\mathbb{R}^{N})},\|u\|^{\alpha^{+}}_{L^{q(\cdot)}(\mathbb{R}^{N})}\|v\|^{\beta^{+}}_{L^{q(\cdot)}(\mathbb{R}^{N})}\Big\}.
 \end{split}
 \end{equation}
 
 On the other hand, since the following continuous embedding is valid, i.e. $W^{s,p(\cdot,\cdot)}(\mathbb{R}^{N})\hookrightarrow L^{q(\cdot)}(\mathbb{R}^{N})$, for each $q \in C_{+}(\mathbb{R}^{N})$ satisfying $p_{\star}(x) \leqslant q(x)\leqslant  \overline{p}^{\star}_{s}$ for all $x \in \mathbb{R}^{N}$, we conclude from \eqref{4.4} that
  \begin{equation}\label{holder}
 \begin{split}
 \int_{\mathbb{R}^{N}}|u|^{\alpha(x)}|v|^{\beta(x)}\,dx \leqslant &c_{2}\max \Big\{\|u\|^{\alpha^{-}}\|v\|^{\beta^{-}}, \|u\|^{\alpha^{-}}\|v\|^{\beta^{+}},  \|u\|^{\alpha^{+}}\|v\|^{\beta^{-}},\|u\|^{\alpha^{+}}\|v\|^{\beta^{+}}\Big\}\\ \leqslant& c_{2}\max \Big\{\|(u,v)\|^{q^{-}}, \|(u,v)\|^{q^{+}}\Big\}.
 \end{split}
 \end{equation}

      \noindent In particular, from definition of $\big\|u^{\frac{\alpha(\cdot)}{q(\cdot)}} v^{\frac{\beta(\cdot)}{q(\cdot)}}\big\|_{L^{q(\cdot)}(\mathbb{R}^{N})}$, \eqref{4.4}  and \eqref{holder}, yields
$\mathcal{S}_{\alpha\beta}>0$ (see \eqref{SS}).

  \subsection{Technical lemmas}
The following Lemma \ref{4.4a} and  \ref{4.5inf} are essential to prove the Theorem \ref{lions}  and \ref{lionsinfity} respectively. The proof of these results are similar to \cite[Lemmas 4.4 and  4.5]{kihosim}, therefore the proofs are omitted.
  \begin{lemma}\label{4.4a}
  Let $x_{0} \in \mathbb{R}^{N}$ be fixed and let $\Psi \in C^{\infty}(\mathbb{R}^{N})$ be such that $0 \leqslant \Psi \leqslant 1,$ $\Psi\equiv 1$ on $B_1$, $supp (\Psi) \subset B_2$ and $\| \nabla \Psi\|_{\infty} \leqslant 2.$  For $\rho>0$,  define $\Psi_{\rho}(x):= \Psi\big( \frac{x-x_0}{\rho}\big)$ for $  x \in \mathbb{R}^{N}$. Let
$(\mathit{P})$ hold and let $\{(u_k, v_k) \}_{k \in \mathbb{N}}$ be as in Theorem \ref{lions}. Then, we have
\begin{equation}\label{4000}
\limsup_{\rho\to 0^{+}}\limsup_{k\to +\infty} \int_{\mathbb{R}^{N}}\int_{\mathbb{R}^{N}}\big(|u_{k}(x)|^{p(x,y)}+ |v_{k}(x)|^{p(x,y)}\big)\frac{|\Psi_{\rho}(x)-\Psi_{\rho}(y)|^{p(x,y)}}{|x-y|^{N+sp(x,y)}} \,dy\,dx=0.
\end{equation}
  \end{lemma}

  \begin{lemma}\label{4.5inf}
   Let  $\Phi \in C^{\infty}(\mathbb{R}^{N})$  be such that $0 \leqslant \Phi \leqslant 1,$ $\Phi\equiv 1$ on $\mathbb{R}^{N}\setminus B_2$,  and $\| \nabla \Phi\|_{\infty} \leqslant 2.$  For $R>0$,  define $\Phi_{R}(x):= \Phi\big( \frac{x}{R}\big)$ for $  x \in \mathbb{R}^{N}$. Let
$(\mathit{P})$ hold and let $\{(u_k, v_k) \}_{k \in \mathbb{N}}$ be as in Theorem \ref{lionsinfity}. Then, we have
\begin{equation*}\label{40001}
\limsup_{R\to\infty}\limsup_{k\to +\infty} \int_{\mathbb{R}^{N}}\int_{\mathbb{R}^{N}}\big(|u_{k}(x)|^{p(x,y)}+ |v_{k}(x)|^{p(x,y)}\big)\frac{|\Phi_{R}(x)-\Phi_{R}(y)|^{p(x,y)}}{|x-y|^{N+sp(x,y)}} \,dy\,dx =0.
\end{equation*} 
    \end{lemma}
 
  \noindent  The next lemma is a variant of the Brezis–Lieb Lemma, for vectorial critical terms.
    \begin{lemma}\label{lieb}
    Let   $\{(u_{k}, v_{k})\}_{k\in \mathbb{N}}$ be a bounded sequence in $\mathit{W}$ and $(u_{k}(x), v_{k}(x))\to (u(x),v(x))$ for a.e. $x \in\mathbb{R}^{N}$. Then 
    $$\lim_{k\to+\infty}\int_{\mathbb{R}^{N}}\left( |u_k|^{\alpha(x)}|v_{k}|^{\beta(x)}-|\tilde{u}_k|^{\alpha(x)}|\tilde{v_k}|^{\beta(x)}-|u|^{\alpha(x)}|v|^{\beta(x)}\right)\,dx=0,$$
    where $\tilde{u}_{k}=u_k-u$ and $\tilde{v}_{k}=v_{k}-v.$
    \end{lemma}
    \begin{proof}
    Note that
    \begin{equation}\label{igual}
    \begin{split}
    |u_k|^{\alpha(x)}|v_k|^{\beta(x)}- |\tilde{u}_{k}|^{\alpha(x)}|\tilde{v}_{k}|^{\beta(x)}= \big[|u_k|^{\alpha(x)} - |\tilde{u}_{k}|^{\alpha(x)}]|v_k|^{\beta(x)}+ |\tilde{u}_{k}|^{\alpha(x)}[|v_k|^{\beta(x)}-|\tilde{v}_{k}|^{\beta(x)}\big]
    \end{split}
    \end{equation}
    for all $ x \in \mathbb{R}^{N}.$
    Now, for $\varepsilon>0$ fixed, there exists a constant $c(\varepsilon)>0$, such that for all $\mathfrak{a},\mathfrak{b} \in \mathbb{R}$, $r \in C_{+}(\mathbb{R}^{N})$ and for all $ x \in \mathbb{R}^{N}$,
    $$||\mathfrak{a}+\mathfrak{b}|^{r(x)}-|\mathfrak{a}|^{r(x)}|\leqslant \varepsilon|\mathfrak{a}|^{r(x)}+c(\varepsilon)|\mathfrak{b}|^{r(x)}.$$
  Hence,  by \eqref{igual}, we obtain
   \begin{equation}\label{igual2}
    \begin{split}
    |u_k|^{\alpha(x)}|v_k|^{\beta(x)} -|\tilde{u}_{k}|^{\alpha(x)}|\tilde{v}_{k}|^{\beta(x)} \leqslant &\big[\varepsilon|\tilde{v}_{k}|^{\alpha(x)}+ c(\varepsilon) |u|^{\alpha(x)}\big]|v_k|^{\beta(x)}\\&+ |\tilde{u}_{k}|^{\alpha(x)}\big[\varepsilon|\tilde{v}_{k}|^{\beta(x)}+c(\varepsilon)|v|^{\beta(x)}\big].
    \end{split}
    \end{equation}
Now, we  define the sequence
    $$f_{\varepsilon,k}(x):= \max\Big\{ | |u_k|^{\alpha(x)}|v_{k}|^{\beta(x)}-|\tilde{u}_k|^{\alpha(x)}|\tilde{v_k}|^{\beta(x)}-|u|^{\alpha(x)}|v|^{\beta(x)}| - \varepsilon |\tilde{u}_{k}|^{\alpha(x)}|\tilde{v}_{k}|^{\beta(x)}, 0 \Big\}.$$
    Then, we conclude that    $f_{\varepsilon,k}(x)\to 0$ a.e. $x \in \mathbb{R}^{N}$ as $k \to +\infty$. Moreover,  by \eqref{igual2} we conclude that  $|f_{\varepsilon,k}(x)|<H_{\varepsilon} \in L^{1}(\mathbb{R}^{N})$, where $$H_{\varepsilon}(x)= \varepsilon |u_k-u|^{\alpha(x)}|v_k|^{\beta(x)}+c(\varepsilon)|u|^{\alpha(x)}|v_k|^{\beta(x)}+c(\varepsilon)|u_k-u|^{\alpha(x)}|v|^{\beta(x)}+|u|^{\alpha(x)}|v|^{\beta(x)}.$$
    \noindent Thus, applying the Dominated Convergence Theorem, we obtain
    \begin{equation}\label{leimitef}
    \lim_{k\to+\infty}\int_{\mathbb{R}^{N}}f_{\varepsilon,k}(x)\,dx= 0.
    \end{equation}
    On the other hand, note that 
    $$||u_k|^{\alpha(x)}|v_{k}|^{\beta(x)}-|\tilde{u}_k|^{\alpha(x)}|\tilde{v_k}|^{\beta(x)}-|u|^{\alpha(x)}|v|^{\beta(x)}|\leqslant f_{\varepsilon,k}(x) + \varepsilon|\tilde{u}_{k}|^{\alpha(x)}|\tilde{v}_{k}|^{\beta(x)}. $$
   Hence, from \eqref{leimitef}, we obtain
     $$\lim_{k\to+\infty}\int_{\mathbb{R}^{N}}\big(|u_k|^{\alpha(x)}|v_{k}|^{\beta(x)}-|\tilde{u}_k|^{\alpha(x)}|\tilde{v_k}|^{\beta(x)}-|u|^{\alpha(x)}|v|^{\beta(x)}\big)\,dx=0.$$
    \end{proof}
    
%%%%%%%%%%%%%%%%%%%%%%%%%%%%%%%%%%%%%%%%%%%%%%%%%%%%%%%%%%%%%%%%%%%%%%%    
    \section{Proof of Theorems \ref{lions} and \ref{lionsinfity}}
%%%%%%%%%%%%%%%%%%%%%%%%%%%%%%%%%%%%%%%%%%%%%%%%%%%%%%%%%%%%%%%%%%%%%%%    \hfill \break 

The main issue of this section is to provide, in details, the proof of 
Theorems \ref{lions} and \ref{lionsinfity}. 

    \subsection{Proof of Theorem \ref{lions}}
      \begin{proof}
     \noindent \textbf{1.} First, we set $\tilde{u}_{k}=u_k-u$ and $\tilde{v}_{k}=v_{k}-v$. Then   
\begin{equation}\label{conver}
(\tilde{u}_{k}, \tilde{v}_{k}) \longrightarrow (0,0)   \mbox{ in } \mathit{W}.
\end{equation}    
From Lemma \ref{mergulho} and  \eqref{conver}, we obtain
    \begin{equation}\label{converr}
    (\tilde{u}_{k}, \tilde{v}_{k})\to (0,0)    \mbox{ in } L^{r(\cdot)}_{loc}(\mathbb{R}^{N})\times L^{r(\cdot)}_{loc}(\mathbb{R}^{N}) \mbox{ as } k \to +\infty
    \end{equation}
    for any $r \in C_{+}(\mathbb{R}^{N})$ satisfying $ r(x)< \overline{p}^{\star}_{s}$ for all $x \in \mathbb{R}^{N}$. Hence, up to a subsequence, we get
    \begin{equation}\label{pontual}
    (\tilde{u}_k(x), \tilde{v}_k(x))\to (0,0)  \mbox{ a.e. } x \in \mathbb{R}^{N}.
    \end{equation}
Using \eqref{c3}, \eqref{converr}, \eqref{pontual}, and Lemma \ref{lieb}, it follows that     
    \begin{equation}\label{5.0}
     |\tilde{u}_k|^{\alpha(x)}|\tilde{v}_k|^{\beta(x)} \overset{\ast}{\rightharpoonup} \nu - |u|^{\alpha(x)}|v|^{\beta(x)}= \overline{\nu} \mbox{ in }  \mathcal{M}(\mathbb{R}^{N}).
    \end{equation}
 Now, denoting   for all $k \in \mathbb{N}$   
   $$\displaystyle{\mathcal{U}_{k}=|\tilde{u}_{k}|^{\overline{p}}+\int_{\mathbb{R}^{N}}\frac{|\tilde{u}_{k}(x)-\tilde{u}_{k}(y)|^{p(x,y)}}{|x-y|^{N+sp(x,y)}}\,dy,} \mbox{ and } \displaystyle{ \mathcal{V}_{k}=|\tilde{v}_{k}|^{\overline{p}}+\int_{\mathbb{R}^{N}}\frac{|\tilde{v}_{k}(x)-\tilde{v}_{k}(y)|^{p(x,y)}}{|x-y|^{N+sp(x,y)}}\,dy},$$
    we have that the sequence $\{(\mathcal{U}_{k}, \mathcal{V}_{k})\}_{k \in \mathbb{N}}$ is bounded in $L^{1}(\mathbb{R}^{2N})$. Then, up to a subsequence (see \cite[Proposition  1.202]{fonseca}), we have
   \begin{equation}\label{5.ao}
    \mathcal{U}_{k}+ \mathcal{V}_{k} \overset{\ast}{\rightharpoonup} \overline{\mu} \mbox{ in }  \mathcal{M}(\mathbb{R}^{N})
   \end{equation}
    for some nonnegative finite Radon measure $\overline{\mu}$ in $\mathbb{R}^{N}$. Let $\phi \in C^{\infty}_{c}(\mathbb{R}^{N})$ and $\mathcal{R}>2$, such that 
    \begin{equation}\label{5.a}
    supp\phi \subset B_{\mathcal{R}} \mbox{ and } d=dist(B^{c}_{\mathcal{R}};supp\phi) \geqslant 1 + \frac{\mathcal{R}}{2}.
    \end{equation}
Hence from \eqref{SS}, we obtain
    \begin{equation}\label{5.a1}
\mathcal{S}_{\alpha\beta}\bigg\|(\phi  \tilde{u})^{\frac{\alpha(\cdot)}{q(\cdot)}}(\phi  \tilde{v})^{\frac{\beta(\cdot)}{q(\cdot)}}\bigg\|_{L^{q(\cdot)}(\mathbb{R}^{N})} \leqslant \|(\phi  \tilde{u}_k,\phi  \tilde{v}_k)\|= \|\phi  \tilde{u}_k\|+ \|\phi  \tilde{v}_k\|.
    \end{equation}
    
\textbf{2.} Now, set $\overline{\nu}_{k}:= |\tilde{u}_k|^{\alpha(x)}|\tilde{v}_k|^{\beta(x)}$, $\vartheta_{k}:= \|(\phi  \tilde{u}_k,\phi  \tilde{v}_k)\|$, and 
$$\overline{\mu}_{k}:= |\tilde{u}_k|^{\overline{p}}+|\tilde{v}_{k}|^{\overline{p}}+ \int_{\mathbb{R}^{N} }\frac{|\tilde{u}_k(x)-\tilde{u}_k(y)|^{p(x,y)}}{|x-y|^{N+sp(x,y)}}\,dy+\int_{\mathbb{R}^{N} }\frac{|\tilde{v}_k(x)-\tilde{v}_k(y)|^{p(x,y)}}{|x-y|^{N+sp(x,y)}} \,dy. $$ 
Given $\varepsilon>0$, there exists a constant $C(\varepsilon)\in (2, \infty)$ such that for all $\mathfrak{c}, \mathfrak{d} \in \mathbb{R}$, $p\in C_{+}(\mathbb{R}^{2N})$  and  $\mbox{ for all } (x,y) \in \mathbb{R}^{2N},$ 
\begin{equation}\label{desiab}
|\mathfrak{c}+\mathfrak{d}|^{p(x,y)}\leqslant(1+ \varepsilon)|\mathfrak{c}|^{p(x,y)}+C(\varepsilon)|\mathfrak{d}|^{p(x,y)}.
\end{equation}
    Thus, due to Proposition \ref{norma} and \eqref{desiab}, we see that
    \begin{equation}\label{5.10a}
    \begin{split}
    1=& \int_{\mathbb{R}^{N}}\bigg|\frac{\phi  \tilde{u}_k}{\vartheta_{k}}\bigg|^{\overline{p}}\,dx+  \int_{\mathbb{R}^{N}}\int_{\mathbb{R}^{N}}\frac{|(\phi  \tilde{u}_k)(x)-(\phi  \tilde{u}_k)(y)|^{p(x,y)}}{\vartheta_{k}^{p(x,y)}|x-y|^{N+sp(x,y)}} \,dy\,dx\\ &+\int_{\mathbb{R}^{N}}\bigg|\frac{\phi  \tilde{v}_k}{\vartheta_{k}}\bigg|^{\overline{p}}\,dx+ \int_{\mathbb{R}^{N}}\int_{\mathbb{R}^{N}}\frac{|(\phi  \tilde{v}_k)(x)-(\phi  \tilde{v}_k)(y)|^{p(x,y)}}{\vartheta_{k}^{p(x,y)}|x-y|^{N+sp(x,y)}} \,dy\,dx\\ \leqslant&   \int_{\mathbb{R}^{N}}\bigg|\frac{\phi  \tilde{u}_k}{\vartheta_{k}}\bigg|^{\overline{p}}\,dx+(1+ \varepsilon)  \int_{\mathbb{R}^{N}}\int_{\mathbb{R}^{N}}\frac{|\phi (x)|^{p(x,y)}| \tilde{u}_k(x)- \tilde{u}_k(y)|^{p(x,y)}}{\vartheta_{k}^{p(x,y)}|x-y|^{N+sp(x,y)}} \,dy\,dx \\&+ C(\varepsilon)  \int_{\mathbb{R}^{N}}\int_{\mathbb{R}^{N}}\frac{|\tilde{u}_k(x)|^{p(x,y)}| \phi (x)- \phi (y)|^{p(x,y)}}{\vartheta_{k}^{p(x,y)}|x-y|^{N+sp(x,y)}} \,dy\,dx+ \int_{\mathbb{R}^{N}}\bigg|\frac{\phi  \tilde{v}_k}{\vartheta_{k}}\bigg|^{\overline{p}}\,dx \\&+  (1+ \varepsilon)  \int_{\mathbb{R}^{N}}\int_{\mathbb{R}^{N}}\frac{|\phi (x)|^{p(x,y)}| \tilde{v}_k(x)- \tilde{v}_k(y)|^{p(x,y)}}{\vartheta_{k}^{p(x,y)}|x-y|^{N+sp(x,y)}} \,dy\,dx \\&+ C(\varepsilon)  \int_{\mathbb{R}^{N}}\int_{\mathbb{R}^{N}}\frac{|\tilde{v}_k(x)|^{p(x,y)}| \phi (x)- \phi (y)|^{p(x,y)}}{\vartheta_{k}^{p(x,y)}|x-y|^{N+sp(x,y)}} \,dy\,dx.
    \end{split}
    \end{equation}
We define, for all $k \in \mathbb{N}$
\begin{equation}\label{nota}
I_{k}:=  \int_{\mathbb{R}^{N}}\int_{\mathbb{R}^{N}}\frac{(|\tilde{u}_k(y)|^{p(x,y)}+|\tilde{v}_k(y)|^{p(x,y)})| \phi (x)- \phi (y)|^{p(x,y)}}{\vartheta_{k}^{p(x,y)}|x-y|^{N+sp(x,y)}} \,dy\,dx.
\end{equation}    
Then, from  Proposition \ref{norma} and \eqref{5.10a}, we deduce  that 
    \begin{equation}\label{nota0}
    \begin{split}
    1 \leqslant & \frac{(1+ \varepsilon)(\|\phi \|^{\overline{p}}_{\infty}+1)}{\min\{\vartheta_{k}^{\overline{p}}, \vartheta_{k}^{p^{-}}\}}\big(1+ \|\tilde{v}_{k}\|^{\overline{p}}\big)+\frac{(1+ \varepsilon)(\|\phi \|^{\overline{p}}_{\infty}+1)}{\min\{\vartheta_{k}^{\overline{p}}, \vartheta_{k}^{p^{-}}\}}\big(1+ \|\tilde{u}_{k}\|^{\overline{p}}\big) + C(\varepsilon)I_{k} \\ \leqslant&  2\frac{(1+ \varepsilon)(\|\phi \|^{\overline{p}}_{\infty})}{\min\{\vartheta_{k}^{\overline{p}}, \vartheta_{k}^{p^{-}}\}}(1+ \|(\tilde{v}_{k},\tilde{v}_{k})\|^{\overline{p}}) + C(\varepsilon)I_{k}.
    \end{split}
    \end{equation}
Moreover, from the symmetry of function $p$ we get
    \begin{equation*}
I_{k}:=  \int_{\mathbb{R}^{N}}\int_{\mathbb{R}^{N}}\frac{(|\tilde{u}_k(x)|^{p(x,y)}+|\tilde{v}_k(x)|^{p(x,y)})| \phi (x)- \phi (y)|^{p(x,y)}}{\vartheta_{k}^{p(x,y)}|x-y|^{N+sp(x,y)}} \,dy\,dx.
\end{equation*} 
Thus, once that $supp \phi  \subset B_{\mathcal{R}}$ and $\vartheta_{k}^{p(x,y)} \geqslant \min\{\vartheta_{k}^{\overline{p}}, \vartheta_{k}^{p^{-}}\}$ for all $x, y \in \mathbb{R}^{N},$ it follows that
    \begin{equation}\label{s2}
    \begin{split}
    I_{k}\leqslant& \frac{1}{\min\{\vartheta_{k}^{\overline{p}}, \vartheta_{k}^{p^{-}}\}}\bigg[ \int_{B^{c}_{\mathcal{R}}}(|\tilde{u}_k(x)|^{p_{\star}(x)}+|\tilde{u}_k(x)|^{\overline{p}})\int_{B_{\mathcal{R}}}\frac{| \phi (x)- \phi (y)|^{p(x,y)}}{\vartheta_{k}^{p(x,y)}|x-y|^{N+sp(x,y)}} \,dy\,dx\\&+  \int_{B_{\mathcal{R}}}(|\tilde{u}_k(x)|^{p_{\star}(x)}+|\tilde{u}_k(x)|^{\overline{p}})\int_{B^{c}_{\mathcal{R}}}\frac{| \phi (x)- \phi (y)|^{p(x,y)}}{\vartheta_{k}^{p(x,y)}|x-y|^{N+sp(x,y)}} \,dy\,dx\\&+  \int_{B_{\mathcal{R}}}(|\tilde{u}_k(x)|^{p_{\star}(x)}+|\tilde{u}_k(x)|^{\overline{p}})\int_{B_{\mathcal{R}}}\frac{| \phi (x)- \phi (y)|^{p(x,y)}}{\vartheta_{k}^{p(x,y)}|x-y|^{N+sp(x,y)}} \,dy\,dx\\ &+ \int_{B^{c}_{\mathcal{R}}}(|\tilde{v}_k(x)|^{p_{\star}(x)}+|\tilde{v}_k(x)|^{\overline{p}})\int_{B_{\mathcal{R}}}\frac{| \phi (x)- \phi (y)|^{p(x,y)}}{\vartheta_{k}^{p(x,y)}|x-y|^{N+sp(x,y)}} \,dy\,dx\\&+  \int_{B_{\mathcal{R}}}(|\tilde{v}_k(x)|^{p_{\star}(x)}+|\tilde{v}_k(x)|^{\overline{p}})\int_{B^{c}_{\mathcal{R}}}\frac{| \phi (x)- \phi (y)|^{p(x,y)}}{\vartheta_{k}^{p(x,y)}|x-y|^{N+sp(x,y)}} \,dy\,dx\\&+  \int_{B_{\mathcal{R}}}(|\tilde{v}_k(x)|^{p_{\star}(x)}+|\tilde{v}_k(x)|^{\overline{p}})\int_{B_{\mathcal{R}}}\frac{| \phi (x)- \phi (y)|^{p(x,y)}}{\vartheta_{k}^{p(x,y)}|x-y|^{N+sp(x,y)}} \,dy\,dx\bigg].
    \end{split}    \end{equation}
 
 \medskip  
\textbf{3.} Now, we estimate each integral in the right-hand side of \eqref{s2}.
Note that $\{(\tilde{u}_k, \tilde{v}_{k})\}_{k \in \mathbb{N}}$ is a bounded sequence in $\mathit{W},$ then   $\{\tilde{u}_k\}_{k \in \mathbb{N}}$ 
    and $\{\tilde{v}_{k}\}_{k \in \mathbb{N}}$ are bounded in $W^{s,p(\cdot, \cdot)}(\mathbb{R}^{N})$. Moreover, by \eqref{remark}, we have 
    \begin{equation}\label{5.3}
    \max_{t \in \{\overline{p}, p_{\star}\}}\sup_{k \in\mathbb{N}} \int_{\mathbb{R}^{N}}|\tilde{u}_{k}|^{t(x)}\, dx\leqslant c_{3},
    \end{equation}
     \begin{equation}\label{5.4}
    \max_{t \in \{\overline{p}, p_{\star}\}}\sup_{k \in\mathbb{N}} \int_{\mathbb{R}^{N}}|\tilde{v}_{k}|^{t(x)}\, dx\leqslant c_{4}.
    \end{equation}
  Let $t \in \{\overline{p}, p_{\star}\}.$ From \eqref{5.a}, \eqref{5.3}, and \eqref{5.4}, we obtain 
    \begin{equation}\label{5.5}
    \begin{split}
    &\int_{B^{c}_{\mathcal{R}}}(|\tilde{u}_k(x)|^{t(x)}+ |\tilde{v}_k(x)|^{t(x)})\bigg(\int_{B_{\mathcal{R}}}\frac{| \phi (x)- \phi (y)|^{p(x,y)}}{|x-y|^{N+sp(x,y)}} \,dy\bigg)\,dx \\= & \int_{B^{c}_{\mathcal{R}}}(|\tilde{u}_k(x)|^{t(x)}+|\tilde{v}_k(x)|^{t(x)})\bigg(\int_{supp \phi }\frac{| \phi (y)|^{p(x,y)}}{|x-y|^{N+sp(x,y)}} \,dy\bigg)\,dx\\\leqslant&(1+\|\phi \|^{\overline{p}}_{\infty})\int_{B^{c}_{\mathcal{R}}}\big(|\tilde{u}_k(x)|^{t(x)}+ |\tilde{v}_k(x)|^{t(x)} \big) \bigg(\int_{supp\phi }\frac{dy}{(\mathcal{R}/ 2)^{N+sp(x,y)}} \bigg)\,dx  \\ \leqslant &\frac{(1+\|\phi \|^{\overline{p}}_{\infty})}{(\mathcal{R}/ 2)^{N+sp^{-}}}|B_{\mathcal{R}}|
  \bigg[  \int_{\mathbb{R}^{N}}(|\tilde{u}_k(x)|^{t(x)} + |\tilde{v}_k(x)|^{t(x)})\,dx\bigg]  \leqslant  \frac{c_{5}}{\mathcal{R}^{sp^{-}}}.
    \end{split}
    \end{equation}
On the other hand, $\mbox{ for all } x \in B_{\mathcal{R}}$, from \eqref{5.a}, 
    \begin{equation}\label{5.6}
    \begin{split}
    \int_{B^{c}_{\mathcal{R}}}\frac{| \phi (x)- \phi (y)|^{p(x,y)}}{|x-y|^{N+sp(x,y)}} \,dy=& \int_{B^{c}_{\mathcal{R}}}\frac{| \phi (x)|^{p(x,y)}}{|x-y|^{N+sp(x,y)}} \,dy \\ \leqslant & (1+ \|\phi\|^{\overline{p}}_{\infty})\int_{\{|z|\geqslant 1\}}\frac{ \, dz}{|z|^{N+sp^{-}}} \\= & \big(1+ \|\phi\|^{\overline{p}}_{\infty}\big)\frac{ N|B_{1}|}{sp^{-}}:=c_6 .
    \end{split}
\end{equation} 
Using \eqref{5.6}, we obtain
   \begin{equation}\label{5.7}
    \begin{split}
    \int_{B_{\mathcal{R}}}(|\tilde{u}_k(x)|^{t(x)}+ |\tilde{v}_k(x)|^{t(x)})\bigg(\int_{B_{\mathcal{R}}^{c}}\frac{| \phi (x)- \phi (y)|^{p(x,y)}}{|x-y|^{N+sp(x,y)}} \,dy\bigg)\,dx \\ \leqslant c_{6}\int_{B_{\mathcal{R}}}(|\tilde{u}_k(x)|^{t(x)}+|\tilde{v}_k(x)|^{t(x)})\,dx.
    \end{split}
\end{equation}    
   Now, to estimate the last integral in the right-hand side of \eqref{s2} we notice that for  $x \in B_{\mathcal{R}}$
   \begin{equation*}\label{5.8}
   \begin{split}
   \int_{B_{\mathcal{R}}}\frac{| \phi (x)- \phi (y)|^{p(x,y)}}{|x-y|^{N+sp(x,y)}} \,dy &\leqslant  (1+\|\nabla\phi\|^{\overline{p}}_{\infty})\int_{B_{\mathcal{R}}}\frac{ dy}{|x-y|^{N+(s-1)p(x,y)}} \\  &\leqslant (1+\|\nabla\phi\|^{\overline{p}}_{\infty})\int_{B_{\mathcal{R}}}\bigg( 1 + \frac{ 1}{|x-y|^{N+(s-1)p^{-}}}\bigg)\,dy\\ 
   & \leqslant  (1+\|\nabla\phi\|^{\overline{p}}_{\infty})\bigg[ |B_{\mathcal{R}}|+\int_{B_{2\mathcal{R}}} \frac{ dz}{|z|^{N+(s-1)p^{-}}}\bigg]  \\&= (1+\|\nabla\phi\|^{\overline{p}}_{\infty})\bigg[ |B_{\mathcal{R}}|+ \frac{N|B_{1}|(2\mathcal{R})^{(1-s)p^{-}}}{(1-s)p^{-}}\bigg]=:c_{1}(\mathcal{R}). \end{split}
   \end{equation*}
Hence we have
    \begin{equation}\label{5.9}
    \begin{split}
    \int_{B_{\mathcal{R}}}(|\tilde{u}_k(x)|^{t(x)}+ |\tilde{v}_k(x)|^{t(x)})\bigg(\int_{B_{\mathcal{R}}}\frac{| \phi (x)- \phi (y)|^{p(x,y)}}{|x-y|^{N+sp(x,y)}} \,dy\bigg)\,dx \\ \leqslant  c_{1}(\mathcal{R})\int_{B_{\mathcal{R}}}(|\tilde{u}_k(x)|^{t(x)}+ |\tilde{v}_k(x)|^{t(x)})\,dx.
    \end{split}
    \end{equation}
    Using \eqref{5.5}, \eqref{5.7}, \eqref{5.9} in \eqref{s2}, we obtain
    \begin{equation}\label{5.11}
    \begin{split}
   I_{k}\leqslant& \frac{1}{\min\{\vartheta_{k}^{\overline{p}}, \vartheta_{k}^{p^{-}}\}}\bigg[\frac{c_{7}}{\mathcal{R}^{sp^{-}}}+c_{8}c_{2}(\mathcal{R})\sum_{t \in \{ \overline{p}, p_{\star}\}} \int_{B_{\mathcal{R}}}(|\tilde{u}_k(x)|^{t(x)}+ |\tilde{v}_k(x)|^{t(x)})\,dx\bigg] \\ & \leqslant \frac{c_{9}}{\min\{\vartheta_{k}^{\overline{p}}, \vartheta_{k}^{p^{-}}\}}\bigg[\frac{1}{\mathcal{R}^{sp^{-}}}+c_{2}(\mathcal{R})\sum_{t \in \{ \overline{p}, p_{\star}\}} \int_{B_{\mathcal{R}}}(|\tilde{u}_k(x)|^{t(x)}+ |\tilde{v}_k(x)|^{t(x)})\,dx\bigg].
    \end{split}
    \end{equation}
Then, from \eqref{5.3}, \eqref{5.4}, \eqref{5.11}, and since the sequence $\{(\tilde{u}_{k}, \tilde{v}_{k})\}_{k\in \mathbb{N}}$ is bounded in $\mathit{W}$, using \eqref{nota0}, we can conclude 
$$1\leqslant \frac{c_{3}(\mathcal{R})}{\min\{\vartheta_{k}^{\overline{p}}, \vartheta_{k}^{p^{-}}\}},$$    
   and hence, for all $k \in \mathbb{N}$,   
    $$\vartheta_{k}\leqslant c_{4}(\mathcal{R}). $$
      Therefore, the sequence $\{\vartheta_{k}\}_{k \in \mathbb{N}}$ is bounded in $\mathbb{R}$. Hence, from Bolzano-Weierstrass Theorem, up to a subsequence, we may assume that there exists $\vartheta_{\star}\in [0, \infty)$ such that
    \begin{equation}\label{5.12d}
    \lim_{k\to+ \infty}\vartheta_{k}=\vartheta_{\star}.
    \end{equation}
    Suppose that $\vartheta_{\star}>0.$ Thus from \eqref{5.1} and \eqref{5.11}, we obtain
    \begin{equation}\label{5.13}
    \begin{split}
    1\leqslant& (1+ \varepsilon)\int_{\mathbb{R}^{N}}\bigg(\bigg|\frac{\phi(x)  }{\vartheta_{k}}\bigg|^{p_{\star}(x)}+\bigg|\frac{\phi(x)  }{\vartheta_{k}}\bigg|^{\overline{p}}\bigg)\bigg( | \tilde{u}_k(x)|^{\overline{p}}+ \int_{\mathbb{R}^{N} }\frac{|(  \tilde{u}_k(x)-  \tilde{u}_k(y)|^{p(x,y)}}{|x-y|^{N+sp(x,y)}} \,dy\bigg)\,dx\\& + (1+ \varepsilon)\int_{\mathbb{R}^{N}}\bigg(\bigg|\frac{\phi(x)  }{\vartheta_{k}}\bigg|^{p_{\star}(x)}+\bigg|\frac{\phi(x)  }{\vartheta_{k}}\bigg|^{\overline{p}}\bigg)\bigg( | \tilde{v}_k(x)|^{\overline{p}}+ \int_{\mathbb{R}^{N} }\frac{|(  \tilde{v}_k(x)-  \tilde{v}_k(y)|^{p(x,y)}}{|x-y|^{N+sp(x,y)}} \,dy\bigg)\,dx \\&+ \frac{c_{9}}{{\min\{\vartheta_{k}^{\overline{p}}, \vartheta_{k}^{p^{-}}}\}}\bigg[\frac{1}{\mathcal{R}^{sp^{-}}}+c_{2}(\mathcal{R})\sum_{t \in \{\overline{p}, p_{\star}\}}\int_{B_{\mathcal{R}}}(|\tilde{u}_k(x)|^{t(x)}+|\tilde{v}_k(x)|^{t(x)})\,dx\bigg].
    \end{split}
    \end{equation}
Taking $k\to+\infty$ in the inequality \eqref{5.13} and using \eqref{5.12}, \eqref{5.ao} and the fact that
$$\lim_{k\to+\infty}\int_{B_{\mathcal{R}}}|\tilde{u}_k(x)|^{t(x)}\,dx=\lim_{k\to+\infty}\int_{B_{\mathcal{R}}}|\tilde{v}_k(x)|^{t(x)}\,dx=0, $$    
  for all $ t \in \{\overline{p}, p_{\star}\}$  (see \eqref{converr}),  we obtain
    \begin{equation*}\label{5.55}
1\leqslant (1+ \varepsilon)\int_{\mathbb{R}^{N}}\bigg(\bigg|\frac{\phi(x)  }{\vartheta_{\star}}\bigg|^{p_{\star}(x)}+\bigg|\frac{\phi(x)  }{\vartheta_{\star}}\bigg|^{\overline{p}}\bigg)\,d \overline{\mu}+ \frac{c_9}{{\min\{\vartheta^{\overline{p}}_{\star}, \vartheta^{p^{-}}_{\star}}\}\mathcal{R}^{sp^{-}}}.
    \end{equation*}
Taking $\mathcal{R}\to+\infty$ and then letting $\varepsilon \to 0^{+}$, we conclude from the last inequality that
      \begin{equation*}\label{5.55}
1\leqslant \int_{\mathbb{R}^{N}}\bigg(\bigg|\frac{\phi(x)  }{\vartheta_{\star}}\bigg|^{p_{\star}(x)}+\bigg|\frac{\phi(x)  }{\vartheta_{\star}}\bigg|^{\overline{p}}\bigg)\,d \overline{\mu}.
    \end{equation*}
   \noindent Therefore by Proposition \ref{masmenos}, we obtain the following   estimate
    \begin{equation*}\label{5.12}
    \vartheta_{\star}\leqslant 2 ^{\frac{1}{p^{-}}}\max\bigg\{\|\phi\|_{L^{p_{\star}(\cdot)}_{\overline{\mu}}(\mathbb{R}^{N})}, \|\phi\|_{L^{\overline{p}}_{\overline{\mu}}(\mathbb{R}^{N})}\bigg\}.
    \end{equation*}
     \noindent Hence, from \eqref{5.0}, \eqref{5.a1} and \eqref{5.12d}, we see that
    \begin{equation}\label{5.14}
    \mathcal{S}_{\alpha\beta}\bigg\|\phi  ^{\frac{\alpha(\cdot)}{q(\cdot)}}\phi  ^{\frac{\beta(\cdot)}{q(\cdot)}}\bigg\|_{L^{q(\cdot)}_{\overline{\nu}}(\mathbb{R}^{N})}=\mathcal{S}_{\alpha\beta}\|\phi\|_{L^{q(\cdot)}_{\overline{\nu}}(\mathbb{R}^{N})} \leqslant 2 ^{\frac{1}{p^{-}}}\max\bigg\{\|\phi\|_{L^{p_{\star}(\cdot)}_{\overline{\mu}}(\mathbb{R}^{N})}, \|\phi\|_{L^{\overline{p}}_{\overline{\mu}}(\mathbb{R}^{N})}\bigg\}.
        \end{equation}
If $\vartheta_{\star}=0$, by     \eqref{5.0} and \eqref{5.a1}, we conclude 
    \begin{equation}\label{5.140}
    \|\phi\|_{L^{q(\cdot)}_{\overline{\nu}}(\mathbb{R}^{N})} =0.
    \end{equation}
 Hence \eqref{5.140} holds for any $\phi \in C^{\infty}_{c}(\mathbb{R}^{N})$ and thus \eqref{cc1} follows using the \cite[Lemma 3.2]{bonder2} and
the definition of $\overline{\nu}$ (see \eqref{5.0}).      The fact that $\{x_{i}\}_{i \in I} \subset \mathscr{C}$ can be obtained by an argument similar to that of   \cite[Theorem 3.3]{kihosim} and we omit the proof.

\medskip   
\textbf{4.}   Now, we verify \eqref{cc3}. Let $i \in I$ and for $\rho>0$, define $\Psi_{\rho}$ as in Lemma \ref{4.4a} with $x_{0}$ replaced by $x_i$. Then   $\Psi_{\rho}\in C^{\infty}(\mathbb{R}^{N})$, $0 \leqslant \Psi_{\rho} \leqslant 1$, $\Psi_{\rho}\equiv 1$ on $B_{\rho}(x_i)$, and $supp(\Psi_{\rho})\subset B_{2\rho}(x_i).$
From  \eqref{SS}, we infer that
    \begin{equation}\label{5.56}
    \mathcal{S}_{\alpha\beta}\bigg\|(\Psi_{\rho}u_k)^{\frac{\alpha(\cdot)}{q(\cdot)}}(\Psi_{\rho}v_k)^{\frac{\beta(\cdot)}{q(\cdot)}}\bigg\|_{L^{q(\cdot)}(\mathbb{R}^{N})}  \leqslant \|(\Psi_{\rho}u_k, \Psi_{\rho}v_k)\|.
    \end{equation}
     Taking the limit inferior as $k \to +\infty$ in \eqref{5.56}, using  fact $ |u_k|^{\alpha(x)}|v_k|^{\beta(x)} \overset{\ast}{\rightharpoonup} \nu \mbox{ in }  \mathcal{M}(\mathbb{R}^{N})$ (see \eqref{c3}) and the definition of  $\|\cdot\|_{L^{q(\cdot)}_{\nu}(\mathbb{R}^{N})}$ (see section \ref{sec2}), we obtain
     \begin{equation*}
 \mathcal{S}_{\alpha\beta}\bigg\|\Psi_{\rho}^{\frac{\alpha(\cdot)}{q(\cdot)}}\Psi_{\rho}^{\frac{\beta(\cdot)}{q(\cdot)}}\bigg\|_{L^{q(\cdot)}_{\nu}(B_{2\rho}(x_i))}=   \mathcal{S}_{\alpha\beta}\|\Psi_{\rho}\|_{L^{q(\cdot)}_{\nu}(B_{2\rho}(x_i))}\leqslant \liminf_{k\to +\infty}\|(\Psi_{\rho}u_k, \Psi_{\rho}v_k)\|.
    \end{equation*}
    Hence,
     \begin{equation}\label{5.15}
    \mathcal{S}_{\alpha\beta}\limsup_{\rho\to 0^{+}}\|\Psi_{\rho}\|_{L^{q(\cdot)}_{\nu}(B_{2\rho}(x_i))}\leqslant \limsup_{\rho\to 0^{+}}\liminf_{k\to +\infty}\|(\Psi_{\rho}u_k, \Psi_{\rho}v_k)\|.
    \end{equation}
    On the other hand, we have
   \begin{equation}\label{5.16}
   \begin{split}
    \|\Psi_{\rho}\|_{L^{q(\cdot)}_{\nu}(B_{2\rho}(x_i))}\geqslant& \min\bigg\{ \bigg( \int_{B_{2\rho}(x_i)}|\Psi_{\rho}|^{q(x)}\,d\nu\bigg)^{\frac{1}{q^{+}_{i,\rho}}}, \bigg( \int_{B_{2\rho}(x_i)}|\Psi_{\rho}|^{q(x)}\,d\nu\bigg)^{\frac{1}{q^{-}_{i,\rho}}} \bigg\} \\ \geqslant& \min\bigg\{\nu \big(B_{2\rho}(x_i)\big)^{\frac{1}{q^{+}_{i,\rho}}}, \nu\big( B_{2\rho}(x_i)\big)^{\frac{1}{q^{-}_{i,\rho}}} \bigg\}, 
    \end{split}
    \end{equation} 
 where $\displaystyle{q^{+}_{i,\rho}= \max_{x\in \overline{B_{\rho}(x_i)}}q(x)}$ and $\displaystyle{q^{-}_{i,\rho}= \min_{x\in \overline{B_{\rho}(x_i)}}q(x)}.$  Thus, by continuity of $q$ and the fact that $x_i \in \mathscr{C} $, we obtain an upper of the  right hand side of \eqref{5.15}
 \begin{equation}\label{5.160}
    \limsup_{\rho\to0^{+}}\|\Psi_{\rho}\|_{L^{q(\cdot)}_{\nu}(B_{2\rho}(x_i))}\geqslant \nu_{i}^{\frac{1}{q(x_i)}}=\nu_{i}^{\frac{1}{\overline{p}^{\star}_{s}}}.
        \end{equation}     
     Now to obtain an upper bound of the right hand side \eqref{5.15}, we prove that there are $ \rho_{0} \in(0, 1)$ and $\vartheta_{0}\in(0, +\infty)$, such that
     \begin{equation}\label{5.17}
     0 < \frac{\mathcal{S}_{\alpha\beta}}{2}\nu_{i}^{\frac{1}{q(x_i)}} \leqslant \liminf_{k\to+\infty}\vartheta_{k,\rho}:=\vartheta_{\star,\rho} \leqslant \vartheta_{0} 
     \end{equation}
    $\mbox{ for any } \rho \in (0, \rho_{0})$, where $\vartheta_{k,\rho}= \|(\Psi_{\rho}u_k, \Psi_{\rho}v_k )\|$.
    Indeed, by the continuity of $q $ and the positiveness of $\nu_{i}$, we can choose $\rho \in(0,1)$ such that    
   \begin{equation}\label{5.18}
   \begin{split}
     \mathcal{S}_{\alpha\beta}  \min\bigg\{ \nu_{i}^{\frac{1}{q^{+}_{i,\rho}}}, \nu_{i}^{\frac{1}{q^{-}_{i,\rho}}} \bigg\}> \frac{ \mathcal{S}_{\alpha\beta}}{2}\nu_{i}^{\frac{1}{q(x_i)}} 
    \end{split}
    \end{equation}    
    $\mbox{ for all } \rho \in(0, \rho_{0}).$
 
 Using \eqref{5.14}, \eqref{5.16}, and \eqref{5.18}, we obtain, $ \mbox{ for all } \rho \in (0, \rho_{0})$, that
  \begin{equation*}
  \frac{\mathcal{S}_{\alpha\beta}\nu_{i}^{\frac{1}{q(x_i)}}}{2} \leqslant \vartheta_{\star \rho}.
\end{equation*}    
 On the other hand, by Lemma \ref{4.4a} we can choose $\rho_{0}$ smaller if necessary such that,  $\mbox{ for all } \rho \in (0, \rho_{0})$, we conclude
    \begin{equation}\label{5.19}
    \limsup_{k\to+\infty} \int_{\mathbb{R}^{N}}\int_{\mathbb{R}^{N}}\big(|u_{k}(y)|^{p(x,y)}+|v_{k}(y)|^{p(x,y)}\big)\frac{|\Psi_{\rho}(x)-\Psi_{\rho}(y)|^{p(x,y)}}{|x-y|^{N+sp(x,y)}}\,dy\,dx<1. 
    \end{equation}   
Now, note that 
$$\bigg[\int_{\mathbb{R}^{N}}\big(|\Psi_{\rho}u_k|^{\overline{p}}  +|\Psi_{\rho}v_k|^{\overline{p}}\big)\,dx $$
   \begin{equation}\label{5.20}
   \begin{split}
        & +  \int_{\mathbb{R}^{N}}\int_{\mathbb{R}^{N}}\bigg(\frac{|(\Psi_{\rho}u_k)(x)-(\Psi_{\rho}u_k)(y)|^{p(x,y)}}{|x-y|^{N+sp(x,y)}}+\frac{|(\Psi_{\rho}v_k)(x)-(\Psi_{\rho}v_k)(y)|^{p(x,y)}}{|x-y|^{N+sp(x,y)}}\bigg) \,dy\,dx \bigg]\\   \leqslant &   \int_{\mathbb{R}^{N}}\big( |\Psi_{\rho}u_k|^{\overline{p}}+|\Psi_{\rho}v_k|^{\overline{p}}\big)\,dx \\  & +  2^{\overline{p}-1}  \int_{\mathbb{R}^{N}}\int_{\mathbb{R}^{N}}|\Psi_{\rho}(x)|^{p(x,y)}\bigg(\frac{|u_k(x)-u_k(y)|^{p(x,y)}}{|x-y|^{N+sp(x,y)}}+ \frac{|v_k(x)-v_k(y)|^{p(x,y)}}{|x-y|^{N+sp(x,y)}}\bigg)\,dy\,dx  \\ & + 2^{\overline{p}-1}  \int_{\mathbb{R}^{N}}\int_{\mathbb{R}^{N}}\big(|u_k(y)|^{p(x,y)}+|v_k(y)|^{p(x,y)}\big)\frac{|\Psi_{\rho}(x)-\Psi_{\rho}(y)|^{p(x,y)}}{|x-y|^{N+sp(x,y)}}\,dy\,dx.
   \end{split}
\end{equation}     
Thus, from \eqref{5.19}, \eqref{5.20}, and since the sequence $\{(u_k, v_k)\}_{k \in \mathbb{N}}$ is bounded in $\mathit{W}$, we conclude that,  there exists $\vartheta_{0} \in(0, +\infty)$ such that $\vartheta_{k,\rho}<\vartheta_{0}$ for all $k \in \mathbb{N}$ and $\rho \in (0, \rho_{0})$. Therefore, \eqref{5.17} is proved.

 Next, let $\varepsilon>0$ be arbitrary and fixed. We have,
 \begin{equation}\label{5.1}
    \begin{split}
    1=& \int_{\mathbb{R}^{N}}\bigg( \bigg|\frac{\Psi_{\rho}  \tilde{u}_k}{\vartheta_{k}}\bigg|^{\overline{p}}+ \bigg|\frac{\Psi_{\rho}  \tilde{v}_k}{\vartheta_{k}}\bigg|^{\overline{p}}\bigg)\,dx+   \int_{\mathbb{R}^{N}}\int_{\mathbb{R}^{N}}\frac{|(\Psi_{\rho}  \tilde{u}_k)(x)-(\Psi_{\rho} \tilde{u}_k)(y)|^{p(x,y)}}{\vartheta_{k}^{p(x,y)}|x-y|^{N+sp(x,y)}} \,dy\,dx \\= & \int_{\mathbb{R}^{N}}\bigg( \bigg|\frac{\Psi_{\rho}  \tilde{u}_k}{\vartheta_{k}}\bigg|^{\overline{p}}+ \bigg|\frac{\Psi_{\rho}  \tilde{v}_k}{\vartheta_{k}}\bigg|^{\overline{p}}\bigg)\,dx + 2 \int_{B_{2\rho(x_i)}}\int_{ \mathbb{R}^{N}\setminus B_{2\rho}(x_i) }\frac{|(\Psi_{\rho} \tilde{u}_k)(x)-(\Psi_{\rho}  \tilde{u}_k)(y)|^{p(x,y)}}{\vartheta_{k}^{p(x,y)}|x-y|^{N+sp(x,y)}} \,dy\,dx\\&+ \int_{B_{2\rho(x_i)}}\int_{ B_{2\rho}(x_i) }\frac{|(\Psi_{\rho}  \tilde{u}_k)(x)-(\Psi_{\rho} \tilde{u}_k)(y)|^{p(x,y)}}{\vartheta_{k}^{p(x,y)}|x-y|^{N+sp(x,y)}} \,dy\,dx.
    \end{split}
    \end{equation}
Hence, by utilizing \eqref{desiab} again we have 
      \begin{equation*}\label{5.1}
    \begin{split}
    1\leqslant & \int_{\mathbb{R}^{N}}\bigg(\bigg|\frac{\Psi_{\rho} \tilde{u}_k}{\vartheta_{k}}\bigg|^{\overline{p}}+ \bigg|\frac{\Psi_{\rho}  \tilde{v}_k}{\vartheta_{k}}\bigg|^{\overline{p}}\bigg)\,dx  \\ &+ 2\int_{B_{2\rho(x_i)}}\int_{ \mathbb{R}^{N}\setminus B_{2\rho}(x_i) }\bigg(\bigg|\frac{  \tilde{u}_k(x)}{\vartheta_{k}}\bigg|^{p(x,y)}+ \bigg|\frac{\tilde{v}_k(y)}{\vartheta_{k}}\bigg|^{p(x,y)}\bigg)\frac{|\Psi_{\rho}(x)-\Psi_{\rho} (y)|^{p(x,y)}}{|x-y|^{N+sp(x,y)}} \,dy\,dx\\ &+ (1+\varepsilon) \int_{B_{2\rho(x_i)}}\int_{  B_{2\rho}(x_i) }|\Psi_{\rho}(x)|^{p(x,y)}\frac{|\tilde{u}_k(x)-  \tilde{u}_k(y)|^{p(x,y)}+|\tilde{v}_k(x)-  \tilde{v}_k(y)|^{p(x,y)}}{\vartheta_{k}^{p(x,y)}|x-y|^{N+sp(x,y)}} \,dy\,dx\ \\&+ C(\varepsilon) \int_{B_{2\rho(x_i)}}\int_{  B_{2\rho}(x_i) }\bigg(\bigg|\frac{  \tilde{u}_k(y)}{\vartheta_{k}}\bigg|^{p(x,y)}+ \bigg|\frac{\tilde{v}_k(y)}{\vartheta_{k}}\bigg|^{p(x,y)}\bigg)\frac{|\Psi_{\rho}(x)-\Psi_{\rho} (y)|^{p(x,y)}}{|x-y|^{N+sp(x,y)}} \,dy\,dx.
    \end{split}
    \end{equation*}
Combining this with the fact that $ 0 \leqslant\Psi_{\rho} \leqslant 1,$ we obtain
  \begin{equation}\label{5.21d}
  \begin{split}
  1\leqslant&\frac{C(\varepsilon)}{\min\{\vartheta_{k}^{\overline{p}}, \vartheta_{k}^{p^{-}_{i}}\}}\int_{ \mathbb{R}^{N}}\int_{ \mathbb{R}^{N}}\bigg(| \tilde{u}_k(x)|^{p(x,y)}+ |\tilde{v}_k(x)|^{p(x,y)}\bigg)\frac{|\Psi_{\rho}(x)-\Psi_{\rho} (y)|^{p(x,y)}}{|x-y|^{N+sp(x,y)}} \,dy\,dx\\ & + \frac{(1+\varepsilon)}{\min\{\vartheta_{k}^{\overline{p}}, \vartheta_{k}^{p^{-}_{i}}\}}\int_{ \mathbb{R}^{N}}\Psi_{\rho}\big(U_k(x)+V_k(x)\big)\,dx,
  \end{split}
\end{equation}     
 where     $\displaystyle{p^{-}_{i}=\min_{(x,y)\in B_{2\rho}(x_i)\times B_{2\rho}(x_i)}p(x,y)}$. 
 \noindent For simplicity, for all $k \in \mathbb{N}$,  and $x \in \mathbb{R}^{N}$, we denote
 
 \begin{equation}\label{u1}
     U_k(x)= |u_k(x)|^{{\overline{p}}}+\int_{\mathbb{R}^{N}}\frac{|u_k(x)-u_k(y)|^{p(x,y)}}{|x-y|^{N|+sp(x,y)}}\,dy
 \end{equation}
 and
\begin{equation}\label{v1}
V_k(x)= |v_k(x)|^{{\overline{p}}}+\int_{\mathbb{R}^{N}}\frac{|v_k(x)-v_k(y)|^{p(x,y)}}{|x-y|^{N|+sp(x,y)}}\,dy.
  \end{equation}

 \noindent Hence, using \eqref{5.17} and \eqref{5.21d}, we deduce by \eqref{5.21} that for all $ \rho \in(0, \rho_{0})$ 
      \begin{equation*}\label{5.21}
  \begin{split}
  1\leqslant&\frac{C(\varepsilon)}{\min\{\vartheta_{\star,\rho}^{\overline{p}}, \vartheta_{\star,\rho}^{p^{-}_{i}}\}}\limsup_{k\to+\infty}\int_{ \mathbb{R}^{N}}\int_{ \mathbb{R}^{N}}\bigg(| \tilde{u}_k(x)|^{p(x,y)}+ |\tilde{v}_k(x)|^{p(x,y)}\bigg)\frac{|\Psi_{\rho}(x)-\Psi_{\rho} (y)|^{p(x,y)}}{|x-y|^{N+sp(x,y)}} \,dy\,dx\\ & + \frac{(1+\varepsilon)}{\min\{\vartheta_{\star, \rho}^{\overline{p}}, \vartheta_{\star, \rho}^{p^{-}_{i}}\}}\int_{ \mathbb{R}^{N}}\Psi_{\rho}d\,\mu.
 \end{split}
\end{equation*}
    Thus, for all $ \rho \in(0, \rho_{0})$, we obtain 
    \begin{equation*}
    \begin{split}
    \min\{\vartheta_{\star, \rho}^{\overline{p}}, \vartheta_{\star, \rho}^{p^{-}_{i}}\}\leqslant& C(\varepsilon)\limsup_{k\to+\infty}\int_{ \mathbb{R}^{N}}\int_{ \mathbb{R}^{N}}\bigg(| \tilde{u}_k(x)|^{p(x,y)}+ |\tilde{v}_k(x)|^{p(x,y)}\bigg)\frac{|\Psi_{\rho}(x)-\Psi_{\rho} (y)|^{p(x,y)}}{|x-y|^{N+sp(x,y)}} \,dy\,dx \\&+ (1+\varepsilon)\int_{ \mathbb{R}^{N}}\Psi_{\rho}d\,\mu.
     \end{split}
    \end{equation*}
 
     \noindent Taking limit superior as $\rho\to 0^{+}$ in the last inequality and invoking Lemma \ref{4.4a}, we obtain
     \begin{equation*}
     \vartheta_{\star}\leqslant(1+\varepsilon)^{\frac{1}{\overline{p}}}(\mu_i)^{\frac{1}{\overline{p}}}
     \end{equation*}
     where $\displaystyle{\vartheta_{\star}}:= \limsup_{\rho\to 0^{+}}\vartheta_{k}$ and $\displaystyle{\mu_{i}= \displaystyle{\lim_{\rho\to 0^{+}}\mu(B_{2\rho}(x_i))}}$.
      Then from \eqref{5.15} and \eqref{5.160}, we obtain

\begin{equation}\label{casa}
\mathcal{S}_{\alpha\beta}\nu_{i}^{\frac{      
1 }{\overline{p}^{\star}_{s}}}\leqslant (1+\varepsilon)^{\frac{1}{\overline{p}}}\mu_{i}^{\frac{1}{\overline{p}}}.
\end{equation}      
Since $\varepsilon$ was chosen arbitrarily using \eqref{casa}, we conclude \eqref{cc3}. Hence, $\{x_{i}\}_{i \in I} \subset \mathscr{C}$
 are also atoms of $\mu$. 
 %  atom is a measurable set which has positive measure and contains no set of smaller positive measure

\medskip 
\textbf{5.} Finally, to obtain \eqref{cc2} we note that, for each $\phi \in C_{0}(\mathbb{R}^{N})$, $\phi\geqslant 0$, the functional        
$$(u,v) \longmapsto \int_{\mathbb{R}^{N}}\phi\bigg[\big( |u|^{\overline{p}}+|v|^{\overline{p}}\big)+\int_{\mathbb{R}^{N}}\bigg(\frac{|u(x)-u(y)|^{p(x,y)}}{|x-y|^{N+sp(x,y)}}+\frac{|v(x)-v(y)|^{p(x,y)}}{|x-y|^{N+sp(x,y)}}\bigg)\,dy\bigg]\,dx$$  
 is convex and differentiable in $\mathit{W}$, from this, \eqref{c1} and \eqref{c2}, we obtain that
 \begin{equation*}
 \begin{split}
 &\int_{\mathbb{R}^{N}}\phi\bigg[\big( |u|^{\overline{p}}+|v|^{\overline{p}}\big)+\int_{\mathbb{R}^{N}}\bigg(\frac{|u(x)-u(y)|^{p(x,y)}}{|x-y|^{N+sp(x,y)}}+\frac{|v(x)-v(y)|^{p(x,y)}}{|x-y|^{N+sp(x,y)}}\bigg)\,dy\bigg]\,dx \\ &\leqslant \liminf_{k\to +\infty}\int_{\mathbb{R}^{N}}\phi\bigg[\big( |u_k|^{\overline{p}}+|v_k|^{\overline{p}}\big)+\int_{\mathbb{R}^{N}}\bigg(\frac{|u_k(x)-u_k(y)|^{p(x,y)}}{|x-y|^{N+sp(x,y)}}+\frac{|v_k(x)-v_k(y)|^{p(x,y)}}{|x-y|^{N+sp(x,y)}}\bigg)\,dy\bigg]\,dx \\ & = \int_{\mathbb{R}^{N}}\phi\,d\mu.
 \end{split}
\end{equation*}
  Therefore, 
  $$\mu \geqslant  |u|^{\overline{p}}+|v|^{\overline{p}}+\int_{\mathbb{R}^{N}}\bigg(\frac{|u(x)-u(y)|^{p(x,y)}}{|x-y|^{N+sp(x,y)}}+\frac{|v(x)-v(y)|^{p(x,y)}}{|x-y|^{N+sp(x,y)}}\bigg)\,dy.$$
  Extracting $\mu$ to its atoms ($\{x_{i}\}_{i \in I} \subset \mathscr{C}$) we get \eqref{cc2} which completes the proof.
         \end{proof}
         
%%%%%%%%%%%%%%%%%%%%%%%%%%%%%%%%%%%%%%%%%%%%%%%
\subsection{Proof of Theorem \ref{lionsinfity}}
%%%%%%%%%%%%%%%%%%%%%%%%%%%%%%%%%%%%%%%%%%%%%%%

 Let us now turn to the proof  Theorem \ref{lionsinfity},  that is  Concentration-compactness principle at infinity. The strategy is similar to the one used before but instead of taking a function that concentrates at the finite points $x_j$, we consider
the function $\Phi_{R}$ of Lemma \ref{4.5inf} which is a function that in some sense “concentrates at infinity".
 \begin{proof}
1. Let $\Phi_{R}$ be defined as Lemma \ref{4.5inf}. We can write

    \begin{equation}\label{inf1}
 \begin{split}
 \int_{\mathbb{R}^{N}}\big[ U_k(x)+V_k(x)\big]\,dx  = & \int_{\mathbb{R}^{N}}\Phi_{R}(x)\big[ U_k(x)+V_k(x)\big]\,dx \\ & +\int_{\mathbb{R}^{N}}(1-\Phi_{R}(x))\big[ U_k(x)+V_k(x)\big]\,dx
 \end{split}
\end{equation}
where  $U_k(x)$ adn $V_k(x)$ is given \eqref{u1} and \eqref{v1}. By \eqref{ccinfi} and fact that 
    \begin{equation}\label{inf2}
 \begin{split}
 &\int_{B^{c}_{2R}}\big[ U_k(x)+V_k(x)\big]\,dx\leqslant  \int_{\mathbb{R}^{N}}\Phi_{R}(x)\big[ U_k(x)+V_k(x)\big]\,dx \leqslant \int_{B^{c}_{2R}}\big[ U_k(x)+V_k(x)\big]\,dx,
 \end{split}
\end{equation}
for all $k \in \mathbb{N}$ and $R>0$, we obatin
\begin{equation}\label{inf3}
 \begin{split}
 \mu_{\infty}= \lim_{R\to \infty}\lim_{k\to \infty}  \int_{\mathbb{R}^{N}}\Phi_{R}(x)\big[ U_k(x)+V_k(x)\big]\,dx .
 \end{split}
\end{equation}
On the other hand, the facy that  $(1-\Phi_{R}) \in C^{\infty}_{c}(\mathbb{R}^{N})$ gives
\begin{equation}\label{i75}
 \begin{split}
 \lim_{R\to \infty}\lim_{k\to \infty}  \int_{\mathbb{R}^{N}}(1-\Phi_{R}(x))\big[ U_k(x)+V_k(x)\big]\,dx = \lim_{R\to \infty}  \int_{\mathbb{R}^{N}}(1-\Phi_{R}(x))\,d\mu.
 \end{split}
\end{equation}
Indeed, 
\begin{equation}\label{i76}
 \begin{split}
 \lim_{R\to \infty}  \int_{\mathbb{R}^{N}}(1-\Phi_{R}(x))\,d\mu =\mu(\mathbb{R}^{N}).
 \end{split}
\end{equation}
in views the Dominated Convergence Theorem. Thus, by \eqref{i75} and \eqref{i76}, we obtain 
\begin{equation}\label{inf4}
 \begin{split}
 \lim_{R\to \infty}\lim_{k\to \infty}  \int_{\mathbb{R}^{N}}(1-\Phi_{R}(x))\big[ U_k(x)+V_k(x)\big]\,dx  =\mu(\mathbb{R}^{N}).
 \end{split}
\end{equation}
Therefore using \eqref{inf3} and \eqref{inf4} in \eqref{inf1}, we obtain \eqref{cc2infi}.

\medskip
\textbf{2.} In order to prove  \eqref{cc2inf}, we decompose
    \begin{equation}\label{inf4.a}
 \begin{split}
 \int_{\mathbb{R}^{N}}|u_k(x)|^{\alpha(x)}|v_k(x)|^{\beta(x)}\,dx=  &\int_{\mathbb{R}^{N}}\Phi^{q(x)}_{R}(x)|u_k(x)|^{\alpha(x)}|v_k(x)|^{\beta(x)}\,dx \\ & +\int_{\mathbb{R}^{N}}(1-\Phi^{q(x)}_{R}(x))|u_k(x)|^{\alpha(x)}|v_k(x)|^{\beta(x)}\,dx.
 \end{split}
\end{equation}
We observe that
\begin{equation*}\label{inf5}
 \begin{split}
 \int_{B^{c}_{2R}}|u_k|^{\alpha(x)}|v_k|^{\beta(x)}\,dx \leqslant \int_{\mathbb{R}^{N}}\Phi^{q(x)}_{R}|u_k|^{\alpha(x)}|v_k|^{\beta(x)}\,dx \leqslant \int_{B^{c}_{R}}|u_k|^{\alpha(x)}|v_k|^{\beta(x)}\,dx
 \end{split}
\end{equation*}
and so by \eqref{ccinfi}
\begin{equation}\label{inf6}
 \begin{split}
 \nu_{\infty}= \lim_{R\to \infty}\limsup_{k\to \infty} \int_{\mathbb{R}^{N}}\Phi^{q(x)}_{R}|u_k|^{\alpha(x)}|v_k|^{\beta(x)}\,dx.
 \end{split}
\end{equation}
Arguing similarly as that obtained \eqref{cc2infi} above for with $\Phi_{R}$  is replaced with  $\Phi_{R}^{q(x)}$ we obtain \eqref{cc2inf}.

\medskip 
\textbf{3.} Let us consider again the function $\Phi_{R}$ without loss of generality we assume $\nu_{\infty}>0$. Let $\epsilon \in  (0,1)$ be arbitrary and fixed. By ($\varepsilon_{0}$), we choose $R_1>1$ such that
\begin{equation}\label{inf7}
 \begin{split}
 |p(x,y)-\overline{p}|< \epsilon \mbox{ and } |q(x)-q_{\infty}|< \epsilon \mbox{ for all } |x|,|y|>R_1.
 \end{split}
\end{equation}
From \eqref{SS}, we obtain
 \begin{equation}\label{inf8}
  \begin{split}
\mathcal{S}_{\alpha\beta} \bigg\|(\Phi_{R}u_{k})^{\frac{\alpha(\cdot)}{q(\cdot)}}(\Phi_{R}v_{k})^{\frac{\beta(\cdot)}{q(\cdot)}}\bigg\|_{L^{q(\cdot)}(\mathbb{R}^{N}) } \leqslant  \|( \Phi_{R}u_k,\Phi_{R}v_k)\|.
 \end{split}
\end{equation}
For $R>R_1$, using \eqref{inf7} and Proposition \ref{masmenos}, we have 
\begin{equation*}\label{inf9}
 \begin{split}
 \bigg\|(\Phi_{R}u_{k})^{\frac{\alpha(\cdot)}{q(\cdot)}}(\Phi_{R}v_{k})^{\frac{\beta(\cdot)}{q(\cdot)}}\bigg\|_{L^{q(\cdot)}(\mathbb{R}^{N})}   = &  \bigg\|(\Phi_{R}u_{k})^{\frac{\alpha(\cdot)}{q(\cdot)}}(\Phi_{R}v_{k})^{\frac{\beta(\cdot)}{q(\cdot)}}\bigg\|_{L^{q(\cdot)}(B^{c}_{R})}    \\ \geqslant & \min \bigg \{ \bigg( \int_{B^{c}_{R}} \Phi_{R}^{q(x)}|u_{k}(x)|^{\alpha(x)}|v_{k}(x)|^{\beta(x)}\,dx\bigg)^{\frac{1}{q_{\infty}+\epsilon}},  \\& \bigg( \int_{B^{c}_{R}} \Phi_{R}^{q(x)}|u_{k}(x)|^{\alpha(x)}|v_{k}(x)|^{\beta(x)}\,dx\bigg)^{\frac{1}{q_{\infty}-\epsilon}}\bigg\} 
   \\ \geqslant & \min \bigg \{ \bigg( \int_{B^{c}_{2R}} \Phi_{R}^{q(x)}|u_{k}(x)|^{\alpha(x)}|v_{k}(x)|^{\beta(x)}\,dx\bigg)^{\frac{1}{q_{\infty}+\epsilon}},  \\& \bigg( \int_{B^{c}_{2R}} \Phi_{R}^{q(x)}|u_{k}(x)|^{\alpha(x)}|v_{k}(x)|^{\beta(x)}\,dx\bigg)^{\frac{1}{q_{\infty}-\epsilon}}\bigg\}.
 \end{split}
\end{equation*}
Thus 
\begin{equation}\label{inf10}
  \begin{split}
\liminf_{R\to \infty}\limsup_{k\to \infty} \bigg\|\Phi_{R}u_{k}^{\frac{\alpha(\cdot)}{q(\cdot)}}v_{k}^{\frac{\beta(\cdot)}{q(\cdot)}}\bigg\|_{L^{q(\cdot)}(\mathbb{R}^{N})}  \geqslant \min \big\{ \nu_{\infty}^{\frac{1}{q_{\infty}+\epsilon}}, \nu_{\infty}^{\frac{1}{q_{\infty}-\epsilon}}\big\}.
 \end{split}
\end{equation}
Now, we estimate the right hand side of  \eqref{inf8}.  First, we denote $\eta_{k, R}:=\|(\Phi_{R}u_{k}, \Phi_{R}v_{k})\|$ for simplify. We will show that there exist $R_2 \in  (R_1, \infty)$ and $\eta \in (0,\infty)$ such that 
\begin{equation}\label{inf11}
  \begin{split}
0<\mathcal{S}_{\alpha\beta} \bigg(\frac{\nu_{\infty}}{4}\bigg)^{\frac{1}{q_{\infty}}}   \leqslant \eta_{\star, R}:= \lim_{k\to \infty} \eta_{k,R}< \eta \mbox{ for all } R \in (R_2, \infty).
 \end{split}
\end{equation}
Indeed, we first choose $\overline{\epsilon}>0$ sufficiently small such that
\begin{equation}\label{inf12}
  \begin{split}
\min \bigg\{ \bigg(\frac{\nu_\infty}{2}\bigg)^{\frac{1}{q_{\infty}+\overline{\epsilon}}}, \bigg(\frac{\nu_\infty}{2}\bigg)^{\frac{1}{q_{\infty}-\overline{\epsilon}}}\bigg\} > \bigg( \frac{\nu_{\infty}}{4}\bigg )^{\frac{1}{q_{\infty}}}.
 \end{split}
\end{equation}
Then we can find $\overline{R}_{2}>R_1$ such that
\begin{equation}\label{inf13}
  \begin{split}
 \bigg\|(\Phi_{R}u_{k})^{\frac{\alpha(\cdot)}{q(\cdot)}}(\Phi_{R}v_{k})^{\frac{\beta(\cdot)}{q(\cdot)}}\bigg\|_{L^{q(\cdot)}(\mathbb{R}^{N})} & \geqslant   \min \bigg\{  \bigg( \int_{B^{c}_{R}} \Phi_{R}^{q(x)}|u_{k}(x)|^{\alpha(x)}|v_{k}(x)|^{\beta(x)}\,dx \bigg)^{\frac{1}{q_{\infty}+\overline{\epsilon}}},  \\ &\bigg( \int_{B^{c}_{R}} \Phi_{R}^{q(x)}|u_{k}(x)|^{\alpha(x)}|v_{k}(x)|^{\beta(x)}\,dx \bigg)^{\frac{1}{q_{\infty}-\overline{\epsilon}}} \bigg\}
 \end{split}
\end{equation}
for all $ R>\overline{R}_2.$
Finally, by \eqref{inf6}, we can find $ R_2>\overline{R}_2$  such that
\begin{equation}\label{inf14}
  \begin{split}
 \lim_{k\to \infty}\int_{\mathbb{R}^{N}} \Phi_{R}^{q(x)}|u_{k}(x)|^{\alpha(x)}|v_{k}(x)|^{\beta(x)}\,dx =  \lim_{k\to \infty}\int_{\mathbb{R}^{N}} \Phi_{R}^{q(x)}|u_{k}(x)|^{\alpha(x)}|v_{k}(x)|^{\beta(x)}\,dx > \frac{\nu_{\infty}}{2} \\ 
 \end{split}
\end{equation}
 for all $ R>\overline{R}_2.$
From \eqref{inf13} and \eqref{inf14} we obtain
\begin{equation}\label{inf15}
  \begin{split}
 \limsup_{k\to \infty}\bigg\|(\Phi_{R}u_{k})^{\frac{\alpha(\cdot)}{q(\cdot)}}(\Phi_{R}v_{k})^{\frac{\beta(\cdot)}{q(\cdot)}}\bigg\|_{L^{q(\cdot)}(\mathbb{R}^{N})}  \geqslant   \min \bigg\{  \bigg( \frac{\nu_{\infty}}{2}\bigg)^{\frac{1}{q_{\infty}+\overline{\epsilon}}}, \bigg(\frac{\nu_{\infty}}{2}\bigg)^{\frac{1}{q_{\infty}-\overline{\epsilon}}} \bigg\}
  \end{split}
\end{equation}
and hence by \eqref{inf12}
\begin{equation}\label{inf16}
  \begin{split}
 \limsup_{k\to \infty}\bigg\|(\Phi_{R}u_{k})^{\frac{\alpha(\cdot)}{q(\cdot)}}(\Phi_{R}v_{k})^{\frac{\beta(\cdot)}{q(\cdot)}}\bigg\|_{L^{q(\cdot)}(\mathbb{R}^{N})}  \geqslant  \bigg( \frac{\nu_\infty}{4}\bigg)^{\frac{1}{q_{\infty}-\overline{\epsilon}}} 
  \end{split}
\end{equation}
for all $ R>\overline{R}_2.$
This and \eqref{inf8} yield, for all $ R \in (R_2, \infty)$,
\begin{equation}\label{inf17}
  \begin{split}
\mathcal{S}_{\alpha\beta} \bigg(  \frac{\nu_{\infty}}{4}  \bigg) ^{\frac{1}{q_{\infty}}}   \leqslant \eta_{\star, R} .
 \end{split}
\end{equation}
By a similar argument that obtained \eqref{5.17}, involving Lemma \ref{4.5inf} and choosing $R_2$ larger if necessary we can show that there exist $\eta \in (0, \infty)$ such that  $\eta_{\star, R} \leqslant \eta$ for all  $ R \in (R_2, \infty)$. Thus \eqref{inf11} has been proved.

\noindent We now turn to estimate the right-hand side of \eqref{inf8}. For each $R>R_2$ given let $k_j=k_j(R)$ $(j\in \mathbb{N})$ be a sequence such that 
\begin{equation}\label{inf170}
  \begin{split}
\lim_{j\to \infty}\eta_{k_j,R} =\limsup_{j\to \infty} \, \eta_{k_j,R}  = \eta_{\star,R}.
 \end{split}
\end{equation}
Utilizing Proposition \ref{norma} and \eqref{desiab}    
again, we obtain
\begin{equation}\label{inf190}
    \begin{split}
    1=&\int_{\mathbb{R}^{N}} \bigg(\frac{|\Phi_{r}u_{k_j}|}{\eta^{\overline{p}}_{k_{j},R}}\bigg)^{\overline{p}}\,dx +   \int_{\mathbb{R}^{N}}\int_{\mathbb{R}^{N}}\frac{|(\Phi_{R}  u_{k_{j}})(x)-(\Phi_{R} u_{k_{j}})(y)|^{p(x,y)}}{\eta_{k_{j},R}^{p(x,y)}|x-y|^{N+sp(x,y)}} \,dy\,dx  \\ & + \int_{\mathbb{R}^{N}}\bigg(\frac{|\Phi_{r}v_{k_j}|}{\eta^{\overline{p}}_{k_{j},R}}\bigg)^{\overline{p}}\,dx +   \int_{\mathbb{R}^{N}}\int_{\mathbb{R}^{N}}\frac{|(\Phi_{R}  v_{k_{j}})(x)-(\Phi_{R} v_{k_{j}})(y)|^{p(x,y)}}{\eta_{k_{j},R}^{p(x,y)}|x-y|^{N+sp(x,y)}} \,dy\,dx\\ = & \int_{\mathbb{R}^{N}}\bigg[ \bigg(\frac{|\Phi_{r}u_{k_j}|}{\eta^{\overline{p}}_{k_{j},R}}\bigg)^{\overline{p}}+ \bigg(\frac{|\Phi_{r}v_{k_j}|}{\eta^{\overline{p}}_{k_{j},R}}\bigg)^{\overline{p}}\bigg]\,dx \\ & + 2\int_{B^{c}_{R} }\int_{B_R}\frac{|\Phi_{R}(x)|^{p(x,y)} (| u_{k_{j}}(x)|^{p(x,y)}+ |v_{k_{j}}(y)|^{p(x,y)})}{\eta_{k_{j},R}^{p(x,y)}|x-y|^{N+sp(x,y)}} \,dy\,dx 
    \\&  + \int_{B^{c}_{R} }\int_{B^{c}_{R}}\frac{|(\Phi_{R}  u_{k_{j}})(x)-(\Phi_{R} u_{k_{j}})(y)|^{p(x,y)}}{\eta_{k_{j},R}^{p(x,y)}|x-y|^{N+sp(x,y)}} \,dy\,dx \\&  + \int_{B^{c}_{R} }\int_{B^c_{R}}\frac{|(\Phi_{R}  v_{k_{j}})(x)-(\Phi_{R}v_{k_{j}})(y)|^{p(x,y)}}{\eta_{k_{j},R}^{p(x,y)}|x-y|^{N+sp(x,y)}} \,dy\,dx\\ \leqslant &   \int_{B^c_{R}}\bigg[ \bigg(\frac{|\Phi_{r}u_{k_j}|}{\eta^{\overline{p}}_{k_{j},R}}\bigg)^{\overline{p}}+ \bigg(\frac{|\Phi_{r}v_{k_j}|}{\eta^{\overline{p}}_{k_{j},R}}\bigg)^{\overline{p}}\bigg]\,dx \\&  + \int_{B^c_{R} }\int_{B_R}\frac{|\Phi_{R}(x)-\Phi_{R}(y)|^{p(x,y)} (| u_{k_{j}}(x)|^{p(x,y)}+ |v_{k_{j}}(y)|^{p(x,y)})}{\eta_{k_{j},R}^{p(x,y)}|x-y|^{N+sp(x,y)}} \,dy\,dx\\ &+ C(\epsilon)\int_{B^c_{R} }\int_{B^c_{R}}\frac{|\Phi_{R}(x)-\Phi_{R}(y)|^{p(x,y)} (| u_{k_{j}}(x)|^{p(x,y)}+ |v_{k_{j}}(y)|^{p(x,y)})}{\eta_{k_{j},R}^{p(x,y)}|x-y|^{N+sp(x,y)}} \,dy\,dx \\ & + (1+\epsilon) \int_{B^c_{R} }\int_{B^c_{R}} \frac{|\Phi_{R}(y)|^{p(x,y)}}{\eta_{k_j,R}^{p(x,y)}}\frac{\big | u_{k_{j}}(x)- u_{k_{j}}(y)|^{p(x,y)}}{|x-y|^{N+sp(x,y)}}\,dy\,dx \\ & + (1+\epsilon) \int_{B^c_{R} }\int_{B^c_{R}} \frac{|\Phi_{R}(y)|^{p(x,y)}}{\eta_{k_j,R}^{p(x,y)}}\frac{ | v_{k_{j}}(x)- v_{k_{j}}(y)|^{p(x,y)} }{|x-y|^{N+sp(x,y)}}\,dy\,dx. 
\end{split}
    \end{equation}
Thus, by \eqref{inf190} and the fact that $0\leqslant \Phi_{R}\leqslant 1$
yield
\begin{equation}\label{inf20}
    \begin{split}
    1 \leqslant & \frac{C(\epsilon)}{    \min\{ \eta_{k_j,R}^{\overline{p}}, \eta_{k_j,R}^{p^{-}}\}} 
 \int_{B^c_{R} }\int_{\mathbb{R}^{N}}\frac{|\Phi_{R}(x)-\Phi_{R}(y)|^{p(x,y)} (| u_{k_{j}}(x)|^{p(x,y)}+ |v_{k_{j}}(y)|^{p(x,y)})}{|x-y|^{N+sp(x,y)}} \,dy\,dx \\ & +  \frac{(1+\epsilon)}{\min\{ \eta_{k_j,R}^{\overline{p}+\epsilon}, \eta_{k_j,R}^{\overline{p}-\epsilon}\}}\int_{B^c_{R}} \Phi_{R}(x)\big(U_{k_j}(x) +V_{k_j}(x)\big)\,dx. 
\end{split}
    \end{equation}
Taking the limit superior as $k \to \infty$ in \eqref{inf20} and using \eqref{inf11} and \eqref{inf170}, we obtain 
   \begin{equation}\label{lala} 
   \begin{split}
    1 \leqslant & \frac{C(\epsilon)}{    \min\{ \eta_{\star,R}^{\overline{p}}, \eta_{\star,R}^{p^{-}}\}} 
\limsup_{k\to \infty} \int_{B^c_{R} }\int_{\mathbb{R}^{N}}\frac{|\Phi_{R}(x)-\Phi_{R}(y)|^{p(x,y)} (| u_{k}(x)|^{p(x,y)}+ |v_{k}(y)|^{p(x,y)})}{|x-y|^{N+sp(x,y)}} \,dy\,dx \\ & +  \frac{(1+\epsilon)}{\min\{ \eta_{\star,R}^{\overline{p}+\epsilon}, \eta_{\star,R}^{\overline{p}-\epsilon}\}}\limsup_{k\to \infty}\int_{B^c_{R}} \Phi_{R}(x)\big(U_{k}(x) +V_{k}(x)\big)\,dx. 
\end{split}
    \end{equation}

    \noindent Now, taking the limit as $R\to \infty$ in \eqref{lala}, using the  Lemma \ref{4.5inf} and   \eqref{inf3}, we conclude
\begin{equation*}\label{inf20}
    \begin{split}
    1 \leqslant & \frac{ 1+ \epsilon }{    \min\{ \eta_{\star}^{\overline{p}+\epsilon}, \eta_{\star}^{\overline{p}-\epsilon}\}} \mu_{\infty}, 
\end{split}
    \end{equation*}
this is,
\begin{equation*}\label{inf20}
    \begin{split}
    \eta_{\star} \leqslant (1+\epsilon)^{\frac{1}{\overline{p}-\epsilon}}\max \{ \mu_{\infty}^{\frac{1}{\overline{p}+\epsilon}}, \mu_{\infty}^{\frac{1}{\overline{p}-\epsilon}}  \}
\end{split}
    \end{equation*}
    where $\eta_{\star}:= \displaystyle\liminf_{R\to \infty}\eta_{\star, R}$ and hence $0< \eta_{\star}<\eta$ due to \eqref{inf11}. From this \eqref{inf8} and \eqref{inf10}, we obtain
    \begin{equation*}\label{inf21}
  \begin{split}
\mathcal{S}_{\alpha\beta}  \min \{\nu_{\infty}^{\frac{1}{q_{\infty}+\epsilon}},  \nu_{\infty}^{\frac{1}{q_{\infty}-\epsilon}}\} \leqslant  (1+\epsilon)^{\frac{1}{\overline{p}-\epsilon}}\max \{ \mu_{\infty}^{\frac{1}{\overline{p}+\epsilon}}, \mu_{\infty}^{\frac{1}{\overline{p}-\epsilon}}\}.
 \end{split}
\end{equation*}
    Since $\epsilon$ was chosen arbitrarily in the last inequality \eqref{SSin}, the proof of Theorem \ref{lionsinfity} is complete.
    \end{proof}
    
%%%%%%%%%%%%%%%%%%%%%%%%%%%%%%%%%%%%%%%%%%%%%%%%%%%%%%%%%%%%%%%%%%%%    
\section{\textbf{Existence of nontrivial solutions for the system \eqref{s1}}}
\label{aplication}
%%%%%%%%%%%%%%%%%%%%%%%%%%%%%%%%%%%%%%%%%%%%%%%%%%%%%%%%%%%%%%%%%%%% \hfill \break  
 In this section, we  assume without further mentioning, that the assumptions in Theorem \ref{aplica} are satisfied. Moreover,  we denote by    $C_{i}$, a positive constant for all $i \in \mathbb{N}$.
    \subsection{ Functional properties of the operator $\mathcal{L}_{\mathcal{A}K}$}
    
The following result is used to obtain compactness properties for the operator $\mathcal{L}_{\mathcal{A}K}$, for a proof e.g.   \cite[Lemma 2.10]{bonaldo}.
\begin{lemma} 
\label{s+}
Assume that  $(a_{1})$-$(a_{3})$, and $(\mathcal{K})$   hold. Let $\Phi:\mathit{W}\to \mathbb{R}$ be defined by
\begin{equation*}
\Phi(u,v):= \mathcal{B}(u)+\mathcal{B}(v) \mbox{ for all } (u,v)\in \mathit{W}, 
\end{equation*}
where
$$\mathcal{B}(u):=  \int_{\mathbb{R}^{N}}\int_{\mathbb{R}^{N}}\mathscr{A}(u(x)-u(y))K(x,y)\,dy\,dx+ \int_{\mathbb{R}^{N}}\frac{|u|^{\overline{p}}}{\overline{p}}\,dx,$$  
  $$\mathcal{B}(v):=  \int_{\mathbb{R}^{N}}\int_{\mathbb{R}^{N}}\mathscr{A}(v(x)-v(y))K(x,y)\,dy\,dx+ \int_{\mathbb{R}^{N}}\frac{|v|^{\overline{p}}}{\overline{p}}\,dx,$$
have the following  properties:
\begin{itemize}
\item[$(\mathcal{L}_{1})$] The functional $\Phi$ is well defined on $\mathit{W}$, it is of   class  $C^{1}(\mathit{W}, \mathbb{R})$, and its G\^ateaux derivative is given by where
\begin{equation*}\label{phi'}
\langle \Phi'(u,v),(\varphi, \psi) \rangle = \langle\mathcal{B}'(u),\varphi\rangle+ \langle\mathcal{B}'(v),\psi\rangle \mbox{ for all } (u, v), (\varphi,\psi) \in \mathit{W},
\end{equation*} 
$$\langle\mathcal{B}'(u),\varphi\rangle=  \int_{\mathbb{R}^{N}}\int_{\mathbb{R}^{N}}\mathcal{A}(u(x)-u(y))(\varphi(x)-\varphi(y))K(x,y)\,dy\,dx+ \int_{\mathbb{R}^{N}}|u|^{\overline{p}-2}u\varphi\,dx,$$  
 $$\langle\mathcal{B}'(v),\psi\rangle=  \int_{\mathbb{R}^{N}}\int_{\mathbb{R}^{N}}\mathcal{A}(v(x)-v(y))(\psi(x)-\psi(y))K(x,y)\,dy\,dx+ \int_{\mathbb{R}^{N}}|v|^{\overline{p}-2}v\psi\,dx.$$
\item[$(\mathcal{L}_{2})$]  The functional $\Phi$ is weakly lower semicontinuous, that is, $(u_k,v_k) \rightharpoonup (u,v)$ in $\mathit{W}$ as $ k \to +\infty$ implies that $\displaystyle{\Phi(u,v) \leqslant\liminf_{k\to +\infty} \Phi(u_k,v_k)}.$ 
\item[$(\mathcal{L}_{3})$] The  functional $\Phi' : \mathit{W}\to \mathit{W}'$ is an operator of type $(S_{+})$ on $\mathit{W}$, that is, if 
  \begin{equation*} \label{inffo}
 ( u_k,v_k) \rightharpoonup (u,v)  \mbox{ in }\mathit{W} \mbox{ and } \limsup_{k \to +\infty}\,\langle \,\Phi'(u_k,v_k), (u_k-u,v_k-v) \rangle\leqslant 0,
  \end{equation*}
  then $(u_k,v_k)\to( u,v)$ in $\mathit{W}$ as $k\to +\infty$.
\end{itemize}
\end{lemma} 

\medskip   
For $(u, v) \in \mathit{W}$, we define the energy functional $\mathcal{J}: \mathit{W}\to \mathbb{R}$ by
  \begin{equation*}
  \mathcal{J}(u,v)= \Phi(u,v)  -\lambda\int_{\mathbb{R}^{N}}H(x,u,v)\,dx- \int_{\mathbb{R}^{N}}\frac{|u|^{\alpha(x)}|v|^{\beta(x)}}{q(x)}
\,dx,
    \end{equation*}
  where the functional $\Phi$ was defined in Lemma \ref{s+}.
Then from standard arguments  Lemma \ref{s+} and the hypothesis  $(\mathcal{H})$ the functional   $\mathcal{J}$ is  of class $C^{1}(\mathit{W};\mathbb{R})$ and  for all $(\varphi, \psi)  \in \mathit{W}$. Moreover, we have 
  \begin{equation*}
  \begin{split}
  \langle \mathcal{J}'(u,v), (\varphi, \psi)\rangle =& \langle\Phi'(u,v), (\varphi, \psi)\rangle  -\lambda\int_{\mathbb{R}^{N}}\big[H_{u}(x,u,v)\varphi+ H_{v}(x,u,v)\psi\big]\,dx \\ &-  \int_{\mathbb{R}^{N}}\frac{\alpha(x)|u|^{\alpha(x)-2}u|v|^{\beta(x)}\varphi}{q(x)}\,dx -  \int_{\mathbb{R}^{N}}\frac{\beta(x)|u|^{\alpha(x)}|v|^{\beta(x)-2}v \psi}{q(x)}\,dx.
  \end{split}
    \end{equation*}
 
 We say that the pair $(u, v) \in \mathit{W}$ is a weak solution of system \eqref{s1}, when
  \begin{equation*}
  \begin{split}
   \int_{\mathbb{R}^{N}}\int_{\mathbb{R}^{N}}\big[\mathcal{A}(u(x)-u(y))(\varphi(x)-\varphi(y))+ \mathcal{A}(v(x)-v(y))(\psi(x)-\psi(y))\big]K(x,y)\,dy\,dx\\ +\int_{\mathbb{R}^{N}}\big(|u|^{\overline{p}-2}u\varphi+|v|^{\overline{p}-2}v\psi\big)\,dx = \lambda \int_{\mathbb{R}^{N}}\big[H_{u}(x,u,v)\varphi+ H_{v}(x,u,v)\psi\big]\,dx\\ +  \int_{\mathbb{R}^{N}}\frac{\alpha(x)|u|^{\alpha(x)-2}u|v|^{\beta(x)}\varphi+\beta(x)|u|^{\alpha(x)}|v|^{\beta(x)-2}v \psi}{q(x)}\,dx
    \end{split}
  \end{equation*}
  for all $(\varphi, \psi)  \in \mathit{W}$.

 Clearly, any (weak) solution of \eqref{s1} is a critical points of the Euler–Lagrange functional $\mathcal{J}: \mathit{W} \to \mathbb{R}$ associated with \eqref{s1}, given for all $(u, v) \in \mathit{W}$.

 %Now, on account the hypothesis $(\mathcal{H}),$ given $\varepsilon>0,$ there exists a constant $c_{\varepsilon}>0,$ such that
 %\begin{equation}\label{hipoh}
 %H(x,z)\leqslant\frac{\varepsilon \theta^{+}|z|^{\theta(x)}}{\theta^{-}}+ \frac{c_{\varepsilon}\sigma^{+}|z|^{\sigma(x)}}{\sigma^{-}} 
 %\end{equation}
 % for a.e. $  x \in \mathbb{R}^{N}$ and all $  z \in \mathbb{R}^{2}$. 

%%%%%%%%%%%%%%%%%%%%%%%%%%%%%%%%%%%%%%%%%%%%
  \subsection{Proof of Theorem \ref{aplica}}
%%%%%%%%%%%%%%%%%%%%%%%%%%%%%%%%%%%%%%%%%%%%  

The structural assumptions of Theorem \ref{aplica} imply that the functional $\mathcal{J}$ possesses the geometric features of the Mountain Pass Theorem of Ambrosetti-Rabinowitz. Then, it will be shown here that, the functional $\mathcal{J}$ has the geometric features to get the existence of a Palais–Smale sequence at special levels.
  \begin{lemma}\label{lemaA}
  For any $\lambda \in \mathbb{R}^{+}$, there exist $\delta_{0}>0$  and $\varrho_{0}\in(0,1]$, such that $\mathcal{J}(u,v)\geqslant\delta_{0}>0$ for any $(u,v)\in \mathit{W}$, with $\|(u,v)\|=\varrho_{0}$.  
  \end{lemma}
  \begin{proof}
  Let $(u, v) \in \mathit{W}$ be such that, $\|(u,v)\|\leqslant 1$.
  %small enough. 
  Then from $(a_{2})$, $(a_{3})$, $(\mathcal{K})$, $(\mathcal{H})$,  H\"older inequality (see Proposition \ref{hold3}) and  \eqref{holder}, we obtain
  \begin{equation}\label{eq1}
  \begin{split}
  \mathcal{J}(u,v)\geqslant & d_1 \|(u,v)\|^{\overline{p}} -\frac{\lambda C_1}{\theta^{-}}\bigg[\|(u,v)\|^{\sigma^{-}}+\|(u,v)\|^{\theta^{-}}\\ & +\bigg(\|a\|^{\frac{q^{-}}{q^{-}-\theta^{-}}}_{L^{\frac{q(\cdot)}{q(\cdot)-\theta(\cdot)}}(\mathbb{R}^{N})} + \|b\|^{\frac{q^{-}}{q^{-}-\sigma^{-}}}_{L^{\frac{q(\cdot)}{q(\cdot)-\sigma(\cdot)}}(\mathbb{R}^{N})}\bigg)\|(u,v)\|^{2q^{-}}\bigg] - \frac{c_{2}}{q^{-}}\|(u,v)\|^{q^{-}} \\ \geqslant & d_1 \|(u,v)\|^{\overline{p}} -\lambda C_2\bigg[\|(u,v)\|^{\sigma^{-}}+\|(u,v)\|^{\theta^{-}}+\|(u,v)\|^{2q^{-}}\bigg]- C_{3}\|(u,v)\|^{q^{-}}, 
  \end{split}
  \end{equation}
  where $\displaystyle{d_1= \min\bigg\{\frac{c_{\mathcal{A}b_{0}}}{\overline{p}},\frac{\tilde{c}_{\mathcal{A}\tilde{b}_{0}}}{\overline{p}}, \frac{1}{\overline{p}}\bigg\}>0}.$
  % \mathcal{J}(u,v)\geqslant & d_1 \|(u,v)\|^{\overline{p}} -\lambda\int_{\mathbb{R}^{N}}\big[ a_1(x)|u|^{\theta(x)}+b_{1}(x)|v|^{\sigma(x)-1}|u|+a_{1}(x)|u|^{\theta(x)-1}|v|+b_{1}(x)|v|^{\sigma(x)-1}|u|\big]\,dx \\ &- \frac{C_{2}}{q^{+}}\|(u,v)\|^{q^{-}},
 
 % \begin{equation}\label{eq2}
%  \begin{split}
 % \mathcal{J}(u,v)\geqslant & d_1 \|(u,v)\|^{\overline{p}} -\frac{\lambda\varepsilon \theta^{+}C_{6}}{\theta^{-}} \|(u,v)\|^{\theta^{-}} -\frac{\lambda C_{\varepsilon }\sigma^{+}C_7}{\sigma^{-}} \|(u,v)\|^{\sigma^{-}}- \frac{C_{2}}{q}\|(u,v)\|^{q^{-}}.
%  \end{split}
%  \end{equation}
%  \noindent Therefore, fixing $\varepsilon >0$ small enough  there are positive constants  $C_2$ and $ C_3$ such that
%  \begin{equation*}\label{eq1}
 % \begin{split}
%  \mathcal{J}(u,v)\geqslant & 
 % \end{split}
%  \end{equation*}
  Hence, since $\overline{p}<\theta^{-}<\sigma^{-}< q^{-}< 2q^{-}$ it follows that, there are $0<\varrho_{0}<1$ small enough and
$\delta_{0} > 0$ such that
   \begin{equation*}\label{eq1}
  \begin{split}
  \mathcal{J}(u,v)\geqslant \delta_{0} >0 \,\,\,\, \mbox{ when }  \,\,\,\,  \|(u,v)\|=\varrho_{0}.
  \end{split}
  \end{equation*}
   
  \end{proof}
  \begin{lemma}\label{lemaB}
  There exists a pair $(\mathfrak{e}_1,\mathfrak{e}_2) \in C^{\infty}_{0}(\mathbb{R}^{N})\times C^{\infty}_{0}(\mathbb{R}^{N})$, such that $\mathfrak{e}_1\geqslant 0$ and  $\mathfrak{e}_2\geqslant 0$ in $\mathbb{R}^{N}$, $\mathcal{J}(\mathfrak{e}_1,\mathfrak{e}_2)<0$,  $\|(\mathfrak{e}_1,\mathfrak{e}_2)\|\geqslant 2$ and $\displaystyle{\int_{\mathbb{R}^{N}}|\mathfrak{e}_1|^{\alpha(x)}|\mathfrak{e}_2|^{\beta(x)}\,dx>0}$, for any $\lambda \in \mathbb{R}^{+}$.
  \end{lemma}
  \begin{proof}
 Let $(u,v) \in C^{\infty}_{0}(\mathbb{R}^{N})\times C^{\infty}_{0}(\mathbb{R}^{N})$, be such that $u\geqslant 0 $ and  $v \geqslant 0 $ in $\mathbb{R}^{N}$, $\|(u,v)\|=1$ and  $\displaystyle{\int_{\mathbb{R}^{N}}|u|^{\alpha(x)}|v|^{\beta(x)}\,dx>0}$. Then by $(a_2),$  $(a_3),$ $(\mathcal{K})$, and  $(\mathcal{H})$ for all $\zeta \in \mathbb{R}$, with $\zeta >1$, we have
  \begin{equation}\label{eq4}
  \begin{split}
  \mathcal{J}(\zeta u,\zeta v)\leqslant  & \zeta^{\overline{p}}d_2\|(u, v)\|^{\overline{p}}- \frac{\zeta^{q^{-}}}{q^{+}}\int_{\mathbb{R}^{N}}|u|^{\alpha(x)}|v|^{\beta(x)}\,dx ,
  \end{split}
  \end{equation}
  where $\displaystyle{d_2= \max \bigg \{\frac{C_{\mathcal{A}}b_1}{\overline{p}}, \frac{\tilde{C}_{\mathcal{A}}\tilde{b}_1}{\overline{p}}, \frac{1}{\overline{p}}\bigg\}}$.
 Thus, once that $\overline{p}< q^{-} $, taking   $\zeta \to+\infty$ in \eqref{eq4}, we obtain
  $$\mathcal{J}(\zeta u,\zeta v)\to -\infty.$$
   Therefore, taking $(\mathfrak{e}_1,\mathfrak{e}_2)=\zeta_{0} \, (u,v)$ with $\zeta_{0}>0$ large enough, we conclude the proof.
  \end{proof}
  
  \medskip
  Now, we discuss the compactness property for the functional $\mathcal{J}$, given by the (PS) condition at a suitable level. To this aim, we fix $\lambda > 0$  and set
  $$c_{\lambda}=\inf_{\gamma\in \Gamma}\max_{\zeta\in[0,1]}\mathcal{J}(\gamma(\zeta)),$$
  where 
  $$\Gamma=\big\{ \gamma \in C([0,1], \mathit{W}): \gamma(0)=(0,0),\, \mathcal{J}(\gamma(1))<0\big\}.$$
 Note that, by Lemma \ref{lemaA}, we have that $c_{\lambda}>0$ for $\lambda >0$. In particular $\|(\mathfrak{e}_1,\mathfrak{e}_2)\|\geqslant 2>\varrho_{0}$ since 
 $\varrho_{0}\in(0,1]$. The following result reports an important assymptotic limit of the level $c_{\lambda}$ as $\lambda\to \infty$.
 %  Moreover, we have the following result.
  
  \begin{lemma}\label{lemaC}The set of critical levels $\{c_{\lambda}\}_{\lambda\in \mathbb{R}^{+}}$ verifies
  $$\lim_{\lambda\to + \infty}c_{\lambda}=0.$$
    \end{lemma}
    \begin{proof}
    From Lemmas \ref{lemaA} and  \ref{lemaB}, the functional $\mathcal{J}$ satisfies the mountain pass geometry at $ (0, 0)$ and $(\mathfrak{e}_1, \mathfrak{e}_2)$, hence there exists $\zeta_{\lambda}>1$, such that
     $$\mathcal{J}(\zeta_{\lambda}\mathfrak{e}_{1},\zeta_{\lambda}\mathfrak{e}_{2})= \max_{\zeta>0}\mathcal{J}(\zeta \mathfrak{e}_1,\zeta \mathfrak{e}_2).$$
     Consequently, it follows that
     $$\langle  \mathcal{J}'(\zeta_{\lambda}\mathfrak{e}_{1},\zeta_{\lambda}\mathfrak{e}_{2}),(\mathfrak{e}_1, \mathfrak{e}_2) \rangle= 0.$$
 Thus, we have
 \begin{equation}\label{limitado}
 \begin{split}
 \int_{\mathbb{R}^{N}}\int_{\mathbb{R}^{N}}\big[\mathcal{A}(\zeta_{\lambda}\mathfrak{e}_{1}(x)-\zeta_{\lambda}\mathfrak{e}_{1}(y))(\mathfrak{e}_{1}(x)-\mathfrak{e}_{1}(y))+\mathcal{A}(\zeta_{\lambda}\mathfrak{e}_{2}(x)-\zeta_{\lambda}\mathfrak{e}_{2}(y))(\mathfrak{e}_{2}(x)-\mathfrak{e}_{2}(y))\big]K(x,y)\,dy\,dx\\+\int_{\mathbb{R}^{N}}|\zeta_{\lambda}|^{\overline{p}-1}|\mathfrak{e}_{1}|^{\overline{p}}\,dx +\int_{\mathbb{R}^{N}}|\zeta_{\lambda}|^{\overline{p}-1}|\mathfrak{e}_{2}|^{\overline{p}}\,dx = \lambda \int_{\mathbb{R}^{N}}\big[ \mathrm{H}_{u}(x, \zeta_{\lambda}\mathfrak{e}_{1}, \zeta_{\lambda}\mathfrak{e}_{2} )\mathfrak{e}_{1}+\mathrm{H}_{v}(x, \zeta_{\lambda}\mathfrak{e}_{1}, \zeta_{\lambda}\mathfrak{e}_{2} )\mathfrak{e}_{2}\big]\,dx\\+ \int_{\mathbb{R}^{N}} \zeta_{\lambda}^{q(x)-1}|\mathfrak{e}_1|^{\alpha(x)}|\mathfrak{e}_2|^{\beta(x)}\,dx.
 \end{split}
\end{equation}     
    Then, since $\lambda>0$ and taking into account the hypothesis $(\mathcal{H})$, we conclude that
     \begin{equation}\label{limitada}
 \begin{split}
 \int_{\mathbb{R}^{N}}\int_{\mathbb{R}^{N}}\big[\mathcal{A}(\zeta_{\lambda}\mathfrak{e}_{1}(x)-\zeta_{\lambda}\mathfrak{e}_{1}(y))(\mathfrak{e}_{1}(x)-\mathfrak{e}_{1}(y))+\mathcal{A}(\zeta_{\lambda}\mathfrak{e}_{2}(x)-\zeta_{\lambda}\mathfrak{e}_{2}(y))(\mathfrak{e}_{2}(x)-\mathfrak{e}_{2}(y))\big]K(x,y)\,dy\,dx\\+\int_{\mathbb{R}^{N}}|\zeta_{\lambda}|^{\overline{p}-1}|\mathfrak{e}_{1}|^{\overline{p}}\,dx +\int_{\mathbb{R}^{N}}|\zeta_{\lambda}|^{\overline{p}-1}|\mathfrak{e}_{2}|^{\overline{p}}\,dx \geqslant \int_{\mathbb{R}^{N}} \zeta_{\lambda}^{q(x)-1}|\mathfrak{e}_1|^{\alpha(x)}|\mathfrak{e}_2|^{\beta(x)}\,dx.
 \end{split}
\end{equation}  
  \textbf{Claim.} The sequence $ \{\zeta_{\lambda}\}_{\lambda\in \mathbb{R}^{+}}$ is bounded in $\mathbb{R}$. 
  
  \noindent \textit{Proof of Claim.} Define the set $\Xi=\{\lambda>0:\zeta_{\lambda}\|(\mathfrak{e}_1,\mathfrak{e}_2)\| \geqslant 1\}$, with  $\zeta_{\lambda}>1$. Thus, from $(a_1)$, $(a_2)$, $(a_3)$, $(\mathcal{K})$, and \eqref{limitada}, we see that
     \begin{equation*}
     \begin{split}
    \zeta_{\lambda}^{-1}d_3\bigg\lbrace \int_{\mathbb{R}^{N}}\int_{\mathbb{R}^{N}}\bigg[\frac{|\zeta_{\lambda}\mathfrak{e}_1 (x)-\zeta_{\lambda} \mathfrak{e}_1 (y)|^{p(x,y)}}{|x-y|^{N+sp(x,y)}}+ \frac{|\zeta_{\lambda}\mathfrak{e}_2 (x)-\zeta_{\lambda}\mathfrak{e}_2 (y)|^{p(x,y)}}{|x-y|^{N+sp(x,y)}}\bigg]\,dy\,dx\\
     + \int_{\mathbb{R}^{N}}\big(|\zeta^{-1}_{\lambda}\mathfrak{e}_1|^{\overline{p}}+|\zeta^{-1}_{\lambda}\mathfrak{e}_2|^{\overline{p}}\big)\,dx\bigg\rbrace\geqslant \zeta_{\lambda}^{-1} \int_{\mathbb{R}^{N}}\zeta_{\lambda}^{q(x)}|\mathfrak{e}_1|^{\alpha(x)}|\mathfrak{e}_2|^{\beta(x)}\,dx,
     \end{split}
     \end{equation*}
    where $d_3 = \max  \bigg\{ 1, \frac{\overline{p}C_{\mathcal{A}}b_1}{p^{-}}, \frac{\overline{p}\tilde{C}_{\mathcal{A}}\tilde{b}_1}{p^{-}}  \bigg \}$.  Hence, on account of Proposition \ref{norma}, we obtain 
  \begin{equation*}
   d_{3}\|(\mathfrak{e}_1,\mathfrak{e}_2)\|^{\overline{p}}\geqslant \zeta_{\lambda}^{q^{-}-  \overline{p}}\int_{\mathbb{R}^{N}}|\mathfrak{e}_1|^{\alpha(x)}|\mathfrak{e}_2|^{\beta(x)}\,dx
\end{equation*}
for any $\lambda \in \Xi$, which yields that $ \{\zeta_{\lambda}\}_{\lambda\in \mathbb{R}^{+}}$ is bounded  since $q^{-}>\overline{p}$ and $\displaystyle{\int_{\mathbb{R}^{N}}|\mathfrak{e}_1|^{\alpha(x)}|\mathfrak{e}_2|^{\beta(x)}\,dx}>0$.
     
\noindent Fix now a sequence $\{\lambda_{k}\}_{k \in \mathbb{N}}\subset\mathbb{R}^{+}$, such that $\lambda_{k}\to+\infty$ as $k\to+\infty$. Then the sequence $\{\zeta_{\lambda_{k}}\}_{k \in\mathbb{N}}$ is bounded. Hence, there exist $\zeta_0\geqslant0$ and a subsequence of $\{\lambda_{k}\}_{k \in \mathbb{N}}$, still denoted by $\{\lambda_{k}\}_{k \in \mathbb{N}}$, such that
\begin{equation*}\label{falta}
\lim_{k\to+\infty}\zeta_{\lambda_{k}}= \zeta_{0}.
\end{equation*}     
 From  $(a_1)$ 
     and \eqref{limitado},  for all $k \in \mathbb{R}^{N}$, we obtain that 
     \begin{equation}\label{c.2.0}
     \begin{split}
     \lambda_{k}\int_{\mathbb{R}^{N}} \big[\mathrm{H}_{u}(x, \zeta_{\lambda_{k}}\mathfrak{e}_1,\zeta_{\lambda_{k}}\mathfrak{e}_2 )\mathfrak{e}_1+ \mathrm{H}_{v}(x, \zeta_{\lambda_{k}}\mathfrak{e}_1, \zeta_{\lambda_{k}}\mathfrak{e}_2 )\mathfrak{e}_2\big]\,dx+ \zeta_{\lambda_{k}}^{q^{-}-1}\int_{\mathbb{R}^{N}}|\mathfrak{e}_1|^{\alpha(x)}|\mathfrak{e}_2|^{\beta(x)}\,dx\leqslant C_{4}.
     \end{split}
     \end{equation}
     We assert that $\zeta_0=0$. Otherwise, due to the hypothesis $(\mathcal{H})$ and the Dominated Convergence Theorem, we have
     $$\lim_{k\to +\infty}\int_{\mathbb{R}^{N}} \mathrm{H}_{u}(x, \zeta_{\lambda_{k}}\mathfrak{e}_1, \zeta_{\lambda_{k}}\mathfrak{e}_2 )\mathfrak{e}_1\,dx= \int_{\mathbb{R}^{N}} \mathrm{H}_{u}(x, \zeta_{0}\mathfrak{e}_1,\zeta_0 \mathfrak{e}_2)e_1\,dx$$
and     
     
      $$\lim_{k\to +\infty}\int_{\mathbb{R}^{N}} \mathrm{H}_{v}(x, \zeta_{\lambda_{k}}\mathfrak{e}_1, \zeta_{\lambda_{k}}\mathfrak{e}_2 )\mathfrak{e}_2\,dx=\int_{\mathbb{R}^{N}} \mathrm{H}_{v}(x, \zeta_{0}\mathfrak{e}_1,\zeta_0 \mathfrak{e}_2)\mathfrak{e}_2\,dx. $$
     In particular, as  $k \to +\infty$, 
  \begin{equation}\label{convera}
    \begin{split}
    \lim_{k\to +\infty}\int_{\mathbb{R}^{N}} \big[\mathrm{H}_{u}(x, \zeta_{\lambda_{k}}\mathfrak{e}_1, \zeta_{\lambda_{k}}\mathfrak{e}_2 )\mathfrak{e}_1+\mathrm{H}_{v}(x, \zeta_{\lambda_{k}}\mathfrak{e}_1, \zeta_{\lambda_{k}}\mathfrak{e}_2 )\mathfrak{e}_2\big]\,dx=
    \end{split}
\end{equation}  
$$\int_{\mathbb{R}^{N}} \big[\mathrm{H}_{u}(x, \zeta_{0}\mathfrak{e}_1,\zeta_0 \mathfrak{e}_2)\mathfrak{e}_1+\mathrm{H}_{u}(x, \zeta_{0}\mathfrak{e}_1, \zeta_0 \mathfrak{e}_2)\mathfrak{e}_2\big]\,dx>0.$$
 
    \noindent Then by, \eqref{convera}, $(\mathcal{H})$ and the fact that
     \begin{equation*}\label{c4}
     \int_{\mathbb{R}^{N}}|e_1|^{\alpha(x)}|e_2|^{\beta(x)}\,dx>0
     \end{equation*}
(see  Lemma \ref{lemaB}), recalling that $\lambda_{k}\to +\infty$, we conclude that
   \begin{equation*}
   \begin{split}
  & \lim_{k\to+\infty}\bigg[\lambda_{k}\int_{\mathbb{R}^{N}} \big[\mathrm{H}_{u}(x, \zeta_{\lambda_{k}}\mathfrak{e}_1,\zeta_{\lambda_{k}}\mathfrak{e}_2 )\mathfrak{e}_1+ \mathrm{H}_{v}(x, \zeta_{\lambda_{k}}\mathfrak{e}_1, \zeta_{\lambda_{k}}\mathfrak{e}_2 )\mathfrak{e}_2\big]\,dx \\  & +  \zeta_{\lambda_{k}}^{q^{-}-1}\int_{\mathbb{R}^{N}}|\mathfrak{e}_1|^{\alpha(x)}|\mathfrak{e}_2|^{\beta(x)}\,dx\bigg]=+ \infty, 
   \end{split}
\end{equation*}     
which is a contradiction with \eqref{c.2.0}. Thus $\zeta_0 = 0$ and $\zeta_{\lambda}\to 0$ as $\lambda \to +\infty$, since the sequence $\{\lambda_{k}\}_{k \in  \mathbb{N}}$ is arbitrary.
    
\medskip    
 Finally, let $\overline{\gamma}(\zeta)=\zeta(\mathfrak{e}_1,\mathfrak{e}_2)$, $\zeta\in[0,1]$, thus $\overline{\gamma} \in \Gamma$, and
    \begin{equation}\label{estrela}
    0<c_{\lambda}\leqslant\max_{\zeta\geqslant 0}\mathcal{J}(\overline{\gamma}(\zeta))\leqslant \mathcal{J}(\zeta_{\lambda}\mathfrak{e}_1,\zeta_{\lambda}\mathfrak{e}_2)\leqslant \zeta_{\lambda}^{\overline{p}}d_2\|(\mathfrak{e}_1, \mathfrak{e}_2)\|^{\overline{p}}- \frac{\zeta_{\lambda}^{q^{-}-1}}{q^{+}}\int_{\mathbb{R}^{N}}|\mathfrak{e}_1|^{\alpha(x)}|\mathfrak{e}_2|^{\beta(x)}\,dx.
    \end{equation}
     Once that $\zeta_{\lambda}\to 0 $ as $\lambda\to+\infty$, and $\mathfrak{e}_1$ and $\mathfrak{e}_2$ do not depend on $\lambda$,  from \eqref{estrela},  we conclude  that
     $$\lim_{\lambda\to+\infty}c_{\lambda}=0.$$
     
    \end{proof}

    \begin{proposition} 
    \label{lemaD}
   There exists $\lambda^{\star}>0$ such that, for all $ \lambda \geqslant \lambda^{\star}$, $\mathcal{J}$ satisfies the $(PS)_{c_{\lambda}}$ condition on $\mathit{W}$.
  \end{proposition}
    \begin{proof}
    Let $\{(u_k,v_k)\}_{k \in\mathbb{N}}\subset\mathit{W}$ be a Palais-Smale sequence of $\mathcal{J}$ at level $c_{\lambda}$ for any $\lambda>0$, that is,
    \begin{equation}\label{pale}
    \mathcal{J}(u_k,v_k)\to c_{\lambda} \mbox{ and }  \mathcal{J}'(u_k,v_k)\to  0 \mbox{ in } \mathit{W}' \mbox{ as } k\to+\infty.
    \end{equation}
Then, using $(a_1)$-$(a_3)$, $(\mathcal{K})$, and due to the hypothesis $(\mathcal{H})$, we have
 \begin{equation}\label{2.13}
 \begin{split}
 \mathcal{J}(u_k,v_k)-\frac{1}{\theta^{-}}\langle \mathcal{J}'(u_k,v_k), (u_k,v_k)\rangle \geqslant & d_{4}\bigg[  \int_{\mathbb{R}^{N}}\int_{\mathbb{R}^{N}}\frac{|u_{k}(x)-u_k(y)|^{p(x,y)}}{|x-y|^{N+sp(x,y)}}\,dy\,dx+\int_{\mathbb{R}^{N}}|u_k|^{\overline{p}}\,dx \\& +  \int_{\mathbb{R}^{N}}\int_{\mathbb{R}^{N}}\frac{|v_{k}(x)-v_k(y)|^{p(x,y)}}{|x-y|^{N+sp(x,y)}}\,dy\,dx+ \int_{\mathbb{R}^{N}} |v_k|^{\overline{p}}\,dx \bigg] \\& + \big(\frac{1}{\theta^{-}}-\frac{1}{q^{-}}\big ) \int
 _{\mathbb{R}^{N}}|u_k|^{\alpha(x)}|v_{k}|^{\beta(x)}\,dx\bigg] \\ \geqslant & d_{4}\bigg[  \int_{\mathbb{R}^{N}}\int_{\mathbb{R}^{N}}\frac{|u_{k}(x)-u_k(y)|^{p(x,y)}}{|x-y|^{N+sp(x,y)}}\,dy\,dx+\int_{\mathbb{R}^{N}}|u_k|^{\overline{p}}\,dx \\& +  \int_{\mathbb{R}^{N}}\int_{\mathbb{R}^{N}}\frac{|v_{k}(x)-v_k(y)|^{p(x,y)}}{|x-y|^{N+sp(x,y)}}\,dy\,dx+ \int_{\mathbb{R}^{N}} |v_k|^{\overline{p}}\,dx \bigg]
 \end{split}
\end{equation}   
where $d_{4}= \min \bigg\{ \big(\frac{1}{\overline{p}}-\frac{1}{\theta^{-}}\big ) c_{\mathcal{A}}b_{0}, \big(\frac{1}{\overline{p}}-\frac{1}{\theta^{-}}\big ) \tilde{c}_{\mathcal{A}}\tilde{b}_{0} \bigg\}>0$. 

\medskip  
Now, we suppose  that the sequence $\{(u_k,v_k)\}_{k \in\mathbb{N}}$ is unbounded in  $\mathit{W}$. Without loss of generality we consider $\|(u_k,v_k)\|\geqslant 1$. Then, from \eqref{pale} and \eqref{2.13} we conclude that, there exists $d_{\lambda}>0$, such that as $k \to +\infty$, 
  \begin{equation}\label{2.14}
  c_{\lambda}+ d_{\lambda}\|(u_k,v_k)\|+o(1)\geqslant d_{4}\|(u_k,v_k)\|^{p^{-}}.
\end{equation} 
Dividing \eqref{2.14} by  $\|(u_k,v_k)\|^{p^{-}}$ and taking $k\to+\infty$ once that $p^{-}>1$, we obtain a contradiction.
\noindent Therefore, the sequence $\{(u_k,v_k)\}_{k \in\mathbb{N}}$ is bounded in $\mathit{W}$ and hence
\begin{equation}\label{limite}
\begin{split}
C_{\star}:= \sup_{k \in \mathbb{N}} \bigg\{  \int_{\mathbb{R}^{N}}\int_{\mathbb{R}^{N}}\bigg(\frac{|u_{k}(x)-u_k(y)|^{p(x,y)}}{|x-y|^{N+sp(x,y)}}  + \frac{|v_{k}(x)-v_k(y)|^{p(x,y)}}{|x-y|^{N+sp(x,y)}}\bigg)\,dy\,dx  &  \\+\int_{\mathbb{R}^{N}}\big(|u_k|^{\overline{p}}+ |v_k|^{\overline{p}}\big)\,dx \bigg\}< \infty.
\end{split}
\end{equation}  
Moreover, from Proposition \ref{mista},  H\"older inequality (see Proposition \ref{hold3}), Lemma \ref{mergulho}, and since the sequence $\{(u_k,v_k)\}_{k \in \mathbb{N}}$ is bounded in $\mathit{W}$,  we have
   \begin{equation*}
   \begin{split}
   \int_{\mathbb{R^{N}}} ||u_{k}|^{\alpha(x)-2}u_{k}|v_k|^{\beta(x)}|^{\frac{q(x)}{q(x)-1}}\,dx \leqslant C_5.
      \end{split}
\end{equation*}     
 Similarly, we conclude that
    \begin{equation*}
   \begin{split}
   \int_{\mathbb{R^{N}}} ||u_{k}|^{\alpha(x)}|v_k|^{\beta(x)-2}v_{k}|^{\frac{q(x)}{q(x)-1}}\,dx \leqslant C_6.
   \end{split}
\end{equation*}  
  
 \noindent  Thus, since the sequence  $\{(u_k,v_k)\}_{k \in\mathbb{N}}$ is bounded in   the reflexive Banach space $\mathit{W}$,  there exist $(u_\lambda,v_\lambda)\in \mathit{W}$,   and bounded nonnegative Radon measures $\mu$, $\nu$ in $\mathbb{R}^{N}$ (see \cite[Proposition 1.202]{fonseca}) such that, up to a subsequence, still denoted by $\{(u_k,v_k)\}_{k \in\mathbb{N}}$, we obtain
\begin{equation}\label{f01}
 (u_{k},v_{k})\rightharpoonup (u_{\lambda},v_{\lambda}) \mbox{ in } \mathit{W},
  \end{equation}
\begin{equation}\label{f1}(u_{k}(x),v_{k}(x))\to (u_{\lambda}(x),v_{\lambda}(x)) \mbox{ a.e. in } \mathbb{R}^{N},\\
 \end{equation}
 
  \begin{equation}\label{f2} (u_{k},v_{k})\to (u_{\lambda},v_{\lambda}) \mbox{ in } L^{r(\cdot)}_{loc}(\mathbb{R}^{N})\times  L^{r(\cdot)}_{loc}(\mathbb{R}^{N}), \,\,\,\,  r(x)<p^{*}_{s}(x) \mbox{ for all } x \in \mathbb{R}^{N},
  \end{equation}
 
  \begin{equation}\label{f02} |u_{k}|^{\alpha(x)-2}u_k|v_{k}|^{\beta(x)}\rightharpoonup  |u_{\lambda}|^{\alpha(x)-2}u_{\lambda}|v_{\lambda}|^{\beta(x)} \mbox{  in } L^{\frac{q(\cdot)}{q(\cdot)-1}}(\mathbb{R}^{N}), \end{equation}
  \begin{equation}\label{f3} 
     |u_{k}|^{\alpha(x)}|v_{k}|^{\beta(x)-2}v_k\rightharpoonup  |u_{\lambda}|^{\alpha(x)}|v_{\lambda}|^{\beta(x)-2}v_{\lambda} \mbox{  in } L^{\frac{q(\cdot)}{q(\cdot)-1}}(\mathbb{R}^{N}), 
     \end{equation}
  \begin{equation}\label{f4} 
  \bigg[ |u_k|^{\overline{p}}+|v_{k}|^{\overline{p}}+ \int_{\mathbb{R}^{N} }\frac{|u_k(x)-u_k(y)|^{p(x,y)}}{\zeta^{p(x,y)}|x-y|^{N+sp(x,y)}} \,dy+\int_{\mathbb{R}^{N} }\frac{|v_k(x)-v_k(y)|^{p(x,y)}}{\zeta^{p(x,y)}|x-y|^{N+sp(x,y)}} \,dy\bigg] \overset{\ast}{\rightharpoonup}  \mu \mbox{ in }  \mathcal{M}(\mathbb{R}^{N}),
   \end{equation}
    \begin{equation}\label{2.16}
    \begin{split}
     |u_k|^{\alpha(x)}|v_k|^{\beta(x)} \overset{\ast}{\rightharpoonup} \nu \mbox{ in }  \mathcal{M}(\mathbb{R}^{N}),
    \end{split}
    \end{equation}
  %  where $d_5=\min\{c_\mathcal{A}b_0, \tilde{c}_%\mathcal{A}\tilde{b}_0,1\}$.
where $\alpha(x)+\beta(x)=q(x)$ for all $x \in \mathbb{R}^{N}$.

Now, applying Theorem \ref{lions}, there exist at most a countable index set $I$,  families of points $\{x_i\}_{i \in I}$ and of nonnegative numbers $ \{ \mu_i \}_{i \in I}$ and $\{ \nu_{i}\}_{i \in I}$, such that
    \begin{equation}\label{f5a} 
\nu= |u_{\lambda}|^{\alpha(x)}|v_{\lambda}|^{\beta(x)}+ \sum_{i \in I}\nu_{i}\delta_{x_i},
\end{equation}
\begin{equation*}\label{f6} 
\mu \geqslant|u_{\lambda}|^{\overline{p}}+|v_{\lambda}|^{\overline{p}} +  \int_{\mathbb{R}^{N} }\frac{|u_{\lambda}(x)-u_{\lambda}(y)|^{p(x,y)}}{\zeta^{p(x,y)}|x-y|^{N+sp(x,y)}} \,dy+\int_{\mathbb{R}^{N} }\frac{|v_{\lambda}(x)-v_{\lambda}(y)|^{p(x,y)}}{\zeta^{p(x,y)}|x-y|^{N+sp(x,y)}} \,dy  + \sum_{i \in I}\mu_{i}\delta_{x_{i}},
\end{equation*}
\begin{equation}\label{aplicalions}
\mathcal{S}_{\alpha\beta}\nu_{i}^{\frac{1}{\overline{p}^{\star}_{s}}}\leqslant \mu_{i}^{\frac{1}{\overline{p}}} \mbox{ for all } i \in I.
\end{equation}
 Moreover, we have
 \begin{equation*}\label{i1}
 \begin{split}
 \limsup_{k\to+\infty}\int_{\mathbb{R}^{N}}\bigg[ |u_k|^{\overline{p}}+|v_{k}|^{\overline{p}}+ \int_{\mathbb{R}^{N} }\frac{|u_k(x)-u_k(y)|^{p(x,y)}}{\zeta^{p(x,y)}|x-y|^{N+sp(x,y)}} \,dy \\ +\int_{\mathbb{R}^{N} }\frac{|v_k(x)-v_k(y)|^{p(x,y)}}{\zeta^{p(x,y)}|x-y|^{N+sp(x,y)}}dy\bigg]\,dx= \mu(\mathbb{R}^{N})+ \mu_{\infty},
  \end{split}
 \end{equation*}
  \begin{equation}\label{i2}
 \limsup_{k\to+\infty}\int_{\mathbb{R}^{N}} |u_{k}|^{\alpha(x)}|v_{k}|^{\beta(x)}\,dx= \nu(\mathbb{R}^{N})+ \nu_{\infty},
  \end{equation} 
   \begin{equation}\label{i3}
\mathcal{S}_{\alpha\beta}\nu_{\infty}^{\frac{1}{q_{\infty}}}\leqslant \mu_{\infty}^{\frac{1}{\overline{p}}}.
  \end{equation} 
  
\noindent Hence, by \eqref{2.13}-\eqref{2.16} as $ k \to +\infty$, we obtain that
 \begin{equation}\label{2.19}
 c_{\lambda}+ o(1) \geqslant d_4\|(u_k,v_k)\|^{p^{-}}+ \bigg( \frac{1}{\theta^{-}}- \frac{1}{q^{-}}\bigg)\int
 _{\mathbb{R}^{N}}|u_k|^{\alpha(x)}|v_{k}|^{\beta(x)}\,dx,
\end{equation}  
and combining this with \eqref{i2}, we obtain
   \begin{equation*}\label{2.123}
 c_{\lambda}+ o(1) \geqslant  \bigg( \frac{1}{\theta^{-}}- \frac{1}{q^{-}}\bigg)(\nu(\mathbb{R}^{N})+ \nu_{\infty}).
\end{equation*}
   We will show that $I= \emptyset$ and $\nu_{\infty}=0$. First, we suppose on the contrary that $I\neq\emptyset$. Let $i \in I $  and for  $\rho>0$, define $\Psi_{\rho}$ as in
Lemma \ref{4.4a} with $x_0$ replaced by $x_i$, this is  $\Psi_{\rho,i}(x):= \Psi\big( \frac{x-x_i}{\rho}\big)$ for $  x \in \mathbb{R}^{N}$.  Then $\Psi_{\rho,i}(u_k,v_k) \in \mathit{W}$ and so $\langle \mathcal{J}'(u_k,v_k), \Psi_{\rho,i}(u_k,v_k)\rangle= o(1)$ as $k\to \infty.$
Therefore
   \begin{equation}\label{22}
   \begin{split}
   o(1)= & \langle \mathcal{J}'(u_k,v_k), \Psi_{\rho,i}(u_k,v_k)\rangle\\  =& \int_{\mathbb{R}^{N}}\int_{\mathbb{R}^{N}} \mathcal{A}(u_k(x)-u_{k}(y))u_k(x)(\Psi_{\rho,i}(x)-\Psi_{\rho,i}(y))K(x,y)\,dy\,dx\\ &+ \int_{\mathbb{R}^{N}}\int_{\mathbb{R}^{N}} \mathcal{A}(v_k(x)-v_{k}(y))v_k(x)(\Psi_{\rho,i}(x)-\Psi_{\rho,i}(y))K(x,y)\,dy\,dx \\&+  \int_{\mathbb{R}^{N}}\int_{\mathbb{R}^{N}} \mathcal{A}(u_k(x)-u_{k}(y))(u_k(x)-u_{k}(y))\Psi_{\rho,i}(y)K(x,y)\,dy\,dx+ \int_{\mathbb{R}^{N}}|u_k|^{\overline{p}}\Psi_{\rho,i}(x)\,dx\\&+  \int_{\mathbb{R}^{N}}\int_{\mathbb{R}^{N}} \mathcal{A}(v_k(x)-v_{k}(y))(v_k(x)-v_{k}(y))\Psi_{\rho,i}(y)K(x,y)\,dy\,dx + \int_{\mathbb{R}^{N}}|v_k|^{\overline{p}}\Psi_{\rho,i}(x)\,dx \\ & 
   - \lambda \int_{\mathbb{\mathbb{R}^{N}}}\Psi_{\rho,i}\big[ \mathrm{H}_{u}(x, u_k, v_k)u_k+ \mathrm{H}_{v}(x, u_k, v_k)v_k\big]\,dx- \int_{\mathbb{\mathbb{R}^{N}}}\Psi_{\rho,i}|u_k|^{\alpha(x)}|v_k|^{\beta(x)}\,dx.
   \end{split}
\end{equation}    
    Now, let $\delta>0$ be arbitrary but fixed, from $(a_2)$, $(\mathcal{K})$, $(\mathcal{H})$,  Young's inequality, and \eqref{limite},  we conclude that
    \begin{equation}\label{limites}
    \begin{split}
     \bigg| \int_{\mathbb{R}^{N}}\int_{\mathbb{R}^{N}} \big[\mathcal{A}(u_k(x)-&u_{k}(y))u_k(x)+\mathcal{A}(v_k(x)-v_{k}(y))v_k(x)\big](\Psi_{\rho,i}(x)-\Psi_{\rho,i}(y))K(x,y)\,dy\,dx\bigg| \\ \leqslant & d_5 \int_{\mathbb{R}^{N}}\int_{\mathbb{R}^{N}} \frac{|u_k(x)-u_k(y)|^{p(x,y)-1}|u_k(y)|(\Psi_{\rho,i}(x)-\Psi_{\rho,i}(y))}{|x-y|^{N+sp(x,y)}}\,dy\,dx \\ & +   d_5 \int_{\mathbb{R}^{N}}\int_{\mathbb{R}^{N}} \frac{|v_k(x)-v_k(y)|^{p(x,y)-1}|v_k(y)|(\Psi_{\rho,i}(x)-\Psi_{\rho,i}(y))}{|x-y|^{N+sp(x,y)}}\,dy\,dx \\      \leqslant & d_5 \delta \int_{\mathbb{R}^{N}}\int_{\mathbb{R}^{N}} \frac{\big(|u_k(x)-u_k(y)|^{p(x,y)}+|v_k(x)-v_k(y)|^{p(x,y)}\big)}{|x-y|^{N+sp(x,y)}}\,dy\,dx  \\ & + C_{7}  \int_{\mathbb{R}^{N}}\int_{\mathbb{R}^{N}}\big(|u_k(y)|^{p(x,y)}+ |v_k(y)|^{p(x,y)}\big)\frac{|\Psi_{\rho,i}(x)-\Psi_{\rho,i}(y)|^{p(x,y)}}{|x-y|^{N+sp(x,y)}}\,dy\,dx   \\
\leqslant & d_5\delta C_{\star}+ C_{7}  \int_{\mathbb{R}^{N}}\int_{\mathbb{R}^{N}}|u_k(y)|^{p(x,y)}\frac{|\Psi_{\rho,i}(x)-\Psi_{\rho,i}(y)|^{p(x,y)}}{|x-y|^{N+sp(x,y)}}\,dy\,dx \\ & +  C_{7}  \int_{\mathbb{R}^{N}}\int_{\mathbb{R}^{N}} |v_k(y)|^{p(x,y)}\frac{|\Psi_{\rho,i}(x)-\Psi_{\rho,i}(y)|^{p(x,y)}}{|x-y|^{N+sp(x,y)}}\,dy\,dx, 
    \end{split}
    \end{equation}
    where $d_5= \max \{ C_{\mathcal{A}}b_1, \tilde{C}_{\mathcal{A}}\tilde{b}_1\}$.
 Taking limit superior in \eqref{limites} as $k \to +\infty$ and limit superior as $ \rho \to 0^{+}$ with taking Lemma \ref{4.4a} into account, we arrive at   
 \begin{equation*}\label{limite1}
 \begin{split}
  \limsup_{\rho\to 0^{+}} \limsup_{k\to +\infty}\bigg| \int_{\mathbb{R}^{N}}\int_{\mathbb{R}^{N}}\big[\mathcal{A}(u_k(x)-u_{k}(y))u_k(x)\big](\Psi_{\rho,i}(x)-\Psi_{\rho,i}(y))K(x,y)\,dy\,dx \\   + \int_{\mathbb{R}^{N}}\int_{\mathbb{R}^{N}}\big[\mathcal{A}(v_k(x)-v_{k}(y))v_k(x)\big](\Psi_{\rho,i}(x)-\Psi_{\rho,i}(y))K(x,y)\,dy\,dx \bigg|\leqslant  d_5C_{\star}\delta.
 \end{split}
\end{equation*} 
   Since $\delta >0$ was chosen arbitrarily, we obtain
   \begin{equation}\label{limite1}
   \begin{split}
 \limsup_{\rho\to 0^{+}} \limsup_{k\to +\infty}\bigg| \int_{\mathbb{R}^{N}}\int_{\mathbb{R}^{N}}\big[\mathcal{A}(u_k(x)-u_{k}(y))u_k(x)\big](\Psi_{\rho,i}(x)-\Psi_{\rho,i}(y))K(x,y)\,dy\,dx \\   + \int_{\mathbb{R}^{N}}\int_{\mathbb{R}^{N}}\big[\mathcal{A}(v_k(x)-v_{k}(y))v_k(x)\big](\Psi_{\rho,i}(x)-\Psi_{\rho,i}(y))K(x,y)\,dy\,dx \bigg|=0.
 \end{split}
\end{equation} 
   Note that the boundedness of $\{(u_k,v_k)\}_{k \in\mathbb{N}}$ in  $\mathit{W}$ implies the the boundedness of $\{(u_k,v_k)\}_{k \in\mathbb{N}}$ in $L^{q(\cdot)}(\mathbb{R}^{N})\times L^{q(\cdot)}(\mathbb{R}^{N})$ due to Theorem \ref{critico} and \eqref{remark}.  
    On the other hand, on account the hypothesis $(\mathcal{H})$,  H\"older inequality (see Proposition \ref{hold3}), we obtain
    \begin{equation}\label{contas} 
    \begin{split}
  0 \leqslant &  \int_{\mathbb{R}^{N}}\big[\Psi_{\rho,i}\mathrm{H}_{u}(x,u_k,v_k)u_k+ \Psi_{\rho,i}\mathrm{H}_{v}(x,u_k,v_k)v_k\big]\,dx \\\leqslant&   \int_{(B(x_{i},2\rho))}\big[a(x)|u_k|^{\theta(x)}+b(x)|v_k|^{\sigma(x)-1}|u_k|+ a(x)|u_k|^{\theta(x)-1}|v_k|+b(x)|v_k|^{\sigma(x)}\big]\,dx
  \\ \leqslant &  \|a\|_{L^{\frac{q(\cdot)}{q(\cdot)-\theta(\cdot)}}(B(x_{i},2\rho))}[1+\|u_k\|^{\theta^{-}}_{L^{q(\cdot)}(\mathbb{R}^{N})}]+  \|b\|_{L^{\frac{q(\cdot)}{q(\cdot)-\sigma(\cdot)}}(B(x_{i},2\rho))}\|v_k\|^{\theta^{-}}_{L^{q(\cdot)}(\mathbb{R}^{N})}\|u_k\|_{L^{q(\cdot)}(\mathbb{R}^{N})} \\ & + \|a\|_{L^{\frac{q(\cdot)}{q(\cdot)-\theta(\cdot)}}(B(x_{i},2\rho))}\|v_k\|^{\sigma^{+}}_{L^{q(\cdot)}(\mathbb{R}^{N})}\|u_k\|_{L^{q(\cdot)}(\mathbb{R}^{N})}+\|b\|_{L^{\frac{q(\cdot)}{q(\cdot)-\sigma(\cdot)}}(B(x_{i},2\rho))}[1+\|u_k\|^{\sigma^{+}}_{L^{q(\cdot)}(\mathbb{R}^{N})}]  \\ \leqslant & C_8\bigg[ \|a\|_{L^{\frac{q(\cdot)}{q(\cdot)-\theta(\cdot)}}(B(x_{i},2\rho))}+  \|b\|_{L^{\frac{q(\cdot)}{q(\cdot)-\sigma(\cdot)}}(B(x_{i},2\rho))}\bigg]. 
    \end{split} 
    \end{equation}
 Hence by \eqref{contas}, we conclude
 \begin{equation}\label{as}
 \lim_{\rho\to 0^{+}}\lim_{k \to +\infty} \int_{\mathbb{R}^{N}}\big[\Psi_{\rho,i}\mathrm{H}_{u}(x,u_k,v_k)u_k+ \Psi_{\rho,i}\mathrm{H}_{v}(x,u_k,v_k)v_k\big]\,dx=0.
\end{equation}    
    In conclusion, passing to the limit in \eqref{22}, using  $(a_2)$, $(\mathcal{K})$, \eqref{f01}-\eqref{2.16}, and  \eqref{limite1}-\eqref{as}, we obtain the crucial inequality for $i \in I$,  as $ \rho\to 0^{+}$
    \begin{equation}\label{26}
    \int_{\mathbb{R}^{N}}\Psi_{\rho,i}\,d\nu + o(1) \geqslant  \int_{\mathbb{R}^{N}}\int_{\mathbb{R}^{N}}\Psi_{\rho,i}\,d\mu.
    \end{equation}
    By Lemma \ref{lemaC} there exists $\lambda^{\star}>0$ sufficiently large, such that
    \begin{equation}\label{27}
    c_{\lambda}< \bigg(\frac{1}{\theta^{-}}- \frac{1}{q^{-}}\bigg)\min \bigg \{ \mathcal{S}_{\alpha\beta}^{(qh)^{+}}, \mathcal{S}_{\alpha\beta}^{(qh)^{-}}\bigg\} 
    \end{equation}
   for  all $ \lambda \geqslant \lambda^{\star}$ and  $h(x)= \frac{\overline{p}}{q(x)-\overline{p}}$ for  all $x \in \mathbb{R}^{N}$ .

Now, we note that from \eqref{aplicalions} and \eqref{26}, for all $i \in I$, we obtain
 $$\mathcal{S}_{\alpha\beta}\nu_{i}^{\frac{\overline{p}}{\overline{p}^{\star}_{s}}}\leqslant \mu_{i}\leqslant\nu_{i}.$$ 
 
  \noindent \textbf{Claim:} We claim that, $\nu_{i}=0$.  
 
  \noindent \textit{Proof of Claim.} Assume by contradiction that $\nu_{i}>0$  for some $i \in I$. Then, we have $\nu_{i}\geqslant \min\bigg\{ \mathcal{S}_{\alpha\beta}^{(qh)^{+}}, \mathcal{S}_{\alpha\beta}^{(qh)^{-}}\bigg\}$ and by \eqref{2.19}, as $ k \to +\infty$, we conclude that  
    \begin{equation*}
    c_\lambda +o(1)\geqslant \bigg(\frac{1}{\theta^{-}}- \frac{1}{q^{-}}\bigg) \int_{\mathbb{R}^{N}}\Psi_{\rho,i}\,d\nu. 
    \end{equation*}
   On the other hand, as $k\to+\infty$ and $\rho\to 0^{+}$, it follows that
    \begin{equation*}
    c_\lambda \geqslant \bigg(\frac{1}{\theta^{-}}- \frac{1}{q^{-}}\bigg)\nu_{i} \geqslant \bigg(\frac{1}{\theta^{-}}- \frac{1}{q^{-}}\bigg)\min\bigg\{ \mathcal{S}_{\alpha\beta}^{(qh)^{+}}, \mathcal{S}_{\alpha\beta}^{(qh)^{-}}\bigg\},  
    \end{equation*}
   which is a contradiction with \eqref{27}. Therefore, $\nu_{i}=0$ for  all $i \in I$. 
   
 \noindent  Consequently, $\mu_{i}=0$ for all  $i \in I$ and for all $\lambda \geqslant \lambda^{\star}$, that is $I= \emptyset$. Thus, as  $k \to +\infty$, we conclude that
    \begin{equation*}
    \begin{split}
     |u_k|^{\alpha(x)}|v_k|^{\beta(x)} \overset{\ast}{\rightharpoonup} \nu = |u|^{\alpha(x)}|v|^{\beta(x)}  \mbox{ in }  \mathcal{M}(\mathbb{R}^{N}).
    \end{split}
    \end{equation*}
 Next, we prove that $\nu_\infty =0$. Suppose on the contrary that    $\nu_\infty >0$. Let $\Phi_{R}$  be the same as in Lemma \ref{4.5inf}. Then, from a similar argument obtained  in \eqref{22}, we arrive at 
 
   $$\langle \mathcal{J}'(u_k,v_k), \Phi_{R}(u_k,v_k)\rangle= o(1).$$
Therefore, we obtain
   \begin{equation}\label{223}
   \begin{split}
   o(1)= & \langle \mathcal{J}'(u_k,v_k), \Phi_{R}(u_k,v_k)\rangle\\  =& \int_{\mathbb{R}^{N}}\int_{\mathbb{R}^{N}} \mathcal{A}(u_k(x)-u_{k}(y))u_k(x)(\Phi_{R}(x)-\Phi_{R}(y))K(x,y)\,dy\,dx\\ &+ \int_{\mathbb{R}^{N}}\int_{\mathbb{R}^{N}} \mathcal{A}(v_k(x)-v_{k}(y))v_k(x)(\Phi_{R}(x)-\Phi_{R}(y))K(x,y)\,dy\,dx \\&+  \int_{\mathbb{R}^{N}}\int_{\mathbb{R}^{N}} \mathcal{A}(u_k(x)-u_{k}(y))(u_k(x)-u_{k}(y))\Phi_{R}(y)K(x,y)\,dy\,dx+ \int_{\mathbb{R}^{N}}|u_k|^{\overline{p}}\Phi_{R}(x)\,dx\\&+  \int_{\mathbb{R}^{N}}\int_{\mathbb{R}^{N}} \mathcal{A}(v_k(x)-v_{k}(y))(v_k(x)-v_{k}(y))\Phi_{R}(y)K(x,y)\,dy\,dx + \int_{\mathbb{R}^{N}}|v_k|^{\overline{p}}\Phi_{R}(x)\,dx \\ & 
   - \lambda \int_{\mathbb{\mathbb{R}^{N}}}\Phi_{R}(x)\big[ \mathrm{H}_{u}(x, u_k, v_k)u_k + \mathrm{H}_{v}(x, u_k, v_k)v_k\big]\,dx - \int_{\mathbb{\mathbb{R}^{N}}}\Phi_{R}(x)|u_k|^{\alpha(x)}|v_k|^{\beta(x)}\,dx.
   \end{split}
\end{equation}    
Now, let $\delta>0$ be arbitrary but fixed, from  
$(a_2)$, $(\mathcal{K})$, $(\mathcal{H})$,  Young's inequality, and \eqref{limite},  we have
    $$\bigg| \int_{\mathbb{R}^{N}}\int_{\mathbb{R}^{N}} \mathcal{A}(u_k(x)-u_{k}(y))u_k(x)(\Phi_{R}(x)-\Phi_{R}(y))\,dy\,dx$$ $$ + \int_{\mathbb{R}^{N}}\int_{\mathbb{R}^{N}} \mathcal{A}(v_k(x)-v_{k}(y))v_k(x)(\Phi_{R}(x)-\Phi_{R}(y))\,dy\,dx\bigg|$$
    \begin{equation}\label{limitesa}
    \begin{split}
&\leqslant d_5C_{\star}\delta + C_{9}  \int_{\mathbb{R}^{N}}\int_{\mathbb{R}^{N}}\big(|u_k(y)|^{p(x,y)}+ |v_k(y)|^{p(x,y)}\big)\frac{|\Phi_{R}(x)-\Phi_{R}(y)|^{p(x,y)}}{|x-y|^{N+sp(x,y)}}\,dy\,dx. 
    \end{split}
    \end{equation}
    
   \noindent Taking the limit superior in \eqref{limitesa} as $k \to +\infty$, the limit superior as $ R \to \infty$ with taking Lemma \ref{4.5inf}, and since $\delta >0$ was arbitrarily, we get
   \begin{equation}\label{limite1a}
   \begin{split}
 \limsup_{R \to \infty} \limsup_{k\to +\infty}\bigg| \int_{\mathbb{R}^{N}}\int_{\mathbb{R}^{N}} \big[ \mathcal{A}(u_k(x)-u_{k}(y))u_k(x)\big](\Phi_{R}(x)-\Phi_{R}(y))K(x,y)\,dy\,dx \\ + \int_{\mathbb{R}^{N}}\int_{\mathbb{R}^{N}} \big[ \mathcal{A}(v_k(x)-v_{k}(y))v_k(x)\big](\Phi_{R}(x)-\Phi_{R}(y))K(x,y)\,dy\,dx \bigg|=0.
 \end{split}
\end{equation} 
  
   \noindent On the other hand, due to the hypothesis and using similar argument of \eqref{as}, we obtain
  \begin{equation}\label{asa}
 \lim_{R\to +\infty}\lim_{k \to +\infty} \int_{\mathbb{R}^{N}}\big[\Phi_{R}\mathrm{H}_{u}(x,u_k,v_k)u_k+ \Phi_{R}\mathrm{H}_{v}(x,u_k,v_k)v_k\big]\,dx=0.
\end{equation}
  
  \noindent  Therefore by  $(a_2)$, $(\mathcal{K})$, \eqref{223},  \eqref{f01}-\eqref{2.16},  \eqref{limite1a} and \eqref{asa} as $R\to \infty$,  we get
  
  \begin{equation}\label{desiin}
 \mu_{\infty}\leqslant\nu_{\infty}.
\end{equation}
  
   \noindent Here we have used \eqref{inf2}  and the fact that  
  \begin{equation*}\label{akas}
\mu_{\infty}= \lim_{R\to \infty}\limsup_{k\to \infty} \int_{\mathbb{R}^{N}}\Phi_{R}(x)(U_{k}(x)+V_{k}(x))\,dx ,
\end{equation*}
  which can be obtained applying $\Phi_{R}$ instead of $\Phi_{R}^{q(x)}$ in \eqref{inf4.a}. Combining \eqref{desiin} with \eqref{i3}, we have
  \begin{equation}\label{i7}
 \mathcal{S}_{\alpha\beta}\nu_{\infty}^{\frac{\overline{p}q_{\infty}}{q_{\infty}-\overline{p}}}\leqslant \mu_{\infty}\leqslant\nu_{\infty}.
  \end{equation}
  The fact that $q_{\infty}= \displaystyle\lim_{|x|\to \infty}q(x) \in [q^{-}, q^{+}]$ yields
  $$(qh)^{-}\leqslant \frac{q_{\infty}\overline{p}}{q_{\infty}-\overline{p}}\leqslant (qh)^{+} \,\,\,\,\, \mbox{ and }  \,\,\,\,\, h^{-}\leqslant \frac{\overline{p}}{q_{\infty}-\overline{p}}\leqslant h^{+},$$
and from \eqref{i7}
 \begin{equation}\label{i8}
 \nu_{\infty} \geqslant \min\{  \mathcal{S}_{\alpha\beta}^{(qh)^{+}},  \mathcal{S}_{\alpha\beta}^{(qh)^{-}}  \}.
\end{equation}  
  Then, by  \eqref{i2}, \eqref{2.19} and \eqref{i8}, we obtain
\begin{equation*}\label{contra}
 c_{\lambda}+ o(1) \geqslant  \bigg( \frac{1}{\theta^{+}}- \frac{1}{q^{+}}\bigg)\min\{  \mathcal{S}_{\alpha\beta}^{(qh)^{+}},  \mathcal{S}_{\alpha\beta}^{(qh)^{-}}  \},
\end{equation*}
  which is a contradiction to  \eqref{27}. Therefore $ \nu_{\infty}=0$.
 
Combining the facts that,   $I= \emptyset$, and $\nu_{\infty}=0$ with \eqref{f5a}, \eqref{i2}, it follows that
  \begin{equation}\label{aw}
    \begin{split}
     \limsup_{k \to +\infty }\int_{\mathbb{R}^{N}}|u_k|^{\alpha(x)}|v_k|^{\beta(x)} \,dx = \int_{\mathbb{R}^{N}} |u_{\lambda}|^{\alpha(x)}|v_{\lambda}|^{\beta(x)}\,dx. 
    \end{split}
    \end{equation}
    By the Fatou's Lemma, we get from  \eqref{f1}  
  \begin{equation}\label{aw1}
    \begin{split}
     \int_{\mathbb{R}^{N}} |u_{\lambda}|^{\alpha(x)}|v_{\lambda}|^{\beta(x)}\,dx \leqslant  \liminf_{k \to +\infty }\int_{\mathbb{R}^{N}}|u_k|^{\alpha(x)}|v_k|^{\beta(x)} \,dx. 
    \end{split}
    \end{equation}
  Thus from \eqref{aw} and \eqref{aw1}, we conclude that
   \begin{equation}\label{aw22}
    \begin{split}
     \lim_{k \to +\infty }\int_{\mathbb{R}^{N}}|u_k|^{\alpha(x)}|v_k|^{\beta(x)} \,dx = \int_{\mathbb{R}^{N}} |u_{\lambda}|^{\alpha(x)}|v_{\lambda}|^{\beta(x)}\,dx. 
    \end{split}
    \end{equation}
  Therefore, by Lemma \ref{lieb} and \eqref{aw22}, we have
    \begin{equation*}\label{aw10}
    \begin{split}
    \lim_{k\to \infty} \int_{\mathbb{R}^{N}} |u_k-u_{\lambda}|^{\alpha(x)}|v_k-v_{\lambda}|^{\beta(x)}\,dx=0.  
    \end{split}
    \end{equation*}
    
  Now, we observe that 
  \begin{equation*}\label{prova}
   \begin{split}
 \langle  \Phi'(u_k,v_k ), (u_k-u_{\lambda}, v_k-v_{\lambda})\rangle  = & \langle 	\mathcal{J}'(u_k,v_k), (u_k, v_k)\rangle \\& +  \lambda \int_{\mathbb{R}^{N}}\big[\mathrm{H}_{u}(x,u_k,v_k)(u_k-u_{\lambda})+\mathrm{H}_{v}(x,u_k,v_k)(v_k-v_{\lambda})\big]\,dx \\& + \int_{\mathbb{R}^{N}}\frac{\alpha(x)}{q(x)}|u_k|^{\alpha(x)-2}u_k|v_k|^{\beta(x)}(u_k-u_{\lambda})\,dx \\&  +\int_{\mathbb{R}^{N}}\frac{\beta(x)}{q(x)}|u_k|^{\alpha(x)}v_k|v_k|^{\beta(x)-2}(v_k-v_{\lambda})\,dx.
    \end{split}
  \end{equation*}
Thus, we have 
  \begin{equation}\label{fim}
   \begin{split}
 \langle  \Phi'(u_k,v_k ), (u_k-u_{\lambda}, v_k-v_{\lambda})\rangle  = 0.
    \end{split}
  \end{equation}
  Indeed, as $(u_k,v_k)\rightharpoonup(u_{\lambda},v_{\lambda})$ in $\mathit{W}$ and $ \mathcal{J}'(u_k,v_k)\to  0 $ as $k \to \infty$. Then, we obtain  
  \begin{equation} \label{fim1}
 \langle\mathcal{J}_{\lambda}'(u_k,v_k), (u_k-u_{\lambda}, v_k-v_{\lambda})\rangle \longrightarrow 0.
  \end{equation}
 Moreover, by $(\mathcal{H})$ and  H\"older inequality (see Proposition \ref{hold3}), we have
 $$\int_{\mathbb{R}^{N}}\big[\mathrm{H}_{u}(x,u_k,v_k)(u_k-u_{\lambda})+ \mathrm{H}_{v}(x,u_k,v_k)(v_k-v_{\lambda})\big]\,dx$$
 \begin{equation}\label{fim2}
 \begin{split}
   \leqslant &  \int_{\mathbb{R}^{N}}\big[|\mathrm{H}_{u}(x,u_k,v_k)||u_k-u_{\lambda}|+|\mathrm{H}_{v}(x,u_k,v_k)||v_k-v_{\lambda}|\big]\,dx \\ \leqslant & \int_{\mathbb{R}^{N}}\big[ a(x)|u_k|^{\theta(x)-1}+b(x)|v_k|^{\sigma(x)-1}\big]|u_k-u_{\lambda}|\,dx \\ & + \int_{\mathbb{R}^{N}}\big[ a(x)|u_k|^{\theta(x)-1}+b(x)|v_k|^{\sigma(x)-1}\big]|v_k-v_{\lambda}|\,dx \\ \leqslant &C_{10}\big[\|u_k\|^{\theta^{-}-1}_{L^{\theta(\cdot)}(\mathbb{R}^{N})}\|a(x)(u_k-u_{\lambda})\|_{L^{\theta(\cdot)}(\mathbb{R}^{N})}\\ & +\|v_k\|^{\sigma^{-}-1}_{L^{\sigma(\cdot)}(\mathbb{R}^{N})}\|b(x)(u_k-u_{\lambda})\|_{L^{\sigma(\cdot)}(\mathbb{R}^{N})} \\ & + \|u_k\|^{\theta^{-}-1}_{L^{\theta(\cdot)}(\mathbb{R}^{N})}\|a(x)(v_k-v_{\lambda})\|_{L^{\theta(\cdot)}(\mathbb{R}^{N})}\\ & +\|v_k\|^{\sigma^{-}-1}_{L^{\sigma(\cdot)}(\mathbb{R}^{N})}\|b(x)(v_k-v_{\lambda})\|_{L^{\sigma(\cdot)}(\mathbb{R}^{N})}\big].
    \end{split}
 \end{equation}
  Then, by \eqref{f01}-\eqref{f2},  \eqref{fim2}, and Lemma \ref{mergulhocom}, as $k \to \infty$, we conclude that 
  \begin{equation}\label{fim22}
 \begin{split}
 \int_{\mathbb{R}^{N}}\big[\mathrm{H}_{u}(x,u_k,v_k)(u_k-u_{\lambda})+ \mathrm{H}_{v}(x,u_k,v_k)(v_k-v_{\lambda})\big]\,dx \longrightarrow 0  .
    \end{split}
 \end{equation}
  Furthermore,  we obtain
  \begin{equation*}\label{fim3}
   \begin{split}
 \int_{\mathbb{R}^{N}}\bigg[\frac{\alpha(x)}{q(x)}|u_k|^{\alpha(x)-2}u_k|v_k|^{\beta(x)}(u_k-u_{\lambda})  +\frac{\beta(x)}{q(x)}|u_k|^{\alpha(x)}v_k|v_k|^{\beta(x)-2}(v_k-v_{\lambda})\bigg]\,dx 
    \end{split}
  \end{equation*}
  \begin{equation}\label{fim3}
   \begin{split}
 = & \int_{\mathbb{R}^{N}}\bigg[\frac{\alpha(x)}{q(x)}|u_k|^{\alpha(x)}|v_k|^{\beta(x)}- \frac{\alpha(x)}{q(x)}|u_k|^{\alpha(x)-2}u_ku_{\lambda}|v_k|^{\beta(x)}\bigg]\,dx\\  &+\int_{\mathbb{R}^{N}}\bigg[\frac{\beta(x)}{q(x)}|u_k|^{\alpha(x)}|v_k|^{\beta(x)} -\frac{\beta(x)}{q(x)}|u_k|^{\alpha(x)}|v_k|^{\beta(x)-2}v_kv_{\lambda}\bigg]\,dx \\  =&  \int_{\mathbb{R}^{N}}|u_k|^{\alpha(x)}|v_k|^{\beta(x)}\,dx -\int_{\mathbb{R}^{N}} \frac{\alpha(x)}{q(x)}|u_k|^{\alpha(x)-2}u_k|v_k|^{\beta(x)}u_{\lambda}\,dx - \int_{\mathbb{R}^{N}} \frac{\beta(x)}{q(x)}|u_k|^{\alpha(x)}|v_k|^{\beta(x)-2}v_kv_{\lambda}\,dx.
    \end{split}
  \end{equation}
  Then,  since $\alpha(x)+\beta(x)=q(x)$ for all $x \in \mathbb{R}^{N}$, by \eqref{f02}, \eqref{f3},  \eqref{aw}, \eqref{aw22}, \eqref{fim3} and Dominated Convergence Theorem,  as $k \to \infty$,  we conclude that
  
  \begin{equation}\label{fim4}
   \begin{split}
 \int_{\mathbb{R}^{N}}\bigg[\frac{\alpha(x)}{q(x)}|u_k|^{\alpha(x)-2}u_k|v_k|^{\beta(x)}(u_k-u_{\lambda})  +\frac{\beta(x)}{q(x)}|u_k|^{\alpha(x)}v_k|v_k|^{\beta(x)-2}(v_k-v_{\lambda})\bigg]\,dx \longrightarrow 0.
    \end{split}
  \end{equation}
  Therefore, \eqref{fim1}, \eqref{fim22} and \eqref{fim4} we conclude  \eqref{fim}. Once that $\Phi'$ is an operator of type  $(S_{+})$   on $\mathit{W}$ (see Lemma \ref{s+}), it follows that $(u_k,v_k)\to(u_{\lambda}, v_{\lambda})$ in $\mathit{W},$  as $k \to \infty$,
 and  we conclude that $\mathcal{J}$  sastifies the $(PS)_{c_\lambda}$ condition.
  
    \end{proof}
    
%%%%%%%%%%%%%%%%%%%%%%%%%%%%%%%%%%%%%%%%%%%%%%%%    
     \subsubsection*{Proof of Theorem \ref{aplica}}
%%%%%%%%%%%%%%%%%%%%%%%%%%%%%%%%%%%%%%%%%%%%%%%%     
    \begin{proof}
    By Lemmas \ref{lemaA}, \ref{lemaB} and  Proposition \ref{lemaD} the hypotheses of Mountain Pass Theorem are satisfied. Hence, we conclude that  $(u_{\lambda},v_{\lambda})$  is weak solution for problem \eqref{s1} for $\lambda\geqslant \lambda^{\star} $.
    Moreover,  $(u_{\lambda},v_{\lambda})$ is  weak solution nontrivial. Indeed, suppose that $(u_{\lambda},v_{\lambda})=(0,0)$, hence $\mathcal{J}(u_{\lambda},v_{\lambda})=0$ and $\mathcal{J}'(u_{\lambda},v_{\lambda})=0$, but  $\mathcal{J}'(u_{\lambda},v_{\lambda})=c_{\lambda}$, hence $c_{\lambda}=0$ that is a contradiction, because from Lemma \ref{lemaA} is know that $c_{\lambda}>0$. Therefore,   $(u_{\lambda},v_{\lambda})$ is nontrivial weak solution for problem \eqref{s1}.
    Finally, note that $\lambda\geqslant \lambda^{\star}$, by using $(a_{1})$-$(a_{3})$, $(\mathcal{K})$, $(\mathcal{H}),$  \eqref{2.13} and  Proposition \ref{norma},    we have
 \begin{equation}\label{2.130}
\begin{split}
c_{\lambda}+ o(1) \geqslant & d_{4}\bigg[  \int_{\mathbb{R}^{N}}\int_{\mathbb{R}^{N}}\frac{|u_{\lambda}(x)-u_{\lambda}(y)|^{p(x,y)}}{|x-y|^{N+sp(x,y)}}\,dy\,dx+\int_{\mathbb{R}^{N}}|u_{\lambda}|^{\overline{p}}\,dx \\& +  \int_{\mathbb{R}^{N}}\int_{\mathbb{R}^{N}}\frac{|v_{\lambda}(x)-v_{\lambda}(y)|^{p(x,y)}}{|x-y|^{N+sp(x,y)}}\,dy\,dx+ \int_{\mathbb{R}^{N}} |v_k|^{\overline{p}}\,dx \bigg] \\& + \bigg( \frac{1}{\theta^{-}}- \frac{1}{q^{-}}\bigg)\int
 _{\mathbb{R}^{N}}|u_{\lambda}|^{\alpha(x)}|v_{\lambda}|^{\beta(x)}\,dx, 
\\  \geqslant & d_4 \min \bigg\{ \|(u_{\lambda}, v_{\lambda})\|^{p^{-}}, \|(u_{\lambda}, v_{\lambda})\|^{\overline{p}} \bigg\}
\end{split}
\end{equation}   
where $d_{4}= \min \bigg\{ \big(\frac{1}{\overline{p}}-\frac{1}{\theta^{-}}\big ) c_{\mathcal{A}}b_{0}, \big(\frac{1}{\overline{p}}-\frac{1}{\theta^{-}}\big ) \tilde{c}_{\mathcal{A}}\tilde{b}_{0} \bigg\}>0$.
 
\noindent   Then, from  Lemma \ref{lemaC} and \eqref{2.130}, we get  
\begin{equation}\label{nona}
\lim_{\lambda \to \infty}\|(u_{\lambda}, v_{\lambda})\|=0.
\end{equation}
 
\noindent  Finally, we proof that both components of solution $(u_{\lambda}, v_{\lambda})$ are nontrivial. Indeed, assume by contradiction  that $u_{\lambda}\neq 0$, but $v_{\lambda}=0$ a.e. in $\mathbb{R}^{N}$. Then, taking $(\varphi,\psi)=(u_{\lambda},0)\in \mathit{W}$ in the definition of weak solution, we have by conditions $(a_1)$-$(a_3)$, $(\mathcal{K})$, $(\mathcal{H})$, H\"older inequality (see Proposition \ref{hold3}) and Proposition \ref{mista}, we obtain 
  \begin{equation}\label{nono}
  \begin{split}
 c_{\mathcal{A}}b_{0}\min\{\|u_{\lambda}\|^{\overline{p}}, \|u_{\lambda}\|^{p^{-}}\} \leqslant & \int_{\mathbb{R}^{N}}\int_{\mathbb{R}^{N}}\big[\mathcal{A}(u{\lambda}(x)-u_{\lambda}(y))(u_{\lambda}(x)-u_{\lambda}(y))K(x,y)\,dy\,dx +\int_{\mathbb{R}^{N}}\big|u_{\lambda}|^{\overline{p}}\,dx \\ = &\lambda \int_{\mathbb{R}^{N}}H_{u}(x,u_{\lambda},0)u_{\lambda}\,dx  \leqslant  \lambda C_{10} \max\{ \| u_{\lambda}\|^{\theta^{-}}, \| u_{\lambda}\|^{\theta^{+}} \}.
  \end{split}
  \end{equation}
  %\begin{equation}
  %\begin{split}
% c_{\mathcal{A}}b_{0}\min\{\|u_{\lambda}\|^{\overline{p}}, \|u_{\lambda}\|^{p^{-}}\} \leqslant & \int_{\mathbb{R}^{N}}\int_{\mathbb{R}^{N}}\big[\mathcal{A}(u(x)-u(y))(u_{\lambda}(x)-u_{\lambda}(y))K(x,y)\,dy\,dx +\int_{\mathbb{R}^{N}}\big|u|^{\overline{p}}\,dx \\ = &\lambda \int_{\mathbb{R}^{N}}H_{u}(x,u_{\lambda},0)u_{\lambda}\,dx  \\  \leqslant & \lambda \int_{\mathbb{R}^{N}} a(x)|u_{\lambda}|^{\theta(x)}\,dx \\  \leqslant & \lambda C_{10} \max\{ \| u_{\lambda}\|^{\theta^{-}}, \| u_{\lambda}\|^{\theta^{+}} \}
 % \end{split}
 % \end{equation}
 
 \noindent Thus,   using    \eqref{nona} and fact that $p^{-}\leqslant \overline{p}<\theta^{-}\leqslant \theta^{+}<\sigma^{-}\leqslant \sigma^{+}< q^{-}$, in \eqref{nono}, we obtain that $u_{\lambda}=0$ a.e. in $\mathbb{R}^{N}$,  which is a contradiction. Similarly, the case $u_{\lambda}=0$ a.e. in $\mathbb{R}^{N}$ and $v_{\lambda}\neq 0$ cannot occur.
  \noindent Therefore we conclude the proof of this Theorem.
    \end{proof}
    
%%%%%%%%%%%%%%%%%%%%%%%%%%%%%%%%%%%%%%%%%%%%%%%%%%    
      \subsection{Proof of Theorem \ref{aplica2}}
%%%%%%%%%%%%%%%%%%%%%%%%%%%%%%%%%%%%%%%%%%%%%%%%%      

In this section we establish important results, which implies that $\mathit{H}$  satisfies $(\mathcal{H}')$, when the function $\mathit{H}$ in \eqref{s1} is independent of $x$.   
    
To begin, we  consider the Euler–Lagrange functional $ \tilde{\mathcal{J}}: \mathit{W}\to \mathbb{R}$ associated to \eqref{ss2},  satisfying $(\mathcal{H}')$,  by

\begin{equation*}
  \tilde{\mathcal{J}}(u,v)= \Phi(u,v)  -\lambda\int_{\mathbb{R}^{N}}H(u,v)\,dx- \int_{\mathbb{R}^{N}}\frac{|u|^{\alpha(x)}|v|^{\beta(x)}}{q(x)}
\,dx,
    \end{equation*}
  where the functional $\Phi$ was defined in Lemma \ref{s+}.

Of course, since $ \tilde{\mathcal{J}}$ is obtained as a particular case of  $ \mathcal{J}$, it is well defined, of class $C^{1}(\mathit{W})$  and the weak solutions
of  \eqref{ss2} are exactly its critical points.
 
Furthermore, we consider 
$$\tilde{c}_{\lambda}=\inf_{\gamma\in \Gamma}\max_{\zeta\in[0,1]}\tilde{\mathcal{J}}(\gamma(\zeta)),$$
  where 
  $$\Gamma=\big\{ \gamma \in C([0,1], \mathit{W}): \gamma(0)=(0,0),\, \tilde{\mathcal{J}}(\gamma(1))<0\big\}.$$ and  $\tilde{c}_{\lambda}$ for $\lambda >0$. Moreover, Lemmas \ref{lemaA}–\ref{lemaC} still hold with much simpler proofs for system \eqref{ss2}.   
  
  Following \cite{silva}, we obtain a version of the Lion's Lemma for   fractional  Sobolev spaces  with variable exponents-$\mathit{W}$ satisfying $(\mathfrak{p})$ and $(\mathfrak{q})$. 
  
  In this result, we will see which gives conditions guaranteeing strong in $ L^{q(\cdot)}(\mathbb{R}^{N})$ convergence of a bounded sequence of $\mathit{W}$.  
Such result will see play a very important role in order to obtain a nontrivial solution for Problem \eqref{ss2} satisfying $(\mathcal{H}')$.

    \begin{lemma}\label{lemalions}
Let  $r(x)\in [p^{+},p^{\star}_{s}(x)[$ for all $x\in \mathbb{R}^{N}.$ Assume that $\{(u_k,u_k)\}_{k\in \mathbb{N}}$ be a bounded sequence in $\mathit{W}$ verifying
\begin{equation}\label{lionsa}
    \lim_{k\to \infty}\sup_{y \in \mathbb{R}^{N}}\int_{B_{R}(y)}(|u_k|^{p(x)}+|v_k|^{p(x)})\,dx=0,
\end{equation}
for some $R>0.$ Then $u_{k}\to 0$ and $v_{k}\to 0$ in $L^{r(\cdot)}(\mathbb{R}^{N})$ for $p(x)\leqslant r(x)<p_{s}^{\star}(x)$ for all $x\in \mathbb{R}^{N}$. 
\end{lemma}

    \begin{proof}
Let  $\{(u_k,u_k)\}_{k\in \mathbb{N}} \subset \mathit{W}$ verifying \eqref{lionsa}, by Lemma \ref{mergulho} the sequences $\{u_k\}_{k\in \mathbb{N}} \subset \mathit{W}$ and $\{v_k\}_{k\in \mathbb{N}} \subset \mathit{W}$ are bounded in $L^{r(\cdot)}(\mathbb{R}^{N})$. Once that $p(x)\leqslant r(x)<p_{s}^{\star}(x)$ for all $x\in \mathbb{R}^{N}$ for given $\varepsilon>0,$   there exists $\Gamma>0$ such that
\begin{equation}\label{322}
    \frac{|\sigma|^{r(x)}}{|\sigma|^{p_{s}^{\star}(x)}}\leqslant \frac{\varepsilon}{3 \mathfrak{k}},\quad |\sigma|\geqslant \Gamma,
\end{equation}
where
$$\mathfrak{k}=\sup_{k}\int_{\mathbb{R}^{N}}|u_{k}|^{p_{s}^{\star}(x)}\,dx.$$
Since $p(x)\leqslant r(x)<p_{s}^{\star}(x)$ for all $x\in \mathbb{R}^{N}$,  we have $|\sigma|^{r(x)}/|\sigma|^{p_{s}^{\star}(x)}\to 0$ as $|\sigma|\to 0,$ then, there exists $\delta>0$ such that

\begin{equation}\label{323}
    \frac{|\sigma|^{r(x)}}{|\sigma|^{p_{s}^{\star}(x)}}\leqslant \frac{\varepsilon}{3 \vartheta},\quad |\sigma|<\delta,
\end{equation}
where 

$$\vartheta:=\sup_{k}\int_{\mathbb{R}^{N}}|u_{k}|^{p(x)}\,dx .$$
\noindent Let us write 
\begin{equation}\label{324}
    \int_{\mathbb{R}^{N}}|u_{k}|^{r(x)}\,dx=\int_{\{|u_{k}|\leqslant \delta \} }|u_{k}|^{r(x)}\,dx+\int_{\{\delta<|u_{k}|< \Gamma \} }|u_{k}|^{r(x)}\,dx+\int_{\{|u_{k}|\geqslant \Gamma \} }|u_{k}|^{r(x)}\,dx.
\end{equation}

\noindent By using \eqref{322} we have 

\begin{equation}\label{325}
    \int_{\{|u_{k}|\geqslant \Gamma\}}|u_{k}|^{r(x)}\,dx\leqslant \frac{\varepsilon}{3 \mathfrak{k}}\int_{\mathbb{R}^{N}}|u_{k}|^{p_{s}^{\star}(x)}\,dx\leqslant \frac{\varepsilon}{3}.
\end{equation}

\noindent From \eqref{323} it follows that
\begin{equation}\label{326}
    \int_{\{|u_{k}|\leqslant \delta \} }|u_{k}|^{r(x)}\,dx\leqslant \frac{\varepsilon}{3\vartheta}\int_{\mathbb{R}^{N} }|u_{k}|^{p^{\star}_{s}(x)}\,dx\leqslant \frac{\varepsilon}{3}.
\end{equation}
There are two cases to analyse. 

\noindent\textbf{Case 1.} We suppose that 
\begin{equation*}
    \lim_{k\to \infty}|\{\delta<|u_{k}|<\Gamma \} |=0.
\end{equation*}
So, there exists $k_{0}\in \mathbb{N}$ such that 

\begin{equation}\label{328}
    |\{\delta<|u_{k}|<\Gamma \}|<\frac{\min\{\delta^{p^-}, \delta^{p^+}\}}{\max \{|\Gamma|^{r^{-}}, |\Gamma|^{r^{+}}\}\max \{|\Gamma|^{p^{-}}, |\Gamma|^{p^{+}}\} }\frac{\varepsilon}{3}\,,\quad k\geqslant k_{0}.
\end{equation}
Hence,

\begin{equation}\label{329}
    |\{\delta<|u_{k}|<\Gamma \} |\leqslant \frac{1}{\min\{\delta^{p^-}, \delta^{p^+} \}}\int_{\{\delta<|u_{k}|<\Gamma \}}|u_{k}|^{p(x)}\,dx\leqslant \frac{ \max \{|\Gamma|^{p^{-}}, |\Gamma|^{p^{+}}\}  }{\min\{\delta^{p^-}, \delta^{p^+} \}}|\{  \delta<|u_{k}|<\Gamma  \}|.
\end{equation}
For $k\geqslant k_{0},$ from \eqref{328} and \eqref{329} it follows that
\begin{equation}\label{330}
    \int_{\{\delta<|u_{k}|<\Gamma\}} |u_{k}|^{r(x)}\,dx\leqslant \frac{\max \{|\Gamma|^{r^{-}}, |\Gamma|^{r^{+}}\}}{\min\{\delta^{p^-}, \delta^{p^+} \}} \int_{\{\delta<|u_{k}|<\Gamma\}}|u_{k}|^{p(x)}\,dx< \frac{\varepsilon}{3}.
\end{equation}

\noindent Hence, from \eqref{325}, \eqref{326} and \eqref{330}, we get
\begin{equation}\label{331}
   \int_{\mathbb{R}^{N}}  |u_{k}|^{r(x)}\,dx\leqslant \varepsilon
\end{equation}
holds true for each $\varepsilon>0,$ which finishes the proof for the Case 1.

\noindent\textbf{Case 2.} Up to a subsequence, we assume that
\begin{equation}\label{332}
    \lim_{k\to \infty}|\delta<|u_{k}|<\Gamma|=\ell\in (0,\infty). 
\end{equation}

\noindent Let us prove that this case does not occur. For this purpose, we show the following statement.

\noindent\textbf{Claim 1.} There exist $y_{0}\in \mathbb{R}^{N}$ and $\tau>0$ such that
\begin{equation}\label{333}
0<\tau\leqslant | \{ \delta < | u_k |< \Gamma \} \cap B_{r}(y_0)|
\end{equation}
holds true for a subsequence of $\{u_{k} \}_{k\in \mathbb{N}}$ which is denoted by $\{u_{k} \}_{k\in \mathbb{N}}.$

\noindent \textit{Proof Claim 1.} Indeed, by contradiction, for each $\varepsilon>0,$ $\iota\in \mathbb{N}$ we have
\begin{equation}\label{334}
| \{ \delta < | u_k |< \Gamma \} \cap B_{r}(y)|< \frac{\epsilon}{2^{\iota}}
\end{equation}
for all $y\in\mathbb{R}^{N}$. Note that the last inequality holds for any subsequence of $\{u_{k} \}_{k\in \mathbb{N}}.$ Without loss of generality, we consider just the sequence $\{u_{k} \}_{k\in \mathbb{N}}.$ Now, choosing $\{y_{\iota}\}_{\iota\in \mathbb{N}}\subset\mathbb{R}^{N}$ such that $\displaystyle\bigcup_{\iota=1}^{\infty}B_{r}(y_{\iota})=\mathbb{R}^{N}$ and using \eqref{334}, we have
\begin{equation}\label{335}
\begin{split}
   \displaystyle | \{ \delta < | u_k |< \Gamma \}  |&=\Bigg| \{ \delta < | u_k |< \Gamma \} \bigcap \bigg(\bigcup_{\iota=1}^{\infty}B_{r}(y_{\iota})\bigg) \Bigg| \\
    &\leqslant \sum_{\iota=1}^{\infty}|\{\delta < | u_k |< \Gamma \} \cap (B_{r}(y_{\iota})) |\\
    &\leqslant \sum_{\iota=1}^{\infty} \frac{\varepsilon}{2^{\iota}}=\varepsilon,
    \end{split}
\end{equation}
where $\varepsilon>0$ is arbitrary. Up to a subsequence by  \eqref{335}, we have that 

\begin{equation*}
    0<\ell=\lim_{k\to \infty}|\delta<|u_{k}|<\Gamma | \leqslant \varepsilon,
\end{equation*}
which makes no sense for $\varepsilon\in (0,\ell),$ which proves the Claims 2.  

 \noindent Now, using the Claim 1 and \eqref{lionsa}, we get
\begin{equation*}
    \begin{split}
        0<\tau\leqslant |\{ \delta<|u_{k}|<\Gamma   \} \cap B_{r}(y_{0})|& \leqslant \frac{1}{\min\{\delta^{p^-}, \delta^{p^+} \}}\int_{B_{r}(y_{0})}|u_{k}|^{p(x)}\,dx
        \\ &\leqslant \frac{1}{\min\{\delta^{p^-}, \delta^{p^+} \}}\sup_{y\in \mathbb{R}^{N} }\int_{B_{r}(y)}|u_{k}|^{p(x)}\,dx \longrightarrow 0
    \end{split}
\end{equation*}
as $k\to \infty.$ This contradiction proves that second case is impossible, i.e., $\ell=0.$ The result then follows from the first case. The same holds for $v_k$  which finishes the proof.
\end{proof}

      \begin{lemma}\label{weak}
Let  $\{(u_k,u_k)\}_{k\in \mathbb{N}}$ be a Palais-Smale sequence in $\mathit{W}$ of $\tilde{\mathcal{J}}$ at the level ${\tilde{c}_{\lambda}}$ for $\lambda>0$. Then 
\begin{itemize}
\item[$(a)$] Up a subsequence $(u_k,v_k)\rightharpoonup (u_{\lambda},v_{\lambda})$ in  $\mathit{W}$ as $k \to \infty$;
\item[$(b)$] The weak limit  $(u_{\lambda},v_{\lambda})$ is a critical point of $\tilde{\mathcal{J}}$.
\end{itemize}
\end{lemma} 
    \begin{proof}
    $(a)$ The proof of this item follows a similar argument of \eqref{f01} in   Proposition \ref{lemaD}.
    
 \noindent $(b)$ Proving now that the weak limit $(u_{\lambda},v_{\lambda})$ is a weak solution of \eqref{ss2} namely $\langle \tilde{\mathcal{J}}'(u_k,v_k), (\varphi, \psi)\rangle =0 $ for all $(\varphi, \psi) \in \mathit{W}$.
 
 \noindent Indeed, given any couple  $(\varphi, \psi) \in C_{c}^{\infty}(\mathbb{R}^{N})\times C_{c}^{\infty}(\mathbb{R}^{N})$  using  \eqref{pale}, we have
  \begin{equation}\label{223}
   \begin{split}
   o(1)= & \langle \tilde{\mathcal{J}}'(u_k,v_k), (\varphi, \psi)\rangle \\  =& \int_{\mathbb{R}^{N}}\int_{\mathbb{R}^{N}} \mathcal{A}(u_k(x)-u_{k}(y))(\varphi(x)-\varphi(y))K(x,y)\,dy\,dx+   \int_{\mathbb{R}^{N}}|u_k|^{\overline{p}-2}u_k\varphi(x)\,dx \\ &+ \int_{\mathbb{R}^{N}}\int_{\mathbb{R}^{N}} \mathcal{A}(v_k(x)-v_{k}(y))(\psi(x)-\psi(y))K(x,y)\,dy\,dx +   \int_{\mathbb{R}^{N}}|v_k|^{\overline{p}-2}v_k\psi(x)\,dx \\&
   - \lambda \int_{\mathbb{\mathbb{R}^{N}}}\big[ \mathrm{H}_{u}( u_k, v_k)\varphi(x) + \mathrm{H}_{v}( u_k, v_k)\psi(x)\big]\,dx \\& - \int_{\mathbb{\mathbb{R}^{N}}}\frac{\alpha(x)|u_k|^{\alpha(x)-2}u_k|v_k|^{\beta(x)}\varphi(x)}{q(x)}\,dx - \int_{\mathbb{\mathbb{R}^{N}}}\frac{\beta(x)|u_k|^{\alpha(x)}|v_k|^{\beta(x)-2}v_k\psi(x)}{q(x)}\,dx.
   \end{split}
\end{equation} 

\noindent Now, from part $(a)$,  up to a subsequence we can assume that  is valid \eqref{f1}, \eqref{f2}, \eqref{f02}, \eqref{f3} and  by \eqref{f2}, we obtain

\begin{equation}\label{acres}
|u_k|<\mathfrak{g}_{R} \,\,\,\,\,  |v_k|<\mathfrak{g}_{R} \,\,\,\, \mbox{ a.e. in } \mathbb{R}^{N}
\end{equation}
for some $\mathfrak{g}_{R}  \in L^{q(\cdot)}(\mathbb{R}^{N})$ and $\overline{p}<q(x)\leqslant \overline{p}^{\star}_{s}$ for all $x \in \mathbb{R}^{N}$.

\noindent Consider  now the sequences
\begin{equation*}
\mathrm{U}_k(x,y):=  \mathcal{A}(u_k(x)-u_{k}(y))K(x,y)|x-y|^{\frac{N}{p(x,y)}+s}
\end{equation*}
and
\begin{equation*}
\mathrm{V}_k(x,y):=  \mathcal{A}(v_k(x)-v_{k}(y))K(x,y)|x-y|^{\frac{N}{p(x,y)}+s}
\end{equation*}
and put 
\begin{equation*}
\mathrm{U}(x,y) :=  \mathcal{A}(u(x)-u(y))K(x,y)|x-y|^{\frac{N}{p(x,y)}+s}  
\end{equation*}

\begin{equation*}
 \mathrm{V}(x,y):=  \mathcal{A}(v(x)-v(y))K(x,y)|x-y|^{\frac{N}{p(x,y)}+s}
\end{equation*}

\noindent \textbf{Claim.} $\{(\mathrm{U}_k, \mathrm{V}_k\}_{k \in \mathbb{N}}$ is  bounded in $L^{p'(\cdot, \cdot)}(\mathbb{R}^{2N};\mathbb{R}^{2})$ for $p'(\cdot, \cdot)=\frac{p(\cdot, \cdot)}{p(\cdot, \cdot)-1}$.

\noindent \textit{Proof of claim.}Indeed, 
\begin{equation*}
\begin{split}
&\int_{\mathbb{R}^{N}}\int_{\mathbb{R}^{N}} |\mathcal{A}(u_k(x)-u_{k}(y))|^{p'(x,y)}|K(x,y)|^{p'(x,y)}\bigg(|x-y|^{\frac{N}{p(x,y)}+s}\bigg)^{\frac{p(x,y)}{p(x,y)-1}}\,dy\,dx   \\ \leqslant & C_{\mathcal{A}}b_1 \int_{\mathbb{R}^{N}}\int_{\mathbb{R}^{N}} |u_k(x)-u_{k}(y)|^{(p(x,y)-1)\cdot \frac{p(x,y)}{p(x,y)-1}} \bigg( \frac{1}{|x-y|^{N+sp(x,y)}} \bigg)^{\frac{p(x,y)}{p(x,y)-1}}|x-y|^{\frac{N+sp(x,y)}{p(x,y)-1}}\,dy\,dx \\  = & C_{\mathcal{A}}b_1 \int_{\mathbb{R}^{N}}\int_{\mathbb{R}^{N}}\frac{ |u_k(x)-u_{k}(y)|^{(p(x,y)}}{|x-y|^{N+sp(x,y)}} \,dy\,dx  < \infty.
\end{split}
\end{equation*}

\noindent Then the sequence   $\{\mathrm{U}_k\}_{k \in \mathbb{N}}$ is  bounded in $L^{p'(\cdot, \cdot)}(\mathbb{R}^{2N})$.  Analogously the sequence   $\{\mathrm{V}_k\}_{k \in \mathbb{N}}$ is  bounded in $L^{p'(\cdot, \cdot)}(\mathbb{R}^{2N})$. Therefore the sequence  $\{(\mathrm{U}_k, \mathrm{V}_k)\}_{k \in \mathbb{N}}$ is  bounded in $L^{p'(\cdot, \cdot)}(\mathbb{R}^{2N};\mathbb{R}^{2})$ and so there exist $(\mathcal{U}, \mathcal{V}) \in L^{p'(\cdot, \cdot)}(\mathbb{R}^{2N}) $ such that  $\mathrm{U}_k \rightharpoonup \mathcal{U} $  and  $\mathrm{V}_k \rightharpoonup \mathcal{V} $ 
in $L^{p'(\cdot, \cdot)}(\mathbb{R}^{2N}) $. But  $(\mathrm{U}_k,  \mathrm{V}_k)\longrightarrow (\mathrm{U}, \mathrm{V}) $  a.e in $\mathbb{R}^{2N}$ and so $(\mathrm{U}, \mathrm{V})=(\mathcal{U}, \mathcal{V})$. 
\noindent Therefore,
\begin{equation*}
\begin{split}
&\int_{\mathbb{R}^{N}}\int_{\mathbb{R}^{N}} |\mathcal{A}(u_k(x)-u_{k}(y))|K(x,y)|x-y|^{\frac{N}{p(x,y)}+s}\tilde{\varphi}(x,y)\,dy\,dx   \\ &\longrightarrow \int_{\mathbb{R}^{N}}\int_{\mathbb{R}^{N}} |\mathcal{A}(u(x)-u(y))|K(x,y)|x-y|^{\frac{N}{p(x,y)}+s}\tilde{\varphi}(x,y)\,dy\,dx 
\end{split}
\end{equation*}
and
\begin{equation*}
\begin{split}
&\int_{\mathbb{R}^{N}}\int_{\mathbb{R}^{N}} |\mathcal{A}(v_k(x)-v_{k}(y))|K(x,y)|x-y|^{\frac{N}{p(x,y)}+s}\tilde{\psi}(x,y)\,dy\,dx   \\ &\longrightarrow \int_{\mathbb{R}^{N}}\int_{\mathbb{R}^{N}} |\mathcal{A}(v(x)-v(y))|K(x,y)|x-y|^{\frac{N}{p(x,y)}+s}\tilde{\psi}(x,y)\,dy\,dx 
\end{split}
\end{equation*}
for all $\tilde{\varphi} \in L^{p(\cdot, \cdot)}(\mathbb{R}^{2N}) $ and $\tilde{\psi} \in L^{p(\cdot, \cdot)}(\mathbb{R}^{2N}).$
\noindent
This holds in particular if we  consider for any $\varphi, \psi \in C^{\infty}_{c}(\mathbb{R}^{N})$ 
\begin{equation*}
\begin{split}
\tilde{\varphi}(x,y)=\frac{\varphi(x)-\varphi(y)}{|x-y|^{\frac{N}{p(x,y)}+s}} \,\,\,\,\,\, \tilde{\psi}(x,y)=\frac{\psi(x)-\psi(y)}{|x-y|^{\frac{N}{p(x,y)}+s}}
\end{split}
\end{equation*}
since that  $\tilde{\varphi}, \tilde{\psi} \in   L^{p(\cdot, \cdot)}(\mathbb{R}^{2N}).$
Then 
\begin{equation*}
\begin{split}
&\int_{\mathbb{R}^{N}}\int_{\mathbb{R}^{N}} |\mathcal{A}(u_k(x)-u_{k}(y))K(x,y)|x-y|^{\frac{N}{p(x,y)}+s}\frac{\varphi(x)-\varphi(y)}{|x-y|^{\frac{N}{p(x,y)}+s}}\,dy\,dx   \\ &\longrightarrow \int_{\mathbb{R}^{N}}\int_{\mathbb{R}^{N}} |\mathcal{A}(u(x)-u(y))K(x,y)|x-y|^{\frac{N}{p(x,y)}+s}\frac{\varphi(x)-\varphi(y)}{|x-y|^{\frac{N}{p(x,y)}+s}}\,dy\,dx 
\end{split}
\end{equation*}
and 
\begin{equation*}
\begin{split}
&\int_{\mathbb{R}^{N}}\int_{\mathbb{R}^{N}} |\mathcal{A}(v_k(x)-v_{k}(y))K(x,y)|x-y|^{\frac{N}{p(x,y)}+s}\frac{\psi(x)-\psi(y)}{|x-y|^{\frac{N}{p(x,y)}+s}}\,dy\,dx  \\ &\longrightarrow \int_{\mathbb{R}^{N}}\int_{\mathbb{R}^{N}} |\mathcal{A}(v(x)-v(y))K(x,y)|x-y|^{\frac{N}{p(x,y)}+s}\frac{\psi(x)-\psi(y)}{|x-y|^{\frac{N}{p(x,y)}+s}}\,dy\,dx, 
\end{split}
\end{equation*}
this is, 
\begin{equation}\label{ww3}
\begin{split}
&\int_{\mathbb{R}^{N}}\int_{\mathbb{R}^{N}} |\mathcal{A}(u_k(x)-u_{k}(y))(\varphi(x)-\varphi(y))K(x,y)\,dy\,dx   \\ & \longrightarrow \int_{\mathbb{R}^{N}}\int_{\mathbb{R}^{N}} |\mathcal{A}(u(x)-u(y))(\varphi(x)-\varphi(y))K(x,y)\,dy\,dx 
\end{split}
\end{equation}
and 
\begin{equation}\label{ww2}
\begin{split}
&\int_{\mathbb{R}^{N}}\int_{\mathbb{R}^{N}} |\mathcal{A}(v_k(x)-v_{k}(y))(\psi(x)-\psi(y))K(x,y)\,dy\,dx   \\ & \longrightarrow \int_{\mathbb{R}^{N}}\int_{\mathbb{R}^{N}} |\mathcal{A}(v(x)-v(y))(\psi(x)-\psi(y))K(x,y)\,dy\,dx.
\end{split}
\end{equation}
On the hand, since $\mathrm{H}$ has standard subcritical growth, we have by  $(\mathcal{H}')$
\begin{equation}\label{dominada}
\begin{split}
|\mathrm{H}_{u}(u_k,v_k)\varphi|+|\mathrm{H}_{v}(u_k,v_k)\psi| \leqslant a|u_{k}|^{\theta(x)-1}(|\varphi|+|\psi|)+ b|v_{k}|^{\sigma(x)-1}(|\varphi|+|\psi|)\leqslant \mathsf{B}\in L^{1}(B_R).
\end{split}
\end{equation}
Thus by  $(\mathcal{H}')$, \eqref{f1}, \eqref{dominada},  and the Dominated Convergence Theorem yields, as $k\to \infty$,
\begin{equation}\label{ww4}
\begin{split}
\int_{\mathbb{R}^{N}}[\mathrm{H}_{u}(u_k,v_k)\varphi+\mathrm{H}_{v}(u_k,v_k)\psi]\,dx \longrightarrow  \int_{\mathbb{R}^{N}}[\mathrm{H}_{u}(u,v)\varphi+\mathrm{H}_{v}(u,v)\psi]\,dx.
\end{split}
\end{equation}
 In conclusion, passing to the limit in \eqref{223}, using \eqref{f1}, \eqref{f2}, \eqref{f02}, \eqref{f3}, \eqref{acres}, \eqref{ww3},    \eqref{ww2}, and \eqref{ww4}, we obtain
 \begin{equation*}\label{223}
   \begin{split}
  \langle \tilde{\mathcal{J}}'(u_k,v_k), (\varphi, \psi)\rangle =0 \,\,\,\,\, \mbox{ for all }  (\varphi, \psi)\in C^{\infty}_{c}(\mathbb{R}^{N})\times C^{\infty}_{c}(\mathbb{R}^{N}). 
   \end{split}
\end{equation*} 
 Consequently, from the density of $C^{\infty}_{c}(\mathbb{R}^{N})$ in $W^{s,p(\cdot, \cdot)}(\mathbb{R}^{N})$, (see \cite[Theorem 3.2]{baalal}), we obtain that $(u_{\lambda}, v_{\lambda})$ is critical point of $\tilde{\mathcal{J}}$ for any $\lambda>0.$
      \end{proof}

    \begin{proposition}\label{bounded}
    There exists $\lambda>0$ and $\{(u_k,u_k)\}_{k\in \mathbb{N}}\subset\mathit{W}$ be a $(PS)_{\tilde{c}_{\lambda}}$ of $\tilde{\mathcal{J}}$ at level $\tilde{c}_{\lambda}$ such that $(u_k,u_k) \rightharpoonup (0,0)$ in $\mathit{W}$ as $k \to \infty$. Then either
    \begin{itemize}
    \item[ $(a)$]  $(u_k,u_k) \longrightarrow (0,0)$ in $\mathit{W}$ or
    \item[ $(b)$] there exists $R>0$ and a sequence $\{y_k\}_{k\in \mathbb{N}}\subset \mathbb{R}^{N}$ such  that 
    $$\limsup_{k \to \infty}\int _{B_{R}(y_k)}(|u_k|^{\overline{p}}+|v_k|^{\overline{p}})\,dx>0.$$
    \end{itemize}
    Moreover, $\{y_k\}_{k\in \mathbb{N}}$ is not bounded in $\mathbb{R}^{N}.$
    \end{proposition} 
  \begin{proof}
  Assume that $(b)$ does not occur. Then, for all $R>0$ 
   $$\limsup_{k \to \infty}\int _{B_{R}(y_k)}(|u_k|^{\overline{p}}+|v_k|^{\overline{p}})\,dx=0.$$
  Therefore, by Lemma \ref{lemalions} and \eqref{remark} up to a subsequence 
  \begin{equation}\label{pa0}
   u_k\to 0 \,\,\,\,\, v_k\to 0 \,\,\,\, \mbox{ in } L^{q(\cdot)}(\mathbb{R}^{N})
  \end{equation}
  as $k \to \infty$ for all $\overline{p}\leqslant q(x)\leqslant \overline{p}^{\star}_{s}$ for all $x \in \mathbb{R}^{N}$.
  
  \noindent Consequently, by $(\mathcal{H}')$,  H\"older inequality (see Proposition \ref{hold3}) and \eqref{pa0}, we have 
  \begin{equation}\label{pa1}
\begin{split}
0 & \leqslant \int_{\mathbb{R}^{N}}[\mathrm{H}_{u}(u_k,v_k)u_k+ \mathrm{H}_{v}(u_k,v_k)v_k]\,dx \\&\leqslant \int_{\mathbb{R}^{N}} a|u_{k}|^{\theta(x)}+ b|v_{k}|^{\sigma(x)-1}|u_k|+a |u_{k}|^{\theta(x)-1}|v_{k}|+b |v_k|^{\sigma(x)}\,dx\longrightarrow 0
\end{split}
\end{equation}
since $\overline{p}<\theta^{-}<\theta^{+}<\sigma^{-}\leqslant \sigma^{+}<q^{-} $     as $k\to \infty$.
Now since    $\{(u_k,u_k)\}_{k\in \mathbb{N}}\subset\mathit{W}$ be a $(PS)_{\tilde{c}_{\lambda}}$ of $\tilde{\mathcal{J}}$ at level $\tilde{c}_{\lambda}$, by similar argument  Proposition \ref{lemaD}, the sequence $\{(u_k,u_k)\}_{k\in \mathbb{N}}$ is bounded in $ \mathit{W}$. Then, up to a subsequence, by   H\"older inequality (see Proposition \ref{hold3}), Proposition \ref{mista} and \eqref{pa0}, as $k\to \infty$, we obtain 
   \begin{equation}\label{pa2}
\begin{split}
 \int_{\mathbb{R}^{N}}\frac{|u_k|^{\alpha(x)}|v_k|^{\beta(x)}}{q(x)}\,dx\longrightarrow 0.
\end{split}
\end{equation}
 
\noindent    On the other hand,  by $(a_2)$-$(a_3)$, and  Proposition \ref{norma} we obtain 
 \begin{equation}\label{pa3}
\begin{split}
 \Phi(u_k,v_k)\geqslant  d_6\min \{ \|(u_k, v_k)\|^{\overline{p}}, \|(u_k, v_k)\|^{p^{-}}  \}>0
\end{split}
\end{equation}
 for $d_6= \min \bigg\{ \frac{1}{\overline{p}}C_{\mathcal{A}b_1}, \frac{1}{\overline{p}}\tilde{C}_{\mathcal{A}}\tilde{b}_1,\overline{p}\bigg\}$.
 Therefore, by condition  $(PS)_{\tilde{c}_{\lambda}}$ of $\tilde{\mathcal{J}}$ at level $\tilde{c}_{\lambda}$, \eqref{pa1},   \eqref{pa2} and \eqref{pa3} implies that 
   \begin{equation}\label{pa4}
\begin{split}
 \int_{\mathbb{R}^{N}}\frac{u_k(x)|^{\alpha(x)}|v_k|^{\beta(x)}}{q(x)}\,dx +o(1) = \Phi(u_k,v_k)\geqslant  d_5\min \{ \|(u_k, v_k)\|^{\overline{p}}, \|(u_k, v_k)\|^{p^{-}}  \}>0.
\end{split}
\end{equation}
 Then, as $k\to \infty$, by \eqref{pa4} and \eqref{pa2}, we obtain  $\|(u_k,v_k)\|\to 0$ and then $(a)$ holds.
 
\noindent  Assume $(b)$ is verified and suppose by contradiction that $\{y_k\}_{k \in \mathbb{N}}$ is bounded in $\mathbb{R}^{N}$.  Consequently there exist $\mathfrak{R}>0$ large that $B_{R}(y_k)\subset B_{\mathfrak{R}}$ for all $k$. Now, Lemma \ref{lemalions} implies that 
$(u_k,v_k)\to (0,0)$  in $ L^{r(\cdot)}_{loc}(\mathbb{R}^{N})\times  L^{r(\cdot)}_{loc}(\mathbb{R}^{N})$ for all $ r(x)<p^{\star}_{s}(x)$, $x \in \mathbb{R}^{N}$.
Therefore
\begin{equation*}
\begin{split}
 0=\lim_{k \to \infty} \int_{B_{M}}(|u_k|^{\overline{p}}+|v_k|^{\overline{p}})\,dx \geqslant  \limsup_{k \to \infty} \int_{B_{R}(y_k)}(|u_k|^{\overline{p}}+|v_k|^{\overline{p}})\,dx>0
\end{split}
\end{equation*}
which gives the required contradiction. In conclusion $\{y_k \}_{k \in \mathbb{N}}$ is not bounded in $\mathbb{R}^{N}$.

\end{proof}    
    
 \subsubsection*{Proof of Theorem \ref{aplica2}}
\begin{proof}
First,   thanks to Lemmas \ref{lemaA}, \ref{lemaB}       and  \ref{weak} for any $\lambda>0$ the functional $\tilde{\mathcal{J}}$ has the geometry of the Mountain Pass Theorem, so that admits a Palais-Smale sequence $\{(u_k, v_k)\}_{k \in \mathbb{N}}$ at level $\tilde{c}_{\lambda}$, which up to subsequence, weakly converges to the limit   $(u_{\lambda}, v_{\lambda}).$ The weak limit $(u_{\lambda}, v_{\lambda})$ is a critical point of $\tilde{\mathcal{J}}$ for all $\lambda > 0$ namely weak solution \eqref{ss2}. It remains to show that the constructed solution $(u_{\lambda}, v_{\lambda})$ is nontrivial. Assume by contradiction that $(u_{\lambda}, v_{\lambda})=(0,0)$. Clearly $\{(u_k, v_k)\}_{k \in \mathbb{N}}$ cannot converge strongly to $(0,0)$ in $\mathit{W}$, since otherwise $\tilde{\mathcal{J}}'(u_{\lambda}, v_{\lambda})=0$ and  $0= \tilde{\mathcal{J}}(u_{\lambda}, v_{\lambda})=\tilde{c}_{\lambda}>0$  by Lemmas \ref{lemaA} and \ref{lemaB}.

\noindent Therefore, by Proposition \ref{bounded} there exist $R>0$ and a sequence $\{y_k\}_{k \in \mathbb{N}}\subset \mathbb{R}^{N}$ such that
\begin{equation*}
    \lim_{k\to \infty}\sup_{y \in \mathbb{R}^{N}}\int_{B_{R}(y_k)}(|u_k|^{\overline{p}}+|v_k|^{\overline{p}})\,dx>0.
\end{equation*}
Now define a new sequence   $\{(\tilde{u}_k, \tilde{v}_k)\}_{k \in \mathbb{N}}$ by
\begin{equation*}
\tilde{u}_{k}= u_{k}(\cdot + y_k) \,\,\,\,\,  \tilde{v}_{k}= v_{k}(\cdot + y_k).
\end{equation*}
Then
\begin{equation}\label{fina}
\tilde{\mathcal{J}}(\tilde{u}_{k}, \tilde{v}_{k}) = \tilde{\mathcal{J}}(u_{k}, v_{k}) \longrightarrow \tilde{c}_{\lambda} \,\,\,\,\,\,  \,\,\,\,\,  \tilde{\mathcal{J}}'(\tilde{u}_{k}, \tilde{v}_{k})\longrightarrow 0.
\end{equation}
Indeed, for all $(\mathfrak{u}, \mathfrak{v})\in \mathit{W}$  with $\|(\mathfrak{u}, \mathfrak{v})\|=1$ putting, $\mathfrak{u}_{k}(z)=\mathfrak{u}(z-y_k) \,\,$ and $ \mathfrak{v}_{k}(z)=\mathfrak{v}(z-y_k)$, in $\mathbb{R}^{N}$, we have
\begin{equation*}
\begin{split} 
|\langle \tilde{\mathcal{J}}'(u_{k}, v_{k}), (\mathfrak{u}_{k}, \mathfrak{v}_{k})\rangle| \leqslant \|  \tilde{\mathcal{J}}'(u_{k}, v_{k})\|_{\mathit{W}'}\| (\mathfrak{u}_{k}, \mathfrak{v}_{k})\|=\|  \tilde{\mathcal{J}}'(u_{k}, v_{k})\|_{\mathit{W}'}
\end{split}
\end{equation*}
since  $1=\| (\mathfrak{u}, \mathfrak{v})\|= \| (\mathfrak{u}_{k}, \mathfrak{v}_{k})\|$.

\noindent On the other hand, with a simple change of variable, it is easy to see that
\begin{equation*}
\begin{split} 
\langle\tilde{\mathcal{J}}'(\tilde{u}_{k}, \tilde{v}_{k}), (\mathfrak{u}, \mathfrak{v})\rangle  =\langle \tilde{\mathcal{J}}'(u_{k}, v_{k}), (\mathfrak{u}_{k}, \mathfrak{v}_{k})\rangle .
\end{split}
\end{equation*}
Therefore as $k \to \infty$
\begin{equation*}
\begin{split} 
\| \tilde{\mathcal{J}}'(u_{k}, v_{k})\|_{\mathit{W}'}=\sup_{\substack{ (\mathfrak{u}, \mathfrak{v}) \in \mathit{W}, \\\|(\mathfrak{u}, \mathfrak{v})\|=1}}|\langle \tilde{\mathcal{J}}'(u_{k}, v_{k}), (\mathfrak{u}, \mathfrak{v})  \rangle | \leqslant \|  \tilde{\mathcal{J}}'(u_{k}, v_{k})\|_{\mathit{W}'}= o(1).
\end{split}
\end{equation*}
Consequently, \eqref{fina}  holds and so $\{(\tilde{u}_k, \tilde{v}_{k})\}_{k \in \mathbb{N}}$ is Palais-Smale sequence at level $\tilde{c}_{\lambda}$. Thus, by Lemma \ref{weak},  applied to the new sequence  it follows  that $\{(\tilde{u}_k, \tilde{v}_k)\}_{k \in \mathbb{N}}$ weakly converge to some $(\tilde{u}_{\lambda}, \tilde{v}_{\lambda})$ in $\mathit{W}$ and  $(\tilde{u}_{\lambda}, \tilde{v}_{\lambda})$ is a critical point of $\tilde{\mathcal{J}}$.  Moreover, by \eqref{fim1} we obtain
\begin{equation*}
0< \lim_{k\to \infty}\sup_{y \in \mathbb{R}^{N}}\int_{B_{R}(y_k)}(|u_k|^{\overline{p}}+|v_k|^{\overline{p}})\,dx = \lim_{k\to \infty}\int_{B_{R}}(|\tilde{u}_k|^{\overline{p}}+|\tilde{v}_{k}|^{\overline{p}})\,d\tilde{x}=\int_{B_{R}}(|\tilde{u}_{\lambda}|^{\overline{p}}+|\tilde{v}_{\lambda}|^{\overline{p}})\,d\tilde{x}.
\end{equation*}
Hence   $(\tilde{u}_{\lambda}, \tilde{v}_{\lambda})  \neq (0,0)$ is the desired nontrivial.
Finally, we can prove as in Theorem \ref{aplica} that both components of $(\tilde{u}_{\lambda}, \tilde{v}_{\lambda})$  are nontrivial and that \eqref{assim} holds. This completes the proof.
\end{proof}

\noindent\textbf{Acknowledgments.}  The first author has received a research grant from FAPERJ (Pós-doc Nota 10), through the grant 260003/014956/2021. The second author has received a research grant from the Coordenação de Aperfeiçoamento de Pessoal de Nível
Superior (CAPES)-Finance Code 001. The third author has received research grants from CNPq through the grant  308064/2019-4, and also by FAPERJ  (Cientista do Nosso Estado) through the grant E-26/201.139/2021.

\medskip
Data sharing not applicable to this article as no data sets were generated or analysed during the current study. 
    
%%%%%%%%%%%%%%%%%%%%%%%%%%%%%%

\bigskip

\noindent Instituto de Matem\'atica,\\
 Universidade Federal do Rio de Janeiro\\
C.P. 68530, Cidade Universitária 2194--970, Rio de Janeiro, Brazil 
\\ e-mails: nn.lauren@gmail.com and wladimir@im.ufrj.br

\bigskip

\noindent Departamento  de Matem\'atica\\
Universidade de Bras\'ilia\\
70297-400 --Distrito Federal -- Brazil \\
e-mails: elardjh2@gmail.com

\end{document}